\documentclass[a4paper,11pt]{amsart}

\newcommand{\red}[1]{\textcolor{red}{#1}}

\usepackage[
  margin=30mm,
  marginparwidth=25mm,     % + <- Width of your marginpar
  marginparsep=2mm,       % + <- Gap between text block and marginpar
  bottom=25mm,
  ]{geometry}

\newcommand{\AJC}[1]{\marginpar{\red{\scriptsize \textbf{AJC:} #1}}}

\usepackage[latin1]{inputenc}
\usepackage[T1]{fontenc}
\usepackage[spanish,english]{babel}
\usepackage{amsmath,amssymb,amsthm}
\usepackage{latexsym}
\usepackage{delarray}
\usepackage{bbm}
\usepackage{hyperref}
\usepackage[pdftex,usenames,dvipsnames]{color}
\usepackage{bbding,trfsigns}
\usepackage{wasysym}
\usepackage{datetime}
\usepackage{mathrsfs}

\usepackage{tikz}

\newtheorem{ThA}{Theorem}

\newtheorem{Th}{Theorem}[section]
\newtheorem{Prop}[Th]{Proposition}
\newtheorem{Lem}[Th]{Lemma}
\newtheorem{Cor}[Th]{Corollary}
\newtheorem{Rem}[Th]{Remark}

\numberwithin{equation}{section}

\newcommand{\N}{\mathbb{N}}

\newcommand{\Z}{\mathbb{Z}}

\newcommand{\MM}{\mathcal{M}}

\newcommand{\PP}{\mathcal{P}}

\DeclareMathOperator{\supp}{supp}

\allowdisplaybreaks % Allows to break a big formula into different pages. Only works with align environment.

\def\mean#1{\mathchoice%
          {\mathop{\kern 0.2em\vrule width 0.6em height 0.69678ex depth -0.58065ex
                  \kern -0.8em \intop}\nolimits_{\kern -0.4em#1}}%
          {\mathop{\kern 0.1em\vrule width 0.5em height 0.69678ex depth -0.60387ex
                  \kern -0.6em \intop}\nolimits_{#1}}%
          {\mathop{\kern 0.1em\vrule width 0.5em height 0.69678ex
              depth -0.60387ex
                  \kern -0.6em \intop}\nolimits_{#1}}%
          {\mathop{\kern 0.1em\vrule width 0.5em height 0.69678ex depth -0.60387ex
                  \kern -0.6em \intop}\nolimits_{#1}}}

\title[Hardy spaces with variable exponent on graphs]
      {Variable exponent Hardy spaces associated with discrete Laplacians on graphs}

\author[V. Almeida]{V. Almeida}
\author[J. J. Betancor]{J. J. Betancor}
\author[A. J. Castro]{A. J. Castro}
\author[L. Rodr\'iguez-Mesa]{L. Rodr\'iguez-Mesa}

\address{\newline
V\'ictor Almeida,
Jorge J. Betancor,
Lourdes Rodr\'iguez-Mesa \newline
Departamento de An\'alisis Matem\'atico,
Universidad de La Laguna, \newline
Campus de Anchieta, Avda. Astrof\'isico S\'anchez, s/n, \newline
        38721 La Laguna (Sta. Cruz de Tenerife), Spain}
\email{valmeida@ull.es, jbetanco@ull.es, lrguez@ull.es}

\address{\newline
       Alejandro J. Castro \newline
       Department of  Mathematics, Nazarbayev University, \newline
        010000 Astana, Kazakhstan}
\email{alejandro.castilla@nu.edu.kz}

\thanks{
The authors are partially supported by Spanish Government grant
MTM2016-79436-P. The third author is also supported by Nazarbayev
University Social Policy Grant.}

\keywords{Graphs, discrete Laplacian, Hardy spaces, variable exponent, square functions, spectral multipliers}
\subjclass[2010]{42B30, 42B35, 60J10}

\begin{document}

\footnotetext{Last modification: \today.}

\begin{abstract}
In this paper we develop the theory of variable exponent Hardy spaces associated with discrete Laplacians on infinite graphs. Our Hardy spaces are defined by square integrals, atomic and molecular decompositions. Also we study boundedness properties of Littlewood-Paley functions, Riesz transforms, and spectral multipliers for discrete Laplacians on variable exponent Hardy spaces.
\end{abstract}

\maketitle

% {\large{\tableofcontents}}

%%%%%%%%%%%%%%%%%%%%%%%%%%%%%%%%%%%%%%%%%%%%%%%%%%%%%%%%%%%%%%%%%%%
\section{Introduction}\label{sec:intro}
%%%%%%%%%%%%%%%%%%%%%%%%%%%%%%%%%%%%%%%%%%%%%%%%%%%%%%%%%%%%%%%%%%%

Variable exponent $L^{p(\cdot)}$-spaces were consider for the first
time by Orlicz \cite{Orl}. Later, Nakano \cite[p. 284]{Nak} pointed
out that those spaces appear as special cases of the ones studied in
\cite{Nak}. Inspired by Nakano's results, modular spaces were
investigated by several authors. An important step in the study of
variable exponent spaces was the paper of Kov\'a\u cik and R\'akosn\'{\i}k
\cite{KR} where the main properties of the Lebesgue and Sobolev
spaces with variable exponent were established. In this century the
variable exponent spaces have been studied systematically and
intensively. Interesting connections between variable exponent
spaces and other areas have been developed (see, for instance,
\cite{AbMesSou}, \cite{AcMi}, \cite{CheLeRa}, \cite{DieThesis},
\cite{RaRu} and \cite{Ruz}). In the recent monographs \cite{CrF} and
\cite{DHHR} the theory of variable exponent Lebesgue and Sobolev
spaces is developed in an exhaustive way.

In this paper we are interested in variable exponent Hardy spaces.
The classical theory of Hardy spaces in ${\mathbb{R}}^n$ has been
extended to a variable exponent setting by Nakai and Sawano
\cite{NS}, Sawano \cite{Sa}, and Cruz-Uribe and Wang \cite{CrW}.
Variable exponent Hardy spaces $H^{p(\cdot)}({\mathbb{R}}^n)$ were
characterized by using Riesz transforms by Yang, Zhuo and Nakai
\cite{YZN}, and intrinsic square functions by Zhuo, Yang and Liang
\cite{ZYL}. Hardy spaces of variable exponent
$H^{p(\cdot)}_L({\mathbb{R}}^n)$ associated with certain classes of
operators $L$ have been investigated by Zhuo and Yang (\cite{YZ2}
and \cite{ZY}) and Yang, Zhang and Zhuo \cite{YZZ}.

Our objective is to study variable exponent Hardy spaces on graphs.
Our graphs will be spaces of homogeneous type in the sense of
Coifman and Weiss \cite{CoW1}. Hardy spaces on spaces of
homogeneous type were firstly studied by Coifman and Weiss
\cite{CoW2} and Mac\'{\i}as and Segovia (\cite{MS} and
\cite{MS0}).

Let $X$ be a set equipped with a non negative quasidistance $d$. Suppose that $\mu$ is a $\sigma$-finite measure defined on a $\sigma$-algebra on $X$ containing the open balls associated with $d$. We assume that $\mu(B)>0$, for every $d$-ball $B$. We say that the triple $(X,\mu,d)$ is a space of homogeneous type when $\mu$ satisfies the following doubling property: there exists $C>0$ such that, for every $x\in X$ and $r>0$, $\mu(B(x,2r))\leq C\mu(B(x,r))$. Note that this implies that there exist $C,D>0$ such that $\mu(B(x,r))\leq C(r/s)^D\mu(B(x,s))$, for every $x\in X$ and $0<s<r$.

Assume that $p:\;X\;\rightarrow\;(0,\infty)$ is a $\mu$-measurable function. If $f$ is a complex $\mu$-measurable function on $X$ we define the modular $\rho _{p(\cdot )}(f)$ of $f$ by
    $$\rho _{p(\cdot )}(f)=\int_X{|f(x)|^{p(x)}d\mu(x)}.$$
The variable exponent Lebesgue space $L^{p(\cdot)}(X)$ is the collection of all complex $\mu$-measurable functions $f$ such that,
there exists $\lambda>0$ for which  $\rho_{p(\cdot)}(\lambda f)<\infty$. The quasinorm $\|\cdot\|_{ L^{p(\cdot)}(X)}$ on $L^{p(\cdot)}(X)$ is defined by
$$\|f\|_{L^{p(\cdot)}(X)}
    =\inf\left\{\lambda>0\ : \ \ \rho _{p(\cdot )}\left(\frac{f}{\lambda}\right)\leq 1\right\},\;\;\;f\in L^{p(\cdot)}(X).$$

As it was mentioned, the theory of variable exponent Lebesgue spaces
in $\mathbb{R}^n$ can  be found in the monographs \cite{CrF} and
\cite{DHHR}. A very important problem in this field was to find
conditions for the exponent $p(\cdot)$ such that the
Hardy-Littlewood maximal operator $\mathcal{M}$ is bounded in
$L^{p(\cdot)}(\Omega)$ when $\Omega$ is a subset of $\mathbb{R}^n$.
Diening \cite{Die1} and Pick and Ruzicka \cite{PR} obtained the
first results concerning this question. Later, Nekvinda \cite{Ne}
and Cruz-Uribe, Fiorenza and Neugebauer \cite{CUFN} improve them.

Maximal functions are useful tools in our developments. The above
results on  $\mathbb{R}^n$ were extended to spaces of homogeneous
type in \cite{HHP}, \cite{Kha} and \cite{KoS} by considering bounded
spaces or using a ball condition as in \cite{Die1}. The most general
result for the Hardy-Littlewood maximal operator $\mathcal{M}$ on
$L^{p(\cdot)}(X)$ when $X$ is unbounded was established by
Adamowicz, Harjulehto and H\"ast\"o \cite{AHH}. We say that
$p:\;X\;\rightarrow\;(0,\infty)$ is log-H\"older continuous in $X$
when the two following properties are satisfied
\begin{enumerate}
 \item[(i)] There exists $C>0$ such that
 $$|p(x)-p(y)|\leq\displaystyle\frac{C}{\log(e+1/d(x,y))},\;\;x,y\in X,\;\;x\neq y.$$
 \item[(ii)] There exist $x_0\in X$, $C>0$ and $a\in \mathbb{R}$ such that
 $$|p(x)-a|\leq\displaystyle\frac{C}{\log(e+d(x,x_0))},\;\;x\in X.$$
\end{enumerate}

As it was proved in \cite[Lemma 2.1]{AHH} the property (ii) does not depend on $x_0$. We consider the following class of exponent $\mathcal{P}^{log}(X)$ defined by
$$\mathcal{P}^{log}(X)=\{p:\;X\;\rightarrow\;(0,\infty)\;:\;1/p\;\mbox{is log-H\"older continuous in }\;X\}.$$
Note that if $p\in\mathcal{P}^{log}(X)$, then $rp\in \mathcal{P}^{log}(X)$, for every $r>0$. If $p:\;X\;\rightarrow\;(0,\infty)$ is a function we define
$$p_+= ess\,sup_{X}p\;\;\mbox{and}\;\;p_-=ess\,inf_{X}p.$$
We always assume that $0<p_-\le p_+<\infty$.

The following result, \cite[Corollary 1.8]{AHH}, will be very useful in the sequel.

\begin{ThA}\label{ThA}
Let $X$ be a space of homogeneous type. Assume that $p\in
\mathcal{P}^{log}(X)$ with $p_->1$. Then, there exists $C>0$ such
that
$$\|\mathcal{M}f\|_{L^{p(\cdot)}(X)}\leq C\|f\|_{L^{p(\cdot)}(X)},\;\;f\in L^{p(\cdot)}(X).$$
\end{ThA}

 %We now introduce atoms in variable exponent setting. A complex measurable function $a$ defined on $X$ is a $(\infty, p(\cdot))$--atom associated with a ball $B$ when
%\begin{itemize}
%\item[$(i)$] $\supp a \subset B$;
%
%\item[$(ii)$] $\|a\|_\infty \leq \|\chi_B\|^{-1}_{L^{p(\cdot)}(\mu)}$,
%
%\item[$(iii)$] $\int_X a d\mu = 0$.
%\end{itemize}

Recently, Zhuo, Sawano and Yang \cite{ZSY} have studied Hardy spaces
with variable exponents on RD spaces. RD spaces are special spaces
of homogeneous type satisfying an inverse doubling condition (see
\cite{HMY}, \cite{NY} and \cite{YZ}). If $\mu$ is a measure defining
a RD space $X$, then $\mu(\{x\})=0$, $x\in X$. This is an important
point because we are going to define Hardy spaces on graphs that are
spaces of homogeneous type but they are not RD spaces.

%Let now $q\in (1,\infty]$. As in \cite[p. 25]{ZSY} we say that a complex valued measurable function $a$ defined on $X$ is a $(q, p(\cdot))$-atom associated with a ball $B$ when
%\begin{itemize}
%\item[$(i)$] $\supp a \subset B$;
%
%\item[$(ii)$] $\|a\|_{L^q(X,\mu)}\leq \frac{\mu(B)^{1/q}}{\|\chi_B\|_{L^{p(\cdot)}(\mu)}}$,
%
%\item[$(iii)$] $\int_X a d\mu = 0$.
%\end{itemize}
%
%It is clear that if $q\in(1,\infty)$ and $a$ is a $(\infty,p(\cdot))$-atom, then $a$ is also a $(q,p(\cdot))$-atom.

We now describe the graphs that we consider in this paper (see
\cite{B} and \cite{BD}). By $\Gamma$ we denote a countably infinite
set. The elements of $\Gamma$ are called vertices of the graph.
$\nu$ represents a nonnegative symmetric function defined on
$\Gamma\times\Gamma$. We say that two vertices $x,y\in\Gamma$ are
neighbors when $\nu(x,y)> 0$, and in that case we write $x\sim y$.
For every $x\in \Gamma$ we denote $deg(x)=card\{y\in \Gamma:\,x\sim
y\}$. We assume that $deg(x)>0$, for every $x\in \Gamma$, and
$\sup_{x\in \Gamma}deg(x)<\infty$. We now define a measure $\mu$ on
$\Gamma$ as follows. For every $x\in\Gamma$, $\mu(\{x\})$ is given
by
$$\mu(\{x\})=\sum_{x\sim y}\nu (x,y).$$
To simplify we write $\mu(x)$ instead of $\mu(\{x\})$, for every $x\in\Gamma$. For every $A\subset\Gamma$ we define $\mu(A)=\sum_{x\in A}\mu(x)$. Thus, $\mu$ is a measure on $\Gamma$. We assume that $\mu(\Gamma)=+\infty$.

If $x,y\in\Gamma$, we say that they are connected when $x\sim y$ or there exist $x_1,x_2,...,x_n\in\Gamma$ such that $x\sim x_1\sim x_2\sim$ ... $\sim x_n\sim y$, and in that case we call $[x,x_1,...,x_n,y]$
a path of length $n+1$ joining $x$ and $y$. We assume that each $x,y\in\Gamma$ are connected. Note that this fact implies that $\mu(x)\neq 0$, for every $x\in\Gamma$. We now define a distance $d$ on $\Gamma$ as follows
\begin{itemize}
\item $d(x,x)=0$;
\item $d(x,y)=1$ if, and only if, $x\neq y$ and $x\sim y$;
\item $d(x,y)=n\geq 2$ when $x\neq y$, $\nu(x,y)=0$ and $n$ is the minimum length of all possible paths  joining $x$ and $y$.
\end{itemize}
As usual we denote by $B(x,r)=\{y\in\Gamma\;:\;d(x,y)<r\}$, for every $x\in\Gamma$ and $r>0$. Note that, in general, a ball has not a unique center and a unique radius. In particular, $B(x,r)=B(x,[r])$, $x\in \Gamma$ and $r>0$, where $[r]=n$ provided that $n-1<r\le n$, with $n\in \mathbb{N}_+$. Here and in the sequel $\mathbb{N}_+=\mathbb{N}\setminus\{0\}$. The distance $d$ defines the discrete topology on $\Gamma$.

We define on $\Gamma$ the Markov kernel $p$ by
$$p(x,y)=\frac{\nu(x,y)}{\mu(x)}, \;\;\;x,y\in\Gamma.$$

The following two useful properties hold
\begin{enumerate}
\item[(i)] $\sum_{y\in\Gamma}p(x,y)=1,\;\;x\in\Gamma$,
\item[(ii)] $p(x,y)\mu(x)=p(y,x)\mu(y), \;\;x,y\in\Gamma$.
\end{enumerate}
By using the Markov kernel we define the Markov operator $P$ by
$$P(f)(x)=\sum_{y\in\Gamma}p(x,y)f(y),\;\;x\in\Gamma,$$
where $f$ is a complex function defined on $\Gamma$. The operator $L=I-P$ is the discrete Laplace operator on $\Gamma$ associated with $\nu$.

For every $n\in\mathbb{N}$, $n\geq 2$, we consider the $n$-th convolution power of the Markov kernel $p_n$ defined by
$$p_n(x,y)=\sum_{z\in\Gamma}p(x,z)p_{n-1}(z,y)\;\;\;x,y\in\Gamma.$$
We consider $p_1=p$. Then, for every $n\in\mathbb{N}_+$, we have that
$$P^n(f)(x)=\sum_{y\in\Gamma}p_n(x,y)f(y),\;\;x\in\Gamma,$$
and $P^0f=f$, where again $f$ is a complex function defined on $\Gamma$.

In our setting we assume as in \cite{BM} the following conditions.
\begin{enumerate}
\item[(a)] The triple $(\Gamma,\mu,d)$ is a space of homogeneous type. Hence, the measure $\mu$ is doubling with respect to the distance $d$ an there exist $C,D>0$ such that
    \begin{equation}\label{doubling}
    \frac{\mu(B(x,r))}{\mu(B(x,s))}\le C\Big(\frac{r}{s}\Big)^D,\,\,\,x\in \Gamma,\,\,\,0<s\le r.
    \end{equation}

\item[(b)] There exist $C,c>0$ such that
\begin{equation}\label{UE}
p_n(x,y)\leq C\frac{\mu(y)}{\mu(B(x,\sqrt{n}))}e^{-c
d^2(x,y)/n},\;\;\;x,y\in\Gamma\;\mbox{and}\;n\in\mathbb{N}_+.
\end{equation}
 If $\mu$ is doubling  you only need that (\ref{UE}) be true when $x=y$ in order to (\ref{UE}) hold for any pair of $x,y\in \Gamma$.
\item[(c)] For certain $\alpha>0$, $\Gamma$ satisfies property $\Delta(\alpha)$, that is, for every $x\in\Gamma$, $x\sim x$, and  $\nu(x,y)\ge \alpha\mu(x)$, for every $x,y\in \Gamma$ such that $x\sim y$.
\end{enumerate}

We recall (see \cite[(3)]{BD}) that when $\mu$ is doubling, $\Gamma$
satisfies  $\Delta(\alpha)$ with $\alpha >0$, and
$(p_n)_{n=1}^\infty$ verifies (\ref{UE}), then there exist $c_4,C_4>0$ such
that, for every $n\in \mathbb{N}_+$, $k\in\mathbb{N}$ and $x,y,\in\Gamma$,
\begin{equation}\label{compuesto}
|\widetilde{p_{n,k}}(x,y)|\leq
C_4\frac{\mu(y)}{n^k\mu(B(x,\sqrt{n}))}e^{-c_4 d^2(x,y)/n},
\end{equation}
where $\widetilde{p_{n,k}}$ represents the kernel of the operator
$(I-P)^kP^n$, for every $n\in \mathbb{N}_+$ and $k\in\mathbb{N}$.

Our first objective is to define Hardy spaces with variable exponents by using Littlewood-Paley square functions. We consider the square function $S_L$ defined by
$$
S_L(f)(x)=\Big(\sum_{k=1}^\infty \sum_{d(x,y)<k}\frac{|k(I-P)P^{[k/2]}(f)(y)|^2}{k\mu(B(y,k))}\mu (y)\Big)^{1/2}.$$
$S_L$ defines a bounded operator in $L^2(\Gamma)$ (\cite[p. 3460]{BD}).

In \cite{B} and \cite{BD} Hardy spaces on $\Gamma$ were defined. We
extend definitions and results in \cite{BD} to variable exponent
settings.

 We say that $f\in L^2(\Gamma)$ is in $\mathbb{H}^{p(\cdot)}_L(\Gamma)$ when $S_L(f)\in L^{p(\cdot)}(\Gamma)$. We define the Hardy space $H^{p(\cdot)}_L(\Gamma)$ as the completion of $\mathbb{H}^{p(\cdot)}_L(\Gamma)$ with respect to the quasi-norm $\|\cdot\|_{H^{p(\cdot)}_L(\Gamma)}$ where
$$\|f\|_{H^{p(\cdot)}_L(\Gamma)}=\|S_L(f)\|_{p(\cdot)},\;\;\;f\in \mathbb{H}^{p(\cdot)}_L(\Gamma).$$
Here and in the sequel we write $\|\cdot\|_{p(\cdot)}$ to refer to
$\|\cdot\|_{L^{p(\cdot)}(\Gamma)}$.

Before giving a characterization of $H^{p(\cdot)}_L(\Gamma)$ by using atoms we introduce tent spaces in our variable exponent setting.

If $\beta>0$ and $x\in \Gamma$, we denote by $\Upsilon_\beta (x)$
the cone with vertex $x$ and aperture $\beta$ defined by
$$
\Upsilon_\beta (x)=\Big\{(y,k)\in \Gamma \times \mathbb{N}_+: d(y,x)<\beta k\Big\}.
$$
To simplify we write $\Upsilon(x)$ to denote $\Upsilon_1(x)$. If $E\subset\Gamma$ and $\beta>0$ we define the tent $T_\beta(E)$ over $E$ with overture $\beta$ as follows:
$$
T_\beta (E)=(\bigcup_{x\in E^c}\Upsilon_\beta (x))^c.
$$
We also write $T(E)$ to refer to $T_1(E)$. If $x\in \Gamma$ and
$r>0$ we can see that the tent $T(B)$ over $B=B(x_0,r_0)$ is the set
given by
$$T(B)=\Big\{(x,k)\in \Gamma \times \mathbb{N}_+: d(x_0,x)\leq r_0-k\Big\}.$$
Assume that $f$ is a complex valued function defined on $\Gamma \times \mathbb{N}_+$. We consider the following sublinear operator
$$(\mathcal{A}f)(x)=\Big(\sum_{(y,k)\in \Upsilon (x)}\frac{|f(y,k)|^2}{k\mu(B(x,k))}\mu (y) \Big)^{1/2},\quad x\in \Gamma .$$

We define the tent space $T_2^{p(\cdot )}(\Gamma )$ as follows. A complex valued function $f$ defined on $\Gamma \times \mathbb{N}_+$ is said to be in $T_2^{p(\cdot )}(\Gamma )$ provided that $\mathcal{A}(f)\in L^{p(\cdot )}(\Gamma )$. On $T_2^{p(\cdot )}(\Gamma )$ we consider the quasi-norm $\|\cdot \|_{T_2^{p(\cdot )}(\Gamma )}$ given by
$$\|f\|_{T_2^{p(\cdot )}(\Gamma )}=\|\mathcal{A}(f)\|_{p(\cdot )},\quad f\in T_2^{p(\cdot )}(\Gamma ).$$

If $0<q<\infty$, we say that a complex valued function $a$ defined on $\Gamma \times \mathbb{N}_+$ is a $(T_2^{p(\cdot )},q)$-atom when there exists a ball $B=B(x_B,r_B)$, with $x_B\in\Gamma$ and $r_B\ge 1$, such that
\begin{enumerate}
\item[(i)] $\supp a\subseteq T(B)$,
\item[(ii)] $\|a\|_{T_2^q(\Gamma )}\leq \mu(B)^{1/q}\|\chi _B\|_{p(\cdot
)}^{-1}$.
\end{enumerate}
Here and in the sequel if $E\subset \Gamma$, we denote by $\chi_E$
the characteristic function supported in $E$. If $0<r<\infty$,
$T_2^r(\Gamma)$ represents the tent space $T_2^{p(\cdot )}(\Gamma)$
when $p(x)=r$, $x\in\Gamma$.

For every sequences $(\lambda _j)_{j\in \mathbb{N}}$ of complex numbers and $(B_j)_{j\in \mathbb{N}}$ of balls we define
$$
\mathcal{A}(\{\lambda _j\},\{B_j\})=\Big\|\Big(\sum_{j\in \mathbb{N}}\frac{|\lambda _j|^{\mathfrak{p}}\chi _{B_j}}{\|\chi _{B_j}\|_{p(\cdot )}^{\mathfrak{p}}}\Big)^{1/{\mathfrak{p}}}\Big\|_{p(\cdot )},
$$
where $\mathfrak{p}=\min \{1,p_-\}$.

%and
%$$
%\mathcal{A}^*(\{\lambda _j\},\{B_j\})=\inf \Big\{\lambda >0: \sum_{j=1}^\infty\sum_{x\in \Gamma }\Big(\frac{|\lambda _j|}{\lambda \|\chi _{B_j}\|_{p(\cdot )}}\Big)^{p(x)}\mu (x)\leq 1\Big\}.
%$$

We now establish a description of $T_2^{p(\cdot )}(\Gamma)$ by using atoms.

\begin{Th}\label{Th1.1}
 Let $p(\cdot)\in \mathcal{P}^{log}(\Gamma)$ and $1<q<\infty$.

 (i) For certain $C>0$ the following property is satisfied: if $f\in T_2^{p(\cdot )}(\Gamma )$ there exist, for each $j\in \mathbb{N}$, $\lambda_j>0$ and a $(T_2^{p(\cdot )},q)$-atom $a_j$ associated with a ball $B_j$ such that
 \begin{equation}\label{eqTh1.4}
 f(x,k)=\sum_{j\in \mathbb{N}}\lambda_ja_j(x,k),\quad (x,k)\in \Gamma \times \mathbb{N}_+,
 \end{equation}
 where the series converges absolutely for every $(x,k)\in \Gamma \times \mathbb{N}_+$ and
 $$
 \mathcal{A}(\{\lambda _j\},\{B_j\})\leq C\|f\|_{T_2^{p(\cdot )}(\Gamma )}.
 $$
 Moreover, $f=\sum_{j\in \mathbb{N}}\lambda _ja_j$ in the sense of  convergence in $T_2^{p(\cdot )}(\Gamma )$.

 (ii) There exists $C>0$ such that if, for every $j\in \mathbb{N}$, $\lambda_j\in \mathbb{C}$ and $a_j$ is
 a $(T_2^{p(\cdot )}q)$-atom associated with the ball $B_j$ such that $\mathcal{A}(\{\lambda _j\},\{B_j\})<\infty$, then the series $f=\sum_{j\in \mathbb{N}}\lambda _ja_j$  converges absolutely for every $(x,k)\in \Gamma \times \mathbb{N}_+$ and in $T_2^{p(\cdot )}(\Gamma )$. Furthermore
 $$
 \|f\|_{T_2^{p(\cdot )}(\Gamma )}\leq C\mathcal{A}(\{\lambda _j\},\{B_j\}).
 $$

 (iii) If $0<r<\infty$ and $f\in T_2^{p(\cdot )}(\Gamma )\cap T_2^r(\Gamma )$, then the series in (\ref{eqTh1.4}) also converges to $f$ in $T_2^r(\Gamma)$.

\end{Th}

By applying Theorem \ref{Th1.1} we obtain an atomic decomposition for the Hardy space $H^{p(\cdot)}_L(\Gamma)$.

Let $M\in \mathbb{N}_+$ and $1<q<\infty$. We say that $a\in L^q(\Gamma )$ is a $(q,p(\cdot ), M)$-atom associated with a ball $B=B(x_B, r_B)$, with $x_B\in \Gamma $ and $r_B\geq 1$, when there exists $b\in L^q(\Gamma )$ satisfying that:

$(i)$ $a=L^Mb$;

$(ii)$ $\supp L^kb\subset B$, $k=0,...,M$;

$(ii)$ $\|L^kb\|_q\leq (r_B)^{M-k}(\mu(B))^{1/q}\|\chi _B\|_{p(\cdot )}^{-1}$, $k=0,...,M$.

A function $f\in L^2(\Gamma )$ is in $\mathbb{H}^{p(\cdot)}_{L,M,at}(\Gamma)$ when there exist, for every $j\in \mathbb{N}$, $\lambda _j\in \mathbb{C}$ and  a $(2,p(\cdot ),M)$-atom $a_j$ associated with the ball $B_j$ such that
$$
f=\sum_{j\in \mathbb{N}} \lambda_ja_j,\quad \mbox{ in }L^2(\Gamma ),
$$
and $\mathcal{A}(\{\lambda_j\},\{B_j\})<\infty$. For every $f\in \mathbb{H}^{p(\cdot)}_{L,M,at}(\Gamma)$ we define  $\|f \|_{H_{L,M,at}^{p(\cdot )}(\Gamma)}$ by
$$
\|f\|_{H_{L,M,at}^{p(\cdot )}(\Gamma)}=\inf \mathcal{A}(\{\lambda _j\},\{B_j\}),
$$
where the infimum is taken of over all the pair of sequences $\{\lambda _j\}_{j\in \mathbb{N}}$ and $\{B_j\}_{j\in \mathbb{N}}$ satisfying that, for every $j\in \mathbb{N}$, $\lambda _j\in \mathbb{C}$ and  there exists a $(2,p(\cdot ), M)$-atom $a_j$ associated with the ball $B_j$  such that $
f=\sum_{j\in \mathbb{N}} \lambda_ja_j,$ in $L^2(\Gamma )$ and $\mathcal{A}(\{\lambda_j\},\{B_j\})<\infty$. By $H^{p(\cdot)}_{L,M,at}(\Gamma)$ we represent the completion of $\mathbb{H}^{p(\cdot)}_{L,M,at}(\Gamma)$ with respect to the quasi-norm $\|\cdot \|_{H_{L,M,at}^{p(\cdot )}(\Gamma)}$.

By using Theorem \ref{Th1.1} we prove that the Hardy spaces $H^{p(\cdot)}_{L}(\Gamma)$ and $H^{p(\cdot)}_{L,M,at}(\Gamma)$ coincide.

\begin{Th}\label{Th1.2}
Let $p\in \mathcal{P}^{log}(\Gamma)$, $r\geq 2$, $r>p_+$, $M\in \mathbb{N}_+$, and $M>2D/p_-$. The following assertions hold.

(a) There exists $C>0$ satisfying that: if, for every $j\in \mathbb{N}$, $\lambda _j\in \mathbb{C}$ and $a_j$ is a $(r,p(\cdot ),M)$-atom associated with the ball $B_j$ such that $\mathcal{A}(\{\lambda_j\},\{B_j\})<\infty$, then the series $\sum_{j\in \mathbb{N}}\lambda _ja_j$ converges in $H^{p(\cdot )}_L(\Gamma )$ and
$$
\|f\|_{H^{p(\cdot )}_L(\Gamma )}\leq C\mathcal{A}(\{\lambda_j\},\{B_j\}),
$$
where $f=\sum_{j\in \mathbb{N}}\lambda _ja_j$.

(b) There exists $C>0$ such that, for every $f\in H^{p(\cdot )}_L(\Gamma )$, there exist, for each $j\in \mathbb{N}$, $\lambda _j\in \mathbb{C}$ and a $(r,p(\cdot ),M)$-atom $a_j$ associated with the ball $B_j$ such that
$$
f=\sum_{j\in \mathbb{N}}\lambda_ja_j,\quad \mbox { in }H^{p(\cdot )}_L(\Gamma ),
$$
and
$$
\mathcal{A}(\{\lambda _j\},\{B_j\})\leq C\|f\|_{H^{p(\cdot )}_L(\Gamma)}.
$$
\end{Th}

The Hardy space $H^{p(\cdot )}_L(\Gamma)$ coincides with the space
$L^{p(\cdot)}(\Gamma)$ provided that $p\in
\mathcal{P}^{\log}(\Gamma)$ and $p_->1$ (see Proposition
\ref{Equal}).
%Now we are in condition to give a maximal characterization of $H^{p(\cdot )}_L(\Gamma )$.
%
%\begin{Th}\label{Th1.3}
%Suppose that $\Gamma$ is a graph satisfying the Poincar\'e's property and $\Delta(\delta)$ for some $\delta >0$. If $p_0\in(0,1)$ is the one given in Theorem \ref{Th1.3}, and $p\in \mathcal{P}^{log}_d$, being $p_->p_0$, $H^{p(\cdot )}_L(\Gamma )$ is contained in $H^{p(\cdot )}_+(\Gamma )$ and, there exists $C>0$ such that
%$$
%\|f\|_{H^{p(\cdot )}_+(\Gamma )}\leq C\|f\|_{H^{p(\cdot )}_L(\Gamma )},\;\;\;f\in H^{p(\cdot )}_L(\Gamma ).
%$$
%\end{Th}

Next we introduce the molecules in our setting that will be useful
to study the boundedness of operators in $H^{p(\cdot )}_L(\Gamma )$
(see Section 5).

Let $M\in \mathbb{N}_+$, $1<q<\infty$, and $\varepsilon >0$. We say
that a function $m:\Gamma\longrightarrow \mathbb{C}$ is a
$(q,p(\cdot ),M,\varepsilon )$-molecule when there exist a function
$b:\Gamma\longrightarrow \mathbb{C}$ and a ball $B=B(x_B,r_B)$ with
$x_B\in \Gamma$ and $r_B\geq 1$ such that

(i) $m=L^Mb$ and,

(ii) For every $k=0,...,M$,
$$
\|L^kb\|_{L^q(\mathfrak{S}_j(B))}\leq (r_B)^{M-k}2^{-j\varepsilon}(\mu(B(x_B,2^jr_B))^{1/q}\|\chi _{B(x_B,2^jr_B)}\|_{p(\cdot )}^{-1},\quad j\in \mathbb{N},
$$
where, for every $j\in \mathbb{N}_+$, $\mathfrak{S}_j(B)=B(x_B,2^{j+1}r_B)\setminus B(x_B,2^{j-1}r_B)$, and $\mathfrak{S}_0(B)=B$.

Every $(q,p(\cdot ),M)$-atom is
also a $(q,p(\cdot ),M, \varepsilon)$-molecule, for every
$\varepsilon >0$.

\begin{Th}\label{Th1.4}
Let $p\in \mathcal{P}^{log}(\Gamma)$, $M\in \mathbb{N}_+$,
$M>2D/p_-$ and $\varepsilon>D/p_-$, $q\ge 2$, and $q>p_+$. There
exists $C>0$ satisfying that: if, for every $j\in \mathbb{N}$,
$\lambda _j\in \mathbb{C}$ and $m_j$ is a
$(q,p(\cdot),M,\varepsilon)$-molecule associated with the ball
$B_j$, such that $\mathcal{A}(\{\lambda _j\},\{B_j\})<\infty$, then
the series $\sum_{j\in \mathbb{N}} \lambda _jm_j\in H^{p(\cdot
)}_L(\Gamma )$ converges in $H^{p(\cdot )}_L(\Gamma )$ and
$$
\|f\|_{H^{p(\cdot )}_L(\Gamma )}\leq
C\mathcal{A}(\{\lambda_j\},\{B_j\}),
$$
where $f=\sum_{j\in \mathbb{N}}\lambda _jm_j$.
\end{Th}

We now consider, for every $f:\Gamma\longrightarrow \mathbb{C}$, the
radial maximal function ${\mathcal M}_+(f)$ given by
$${\mathcal M}_+(f)(x)=\sup_{k\in\mathbb{N}}|P^k(f)(x)|,\;\;\;x\in\Gamma.$$
We define
$$\mathbb{H}^{p(\cdot)}_{L,+}(\Gamma)=\{f\in L^2(\Gamma)\;:\;\mathcal{M}_+(f)\in L^{p(\cdot)}(\Gamma)\},$$
and we denote by $H^{p(\cdot)}_{L,+}(\Gamma)$ the completion of
$\mathbb{H}^{p(\cdot)}_{L,+}(\Gamma)$ with respect to the quasi-norm
$\|\cdot\|_{H^{p(\cdot)}_{L,+}(\Gamma)}$ defined by
$$\|f\|_{H^{p(\cdot)}_{L,+}(\Gamma)}=\|\mathcal{M}_+(f)\|_{L^{p(\cdot)}(\Gamma)},\;\;\;f\in L^2(\Gamma).$$
We establish (see Proposition \ref{M+}) that
$H^{p(\cdot)}_L(\Gamma)$ is a subspace of
$H^{p(\cdot)}_{L,+}(\Gamma)$.

As it was mentioned Zhuo, Sawano and Yang \cite{ZSY} defined Hardy
spaces with variable exponent on homogeneous spaces $(X,\mu,d)$ of
$RD$-type. If $(X,\mu,d)$ is a $RD$-space of homogeneous type,
$\mu(x)=0$, for every $x\in X$. Hence, our graphs are not
$RD$-spaces.

 In \cite{CoW2} in the context of homogeneous type spaces atomic Hardy
spaces are considered. Mac\'{\i}as and Segovia \cite{MS}
characterized those Hardy spaces by using a grand-maximal function
provided that $(X,\mu,d)$ is a normal homogeneous type space. We say
that a homogeneous type space is normal when the following property
holds: there exist $A_1,A_2,K>0$ such that, for every $x\in X$,
\begin{enumerate}
\item[(i)] $A_1r\le \mu(B(x,r))$, $r>0$,
\item[(ii)] $A_2r\ge \mu(B(x,r))$, $r\ge K\mu(x)$,
\item[(iii)] $B(x,r)=\{x\}$, $0<r<K\mu(x)$.
\end{enumerate}
Also, a normal space $(X,d,\mu)$ is said to have order $\alpha>0$
when
$$
|d(x,z)-d(y,z)|\le Cr^{1-\alpha}d(x,y)^{\alpha},
$$
for every $x,y,z\in X$, $d(x,z)<r$, $d(y,z)<r$. Mac\'{\i}as and
Segovia \cite{MS0} proved that if $(X,\mu,d)$ is a space of
homogeneous type there exists a quasimetric $d_1$ on $X$ defining
the same topology as $d$ on $X$ such that $(X,\mu,d_1)$ is a normal
space of order $\alpha$, for some $\alpha>0$. In general, $d$ and
$d_1$ are not comparable.

Later Uchiyama (\cite[Theorem 1 and Corollary 1]{Uchi})
characterized Hardy spaces $H^p(X)$ by using radial maximal
functions associated to certain nonnegative continuous functions
provided that $(X,\mu,d)$ is a space of homogeneous type such that
$\mu(B(x,r))\sim r$, for every $x\in X$ and $r>0$. This result was
extended to $RD$-spaces by Grafakos, Liu, and Yang \cite{GLY1}.

By putting all together the above ideas (see also \cite{Ru1}) we
could think on proving a characterization of our
$H^{p(\cdot)}_L(\Gamma)$ by using the radial maximal function
$\mathcal{M}_+$. The problem is that we would need to consider a
quasimetric $d_1$ topologically equivalent to the graph metric $d$
for which the space $(\Gamma,\mu,d_1)$ is normal and then proceed,
for instance, as in the proof of \cite[Theorem 1.6]{GLY}. But when
we change the quasimetric we can not be sure that our Markov kernel
$p_n$ satisfies the sufficient estimates (exponential upper
bounds,...). At this moment we do not know how to prove the
characterization of the Hardy space $H^{p(\cdot)}_L(\Gamma)$ by
using the radial maximal function $\mathcal{M}_+$.

Yang and Zhuo \cite{YZ2} and \cite {ZY}, and Yang, Zhang and Zhuo
\cite{YZZ} studied variable Hardy spaces associated with operators
$\mathcal{L}$ on $\mathbb{R}^n$ such that the semigroup generated by
$\mathcal{L}$ satisfies some kind of Gaussian or off diagonal
estimates. However, it is not clear for the semigroup generated by
the discrete Laplacian $L$ whether those kind of estimates hold or
not (see \cite{CoMS}, \cite{Pa1} and \cite{Pa2}). Our study relies on
the upper Gaussian estimates for the iterates of the Markov
operators (\ref{UE}).

Bui, Cao, Ky, Yang and Yang \cite{BCKYY} and D. Yang and S. Yang
(\cite{YY1} and \cite{YY2}) studied Musielak-Orlicz-Hardy spaces
$\mathcal{H}_{\varphi,L}(\mathbb{R}^n)$ associated with operators.
Here $\varphi$ is known as a Musielak-Orlicz function. It is an
interesting question to define Musielak-Orlicz-Hardy spaces in our
discrete settings. When we consider as Musielak-Orlicz function
$\varphi(x,t)=t^{p(x)}$, $t>0$ and $x\in \mathbb{R}^n$, the
Musielak-Orlicz-Hardy space $\mathcal{H}_{\varphi,L}(\mathbb{R}^n)$
reduces to the variable exponent Hardy space
$\mathcal{H}_{L}^{p(\cdot)}(\mathbb{R}^n)$. However,
the Musielak-Orlicz-Hardy spaces $\mathcal{H}_{\varphi,L}(\mathbb{R}^n)$
are defined requiring certain conditions for $\varphi$ (for
instance, $\varphi$ is a uniform Muckenhoupt weight) that are not
always satisfied when $\varphi(x,t)=t^{p(x)}$ and $p$ is log-Holder
continuous. Hence, the studies about Musielak-Orlicz-Hardy spaces
and variable exponent Hardy spaces associated with operators do not
cover each other.

This paper is organized in the following way. Section 2 is dedicated
to establish some results that will be very useful throughout this
work. Theorem \ref{Th1.1} is proved in Section 3 while the proofs of
Theorems \ref{Th1.2} and \ref{Th1.4} are established in Section 4. Theorem \ref{Th1.2} is separated in Propositions \ref{H6} and \ref{Hardy}.
In Section 5 we use atomic and molecular characterizations of
$H^{p(\cdot)}_L(\Gamma)$ to study
$H^{p(\cdot)}_L(\Gamma)$-boundedness properties of certain
Littlewood-Paley square functions, Riesz transforms, and spectral
multipliers for the discrete Laplacian $L$.

From now on $C$ and $c$ represent positive constants that can change in each occurrence.

$\mathbf{Acknowledgements}$. The authors would strongly like to give
thanks to Professor Dachun Yang for sending us his paper \cite{ZSY}
(jointly with C. Zhuo and Y. Sawano).

\section{Auxiliary results}

In this section we present some results that will be very useful in
the sequel.

By using Rubio de Francia extrapolation theorem (see, for instance,
\cite{CFMP}) we can obtain the following Fefferman-Stein vector
valued inequality in our variable exponent setting. This property
also can be seen as a special case of \cite[Theorem 2.7]{ZSY}.

\begin{Lem} \label{LemaFS} Assume that $p\in \mathcal{P}^{log}(\Gamma)$, $p_->1$, and $1<q<\infty$. Then,
there exists $C>0$ such that, for every sequence $\{f_j\}_{j\in
\mathbb{N}}\subset L^{p(\cdot)}(\Gamma)$, we have that
$$
\Big\|\Big(\sum_{j\in
\mathbb{N}}\mathcal{M}(f_j)^q\Big)^{1/q}\Big\|_{p(\cdot)}\le
C\Big\|\Big(\sum_{j\in
\mathbb{N}}|f_j|^q\Big)^{1/q}\Big\|_{p(\cdot)}.
$$
\end{Lem}

The following result is a special case of \cite[Proposition
2.11]{ZSY}.

\begin{Lem} \label{LemaSum} Let $p\in \mathcal{P}^{log}(\Gamma)$ and $q\in [1,\infty]\cap(p_+,\infty]$.
There exists $C>0$ such that if, for every $j\in \mathbb{N}$,
$\lambda_j\in \mathbb{C}$, $a_j\in L^q(\Gamma)$ and $B_j$ is a ball
in $\Gamma$ satisfying that

(i) $supp(a_j)\subset B_j$,

(ii) $\|a_j\|_{q}\le \mu(B_j)^{1/q}\|\chi_{B_j}\|_{p(\cdot)}^{-1}$,

\noindent then
$$
\Big\|\Big(\sum_{j=0}^\infty
|\lambda_ja_j|^\mathfrak{p}\Big)^{1/\mathfrak{p}}\Big\|_{p(\cdot)}\le
C\mathcal{A}(\{\lambda_j\},\{B_j\}).
$$
\end{Lem}

The results proved in the next lemma are consequence of Theorem A and Lemma
\ref{LemaSum}.

\begin{Lem} \label{LemaCoc} Let $p\in \mathcal{P}^{log}(\Gamma)$.

(i) If  $0<w<p_-$, there exists $C>0$ such that, for every $x\in \Gamma$,
$\beta>1$ and $r>0$, we have that
$$
\frac{\|\chi_{B(x,\beta r)}\|_{p(\cdot)}}{\|\chi_{B(x,
r)}\|_{p(\cdot)}}\le C\beta^{D/w}.
$$

(ii) If $q\in [1,\infty)\cap (p_+,\infty)$, there exists $C>0$ such that, for every $x\in \Gamma$,
$\beta>1$ and $r>0$, we have that
$$
\frac{\|\chi_{B(x,r)}\|_{p(\cdot)}}{\|\chi_{B(x, \beta
r)}\|_{p(\cdot)}}\le C\Big(\frac{\mu(B(x,r))}{\mu(B(x,\beta
r))}\Big)^{1/q}.
$$
\end{Lem}

\begin{proof}  We consider  $x_0\in
\Gamma$, $r_0>0$, and $\beta>1$.

$(i)$ We choose $0<w<p_-$. We have that, for every $x\in B(x_0,\beta
x_0)$,
\begin{align}\label{2.3.1}
\mathcal{M}(\chi_{B(x_0, r_0)})(x)&\ge \frac{1}{\mu(B(x,2\beta
r_0))}\sum_{y\in B(x,2\beta r_0)}\chi_{B(x_0, r_0)}(y)\mu(y)\nonumber\\
&\ge  \frac{1}{\mu(B(x_0,3\beta r_0))}\sum_{y\in B(x_0,2\beta
r_0)}\chi_{B(x_0, r_0)}(y)\mu(y)\ge C\beta^{-D}\chi_{B(x_0,\beta
r_0)}(x).
\end{align}
By Theorem A we deduce that
\begin{align*}
\|\chi_{B(x_0,\beta r_0)}\|_{p(\cdot)}&\le
C\beta^{D/w}\|(\mathcal{M}(\chi_{B(x_0, r_0)}))^{1/w}\|_{p(\cdot)}\\
&= C\beta^ {D/w}\|\mathcal{M}(\chi_{B(x_0,
r_0)})\|_{p(\cdot)/w}^{1/w}\\
&\le C\beta^{D/w}\|\chi_{B(x_0,r_0)}\|_{p(\cdot)/w}^{1/w}\\
&= C\beta^{D/w}\|\chi_{B(x_0,r_0)}\|_{p(\cdot)}.
\end{align*}

$(ii)$ Let $q\in [1,\infty)\cap (p_+,\infty)$. We define the function $a_0=\lambda_0\chi_{B(x_0, r_0)}$,
where $\lambda_0=(\mu(B(x_0,\beta r_0))/\mu(B(x_0, r_0)))^{1/q}$. It
is clear that $\supp(a_0)\subset B(x_0,\beta r_0)$ and $\|a_0\|_q\le
(\mu(B(x_0,\beta r_0)))^{1/q}$. Then, according to Lemma
\ref{LemaSum}, we get
$$
\|a_0\|_{p(\cdot)}\le C\|\chi_{B(x_0,\beta r_0))}\|_{p(\cdot)},
$$
and the proof is completed.
\end{proof}

By proceeding as in the last proof but using Lemma \ref{LemaFS}
instead of Theorem A we can obtain the following lemma.

\begin{Lem} \label{LemaSumCoc} Let $p\in \mathcal{P}^{log}(\Gamma)$ and $0<w<p_-$.Then, there exists
$C>0$ such that, for every $\beta\ge 1$ and every sequences
$\{x_j\}_{j\in \mathbb{N}}\subset \Gamma$, $\{r_j\}_{j\in \mathbb{N}}\subset
(0,\infty)$, and $\{\lambda_j\}_{j\in \mathbb{N}}\subset \mathbb{C}$ we
have that
$$
\mathcal{A}(\{\lambda_j\},\{B(x_j,\beta r_j)\})\le
C\beta^{D/w}\mathcal{A}(\{\lambda_j\},\{B(x_j,r_j)\}).
$$
\end{Lem}

%%%%%%%%%%%%%%%%%%%%%%%%%%%%%%%%%%%%%%%%%%%%%%%%%%%%%%%%%%%%%%%%%%%
%\AJC{Notas Jorge: pp. T1 $\to$ T21}
\section{Tent spaces of variable exponents on graphs. (Proof of Theorem \ref{Th1.1})}
\label{Sect:Tent}
%%%%%%%%%%%%%%%%%%%%%%%%%%%%%%%%%%%%%%%%%%%%%%%%%%%%%%%%%%%%%%%%%%%

Tent spaces were introduced by Coifman, Meyer and Stein in \cite{CMS}. These spaces play an important role in the development of the theory of Hardy spaces in different settings. Discrete tent spaces were considered in \cite{BD} (see also \cite{B}) to define Hardy spaces associated with operators on graphs. In this section we study tent spaces of variable exponents on graphs. Discrete tent spaces in \cite{BD} are particular cases of our variable exponent tent spaces on graphs.\\

\begin{proof}[Proof of Theorem \ref{Th1.1}]
In order to prove this result we follow the ideas developed in
\cite[Proof of Theorem 1, (c)]{CMS} (see also \cite[Theorem
1.1]{Ru}). We need to make some modifications  because we have
variable exponents (see \cite[Theorem 2.16]{ZYL} for a proof in the
continuous case). We need to introduce the concept of
$\gamma$-density. Suppose that $F$ is a subset of $\Gamma$ such that
the complement $F^c$ of $F$ has finite $\mu$-measure. Let $0<\gamma
<1$. We say that $x\in \Gamma$ has global $\gamma$-density with
respect to $F$ when, for every $r>0$,
$$
\frac{\mu (F\cap B(x,r))}{\mu (B(x,r))}\geq \gamma.
$$
We denote by $F_\gamma ^*$ the set of all those elements of $\Gamma$ with global $\gamma$-density with respect to $F$. It is clear that $F_\gamma ^*\subset F$. Also we have that
$$
(F_\gamma ^*)^c=\{x\in \Gamma : \mathcal{M}(\chi _{F^c})(x)>1-\gamma \}.
$$
Here $\mathcal{M}$ denotes the centered Hardy-Littlewood maximal
function. Since $\mathcal{M}$ is of weak type (1,1) and we are
dealing with a space of homogeneous type,   there exists $C>0$ such
that
\begin{equation}\label{A1}
\mu ((F_\gamma ^*)^c)\leq C\frac{\mu (F^c)}{1-\gamma }.
\end{equation}
Here $C$ does not depend on $F$.

$(i)$ Suppose that $f\in T^{p(\cdot)}_2(\Gamma)\cap T^2_2(\Gamma)$. Let $k\in \mathbb{Z}$. We define
$$
O_k=\{x\in \Gamma : \mathcal{A}(f)(x)>2^k\},
$$
and $F_k=O_k^c$. Since $\mathcal{A}(f)\in L^{p(\cdot )}(\Gamma )$ and $p_+<\infty$, we have that $\sum_{x\in \Gamma }|\mathcal{A}(f)(x)|^{p(x)}\mu (x)<\infty$ and then $\mu (O_k)<\infty $.\\

Fix $\eta,\, \gamma \in (0,1)$. We could take for instance
$\eta=\gamma=1/2$ but to simplify we prefer keep writing $\eta$ and
$\gamma$. A careful reading of \cite[Lemma 2.1]{Ru} allows us to
ensure that there exists $C>0$ for which
\begin{align*}
\sum_{(y,t)\in \cup_{x\in (F_k)_\gamma ^*} \Upsilon_{1-\eta}(x)}|f(y,t)|^2\frac{\mu(y)}{t}&=\sum_{(y,t)\in \cup_{x\in (F_k)_\gamma ^*} \Upsilon_{1-\eta}(x)}|f(y,t)|^2\mu(B(y,t))\frac{\mu(y)}{t\mu(B(y,t))}\\
&\leq C\sum_{x\in F_k} \sum_{(y,t)\in \Upsilon (x)}|f(y,t)|^2\frac{\mu (y)}{t\mu(B(y,t))}\mu (x)\\
&\le C\sum_{x\in F_k}(\mathcal{A}(f)(x))^2\mu (x).
\end{align*}
Since $f\in T_2^2(\Gamma)$ by using dominated convergence theorem we get that
$$
\lim_{k\rightarrow -\infty}\sum_{x\in F_k}(\mathcal{A}(f)(x))^2\mu (x)=\lim_{k\rightarrow -\infty }\sum_{x\in \Gamma} \chi _{F_k}(x)(\mathcal{A}(f)(x))^2\mu (x)=0.
$$
We have that
$$
\sum_{(y,t)\in \cap _{k\in \mathbb{Z}}(\cup _{x\in (F_k)_\gamma ^*}\Upsilon_{1-\eta}(x))}|f(y,t)|^2\frac{\mu(y)}{t}=0.
$$
It follows that $f(y,t)=0$, $(y,t)\in \cap _{k\in \mathbb{Z}}(\cup _{x\in (F_k)_\gamma ^*}\Upsilon_{1-\eta}(x))$. Hence,
$$
\mbox{supp} \;f\subset \bigcup _{k\in \mathbb{Z}}T_{1-\eta }(((F_k)_\gamma ^*)^c).
$$
We apply \cite[Lemma 2.2]{Ru} to $\Omega _k=((F_k)_\gamma ^*)^c$,
$k\in \mathbb{Z}$. Note that $\mu(O_k)\le\mu (\Omega_k )\le \frac{C}{1-\gamma}\mu(O_k)<\infty$ and $\Omega_k
\not=\Gamma$, $k\in \mathbb{Z}$. There exists $\mathcal{C}>0$ such
that for every $k\in \mathbb{Z}$ there exists a set $I_k\subset
\mathbb{N}$, and, for every $n\in I_k$, $x_n^k\in \Gamma$ and
$\varphi _n^k:\Gamma \longrightarrow [0,\infty )$ satisfying, by
taking $r_n^k=d(x_n^k,\Omega_k^c)/10$, that:

\begin{itemize}
\item $\Omega _k=\cup _{n\in I_k}B(x_n^k, r_n^k)$;
\item $B(x_i^k,r_i^k/4)\cap B(x_j^k,r_j^k/4)=\emptyset$, provided that $i,j\in I_k$, $i\not=j$;
\item $ card\{m\in I_k: B(x_n^k, 5r_n^k)\cap B(x_m^k, 5r_m^k)\not=\emptyset \}\leq \mathcal{C},\quad n\in I_k;$
\item $\mbox{supp}\;\varphi _n^k\subset B(x_n^k,2r_n^k)$, $n \in I_k$;
\item $\varphi _n^k(x)\geq \mathcal{C}^{-1}$, $x\in B(x_n^k,r_n^k)$, $n\in I_k$;
\item $\sum_{n\in I_k}\varphi _n^k=\chi _{\Omega_k}$.
\end{itemize}

Since $\mbox{supp}\;f\subset \cup_{k\in \mathbb{Z}}T_{1-\eta}(\Omega_k)$ as in \cite[p. 131]{Ru} we have that, for every $(x,t)\in \Gamma \times \mathbb{N}_+$,
\begin{equation}\label{T1}
f(x,t)=\sum_{k\in \mathbb{Z}}\Big[\sum_{j\in I_k}f(x,t)\varphi _j^k(x)(\chi _{T_{1-\eta}(\Omega_k)}-\chi _{T_{1-\eta}(\Omega_{k+1})}(x,t)\Big].
\end{equation}
%Note that the last series actually has a finite number of nonzero terms for every $(x,t)\in \Gamma\times \mathbb{N}_+$. Also, for every $k\in \mathbb{Z}$ and $j\in I_k$, $\mbox{supp }\;(b_j^k)\subset T_1(B(x_j^k, (2+12/(1-\eta ))r_j^k))\cap (T_{1-\eta}(((F_{k+1})^*))^c)$, where
%$$
%b_j^k(x,t)=f(x,t)\varphi _j^k(x)(\chi _{T_{1-\eta}(\Omega_k)}-\chi _{T_{1-\eta}(\Omega_{k+1})})(x,t),\quad (x,t)\in \Gamma \times \mathbb{N}_+,
%$$
%\red{(\cite[p. 132]{Rus})}.\\

We now write the equality (\ref{T1}) as follows
\begin{equation}\label{T2}
f(x,t)=\sum_{k\in \mathbb{Z}}\sum_{j\in I_k}\lambda _j^ka_j^k(x,t),\,\,\,(x,t)\in \Gamma\times \mathbb{N}_+,
\end{equation}
where
$$
\lambda _j^k=2^k\|\chi _{B(x_j^k,C_\eta r_j^k)}\|_{p(\cdot)},
$$
with $C_\eta=(2+12/(1-\eta ))$ and, for each $(y,t)\in \Gamma\times \mathbb{N}_+$,
$$
a_j^k(y,t)=\frac{1}{\lambda_j^k}f(y,t)\varphi _j^k(y)(\chi _{T_{1-\eta}(\Omega_k)}-\chi _{T_{1-\eta}(\Omega_{k+1})})(y,t),
$$
for every $k\in \mathbb{Z}$ and $j\in I_k$.

Note that, for every $k\in \mathbb{Z}$ and $j\in I_k$, $\mbox{supp
}(a_j^k)\subset T(B(x_j^k, C_\eta r_j^k))\cap
T_{1-\eta}(\Omega _k) \cap (T_{1-\eta}(\Omega_{k+1}))^c$ \cite[p.
132]{Ru}.

We are going to see that there exists $C>0$ such that, for every $k\in \mathbb{Z}$ and $j\in I_k$, $Ca_j^k$ is a $(T_2^{p(\cdot )},2)$-atom. \\

Let $k\in \mathbb{Z}$ and $j\in I_k$. Suppose that $h\in T_2^2(\Gamma)$ with $\|h\|_{T_2^2(\Gamma )}\leq 1$. By using \cite[Lemma 2.1]{Ru} and H\"older's inequality we get
\begin{align*}
&\Big|\sum_{(y,t)\in \Gamma \times \mathbb{N}_+}a_j^k(y,t)h(y,t)\frac{\mu(y)}{t}\Big| \\
&\qquad =\Big|\sum_{(y,t)\in (T_{1-\eta}(\Omega_{k+1}))^c}a_j^k(y,t)h(y,t)\frac{\mu(y)}{t}\Big|\\
& \qquad \leq C\sum_{x\in F_{k+1}}\sum_{(y,t)\in \Upsilon (x)}|a_j^k(y,t)||h(y,t)|\frac{\mu (y)}{t\mu(B(y,t))}\mu (x)\\
& \qquad \leq C\sum_{x\in F_{k+1}}\mathcal{A}(a_j^k)(x)\mathcal{A}(h)(x)\mu (x)\\
& \qquad \leq \frac{C}{\lambda_j^k}\sum_{x\in F_{k+1}}\mathcal{A}(h)(x)\mathcal{A}(f\varphi _j^k\chi _{T(B(x_j^k, C_\eta r_j^k))})(x)\mu (x)\\
& \qquad \leq \frac{C}{\lambda_j^k}\sum_{x\in F_{k+1}\cap B(x_j^k, C_\eta r_j^k)}\mathcal{A}(h)(x)\mathcal{A}(f)(x)\mu (x),
\end{align*}
because $x\in B(x_j^k, C_\eta r_j^k)$ provided that $(y,t)\in \Upsilon (x) \cap T(B(x_j^k, C_\eta r_j^k)))$.\\

Then H\"older's inequality leads to
\begin{align*}
& \Big|\sum_{(y,t)\in \Gamma \times \mathbb{N}_+}a_j^k(y,t)h(y,t)\frac{\mu(y)}{t}\Big| \\
& \qquad \leq \frac{C}{\lambda_j^k}\Big(\sum_{x\in F_{k+1}\cap B(x_j^k, C_\eta r_j^k)}|\mathcal{A}(f)(x)|^2\mu (x)\Big)^{1/2}\nonumber\\
& \qquad \leq C\mu(B(x_j^k, C_\eta r_j^k))^{1/2}\|\chi _{B(x_j^k, C_\eta r_j^k)}\|_{p(\cdot )}^{-1}.
\end{align*}
Since $(T_2^2(\Gamma ))'=T_2^2(\Gamma )$ we conclude that
$$
\|a_j^k\|_{T_2^2(\Gamma )}\leq C\mu(B(x_j^k, C_\eta r_j^k))^{1/2}\|\chi _{B(x_j^k, C_\eta r_j^k)}\|_{p(\cdot )}^{-1}.
$$
Hence $a_j^k/C$ is a $(T_2^{p(\cdot )},2)$-atom. Note that $C$ does not depend on $k$ nor on $j$.\\

In a similar way we can see that, for every $r\in (1,\infty)$, since $(T_2^r(\Gamma))'=T_2^{r'}(\Gamma)$, there exists $C>0$ such that, for every $k\in \mathbb{Z}$ and $j\in I_k$,
\begin{equation}\label{Tr}
\|a_j^k\|_{T_2^r(\Gamma )}\leq C\mu(B(x_j^k, C_\eta r_j^k))^{1/r}\|\chi _{B(x_j^k, C_\eta r_j^k)}\|_{p(\cdot )}^{-1},
\end{equation}
and then $a_j^k/C$ is a $(T_2^{p(\cdot)},r)$-atom.

Our next objective is to see that $\mathcal{A}(\{\lambda _j^k\} ,\{B(x_j^k, C_\eta r_j^k)\} )\leq C\|f\|_{T_2^{p(\cdot )}(\Gamma )}$. In order to do this we proceed as in \cite[p. 1569]{ZYL}. By (\ref{2.3.1}), for every $j,k\in \mathbb{N}$, we have that
$$
\mathcal{M}(\chi _{B(x_j^k ,r_j^k)})(x) \ge C\chi_{B(x_j^k ,C_\eta r_j^k)})(x),\,\,\,x\in \Gamma.
$$

%On the other hand, we can write
%\begin{align*}
%& \sum_{j,k}\sum_{x\in \Gamma}\Big(\frac{|\lambda _j^k}|}{\lambda \|\chi _{B(x_j^k,(2+12/(1-\eta)) r_j^k)}\|_{p(\cdot)}}\Big)^{p(\cdot )}\mu (x)\\
%& \qquad =\int_\Gamma \sum_{j,k}\Big(\frac{|\lambda _{j,k}|^{\mathfrak{p}}\chi _{B(x_j^k, (2+12/(1-\eta))r_j^k)}}{\lambda ^{\mathfrak{p}}\|\chi _{B(x_j^k, (2+12/(1-\eta))r_j^k)}\|_{p(\cdot)}^{1/\mathfrak{p}}}\Big)^{p(\cdot )/\mathfrak{p}}d\mu (x)\\
%& \qquad \leq \int_\Gamma \Big[\frac{1}{\lambda}\Big(\sum_{j,k}\frac{|\lambda _{j,k}|^{\mathfrak{p}}\chi _{B(x_j^k,(2+12/(1-\eta)) r_j^k)}}{\lambda ^{\mathfrak{p}}\|\chi _{B(x_j^k,(2+12/(1-\eta)) r_j^k)}\|_{p(\cdot)}^{1/\mathfrak{p}}}\Big)^{1/\mathfrak{p}}\Big]^{p(\cdot )}d\mu (x).
%\end{align*}

Then, we get
\begin{align*}
 \mathcal{A}(\{\lambda _j^k\},\{B(x_j^k,C_\eta r_j^k)\})
&=\Big\|\Big(\sum_{j,k}\frac{|\lambda _j^k|^{\mathfrak{p}}\chi _{B(x_j^k,C_\eta r_j^k)}}{\|\chi _{B(x_j^k,C_\eta r_j^k)}\|_{p(\cdot)}^{\mathfrak{p}}}\Big)^{1/\mathfrak{p}}\Big\|_{p(\cdot )}\\
& =\Big\|\Big(\sum_{j,k}\Big(2^{kr}\chi _{B(x_j^k,C_\eta r_j^k)})^{\mathfrak{p}/r}\Big)^{1/\mathfrak{p}}\Big\|_{p(\cdot )}\\
&\leq C\Big\|\Big(\sum_{j,k}\Big(\mathcal{M}(2^{kr}\chi _{B(x_j^k,r_j^k)})\Big)^{\mathfrak{p} /r}\Big)^{1/\mathfrak{p}}\Big\|_{p(\cdot )}.
\end{align*}
Here $0<r<\mathfrak{p}$. According to Lemma \ref{LemaFS} it follows that
\begin{align*}
& \mathcal{A}(\{\lambda _j^k\},\{B(x_j^k, C_\eta r_j^k)\})
\leq C\Big\|\Big(\sum_{j,k}\big(\mathcal{M}(2^{kr}\chi _{B(x_j^k,r_j^k)})\Big)^{\mathfrak{p} /r}\Big)^{r/\mathfrak{p}}\Big\|_{p(\cdot )/r}^{1/r}\\
& \qquad \leq C\Big\|\Big(\sum_{j,k}(2^{kr}\chi _{B(x_j^k,r_j^k)})^{\mathfrak{p} /r}\Big)^{r/\mathfrak{p}}\Big\|_{p(\cdot )/r}^{1/r}
\leq C\Big\|\Big(\sum_{k}2^{k\mathfrak{p}}\chi _{\Omega _k}\Big)^{1/\mathfrak{p}}\Big\|_{p(\cdot )}.
\end{align*}
If $x\in \Omega _k$, then $\mathcal{M}(\chi _{O_k})(x)>1-\gamma$. Hence, $\chi _{\Omega_k}\leq\frac{1}{1-\gamma}\mathcal{M}(\chi _{O_k})$. By proceeding as above we obtain
\begin{align*}
\mathcal{A}(\{\lambda _j^k\},\{B(x_j^k,C_\eta r_j^k)\})&\leq C\Big\|\Big(\sum_{k}\Big(\mathcal{M}(2^{kr}\chi _{O_k})\Big)^{\mathfrak{p} /r}\Big)^{r/\mathfrak{p}}\Big\|_{p(\cdot )/r}^{1/r}\\
&\leq C\Big\|\Big(\sum_{k}2^{k\mathfrak{p}}\chi _{O_k}\Big)^{1/\mathfrak{p}}\Big\|_{p(\cdot )}.
\end{align*}

Since $\{x\in \Gamma :\mathcal{A}(f)(x)=+\infty \}= \varnothing$, we can write
\begin{align*}
\sum_{k\in \mathbb{Z}}2^{k\mathfrak{p}}\chi _{O_k}&=\sum_{k\in \mathbb{Z}}2^{k\mathfrak{p}}\sum_{m=k}^\infty \chi _{O_m\setminus O_{m+1}}=\sum_{m\in \mathbb{Z}}\chi_{O_m\setminus O_{m+1}}\sum_{k=-\infty }^m2^{k\mathfrak{p}}\\
&=\frac{1}{1-2^{-\mathfrak{p}}}\sum_{m\in \mathbb{Z}}2^{m\mathfrak{p}}\chi _{O_m\setminus O_{m+1}}\leq \frac{1}{1-2^{-\mathfrak{p}}}(\mathcal{A}(f))^{\mathfrak{p}}.
\end{align*}

We get
\begin{equation}\label{EST}
\mathcal{A}(\{\lambda _j^k\},\{B(x_j^k,C_\eta r_j^k)\})\leq C\|\mathcal{A}(f)\|_{p(\cdot )}.
\end{equation}

We are going to see that the equality (\ref{T2}) also holds in
$T_2^{p(\cdot )}(\Gamma )$, that is, the series converges to $f$ in
$T_2^{p(\cdot )}(\Gamma )$. Assume that $\{(k_\ell ,j_\ell )\}_{\ell
\in \mathbb{N}} $ represents an ordenation in the set $\{(k,j): k\in
\mathbb{Z}, \,\,j\in I_k\}$. Note that the series in (\ref{T2}) is
absolutely pointwisely convergent in $\Gamma\times \mathbb{N}_+$.
Then, we can write
\begin{equation}\label{T4}
f(y,t)=\sum_{\ell\in \mathbb{N}} \lambda _{j_\ell }^{k_\ell }a_{j_\ell }^{k_\ell }(y,t),\quad\;\;(y,t)\in \Gamma \times \mathbb{N}_+.
\end{equation}

Since $f\in T_2^{p(\cdot)}(\Gamma)$, according to (\ref{EST}), the series
$$
\sum_{\ell \in \mathbb{N}} \frac{|\lambda _{j_\ell }^{k_\ell }|^{\mathfrak{p} }}{\|\chi _{B(x_{j_\ell}^{k_\ell }, C_\eta r_{j_\ell }^{k_\ell })}\|_{p(\cdot )}^\mathfrak{p}}
\chi _{B(x_{j_\ell}^{k_\ell }, C_\eta r_{j_\ell }^{k_\ell })}
$$
converges in $L^{p(\cdot )/\mathfrak{p} }(\Gamma )$.

For every $\ell \in \mathbb{N}$, since $\supp a_{j_\ell }^{k_\ell}\subset T(B(x_{j_\ell }^{k_\ell }, C_\eta r_{j_\ell }^{k_\ell }))$, we have that
$$
\supp \mathcal{A}(a_{j_\ell }^{k_\ell })\subset B(x_{j_\ell }^{k_\ell }, C_\eta r_{j_\ell }^{k_\ell }).
$$
Then, by using (\ref{Tr}), Lemma \ref{LemaSum} leads to
\begin{align*}
 \Big\|\mathcal{A}\Big(\sum_{\ell =\alpha }^\beta\lambda _{j_\ell }^{k_\ell }a_{j_\ell }^{k_\ell}\Big)\Big\|_{p(\cdot )}
&\leq \Big\|\sum_{\ell =\alpha }^\beta |\lambda _{j_\ell }^{k_\ell }|\mathcal{A}(a_{j_\ell }^{k_\ell })\Big\|_{p(\cdot )}\\
&\leq \Big\|\Big(\sum_{\ell =\alpha }^\beta (|\lambda _{j_\ell }^{k_\ell }|\mathcal{A}(a_{j_\ell }^{k_\ell }))^\mathfrak{p} \Big)^{1/\mathfrak{p}}\Big\|_{p(\cdot )}\\
&\leq C\Big\|\sum_{\ell =\alpha }^\beta \frac{|\lambda _{j_\ell }^{k_\ell }|^{\mathfrak{p} }}{\|\chi _{B(x_{j_\ell}^{k_\ell }, C_\eta r_{j_\ell }^{k_\ell })}\|_{p(\cdot )}^\mathfrak{p}}
\chi _{B(x_{j_\ell}^{k_\ell }, C_\eta r_{j_\ell }^{k_\ell })}\Big\|_{p(\cdot )/\mathfrak{p}}^{1/\mathfrak{p}},
\end{align*}
provided that $\alpha ,\beta \in \mathbb{N}$, $\alpha <\beta$.\\

 Hence the series $\sum_{\ell=0}^\infty \lambda _{j_\ell }^{k_\ell }a_{j_\ell }^{k_\ell }$ converges in $T_2^{p(\cdot )}(\Gamma )$ to a certain $g\in T_2^{p(\cdot )}(\Gamma )$. Then (see, for instance, \cite[Lemma 2.4]{CrW}),
 $$
\sum_{x\in \Gamma} \Big(\mathcal{A}\Big(\sum_{\ell =1}^{\beta}\lambda_{j_{\ell}}^{k_{\ell}}a_{j_{\ell}}^{k_{\ell}}- g\Big)(x)\Big)^{p(x)}\mu(x)\to 0,\,\,\,\rm{as}\,\,\, \beta\to\infty,
 $$
 and it follows that, for every $x\in \Gamma$,
 $$
 \mathcal{A}\Big(\sum_{\ell =1}^{\beta}\lambda_{j_{\ell}}^{k_{\ell}}a_{j_{\ell}}^{k_{\ell}}- g\Big)(x)\to 0,\,\,\,\rm{as}\,\,\, \beta\to\infty.
 $$
 We deduce that
 $$\sum_{\ell =1}^{\beta}\lambda _{j_{\ell}}^{k_{\ell}}a_{j_{\ell}}^{k_{\ell}}(y,t)\to g(y,t),\,\,\,\rm{as} \,\,\,\beta\to\infty,$$
 for every $(y,t)\in \Gamma\times \mathbb{N}_+$. Hence $g=f$. Thus, we prove that the series in (\ref{T2}) converges to $f$ in $T_2^{p(\cdot )}(\Gamma )$.

 %By \cite[version Lemma 3.2.10, (a)]{DHHR}, there exists an increasing sequence $(m_k)_{k\in \mathbb{N}}\subset \mathbb{N}$ such that $\mathcal{A}(\sum_{\ell =1}^{m_k}\lambda _{j_\ell }^{k_\ell }a_{j_\ell }^{k_\ell }-g)(x)\longrightarrow 0$, as $k\rightarrow \infty $, for every $x\in \Gamma $. Then, for every $x\in \Gamma$, there exits an increasing sequence $\{m_k(x)\}_{k\in \mathbb{N}_+}\subset \mathbb{N}$ such that
%$$
%\lim_{k\rightarrow \infty}\sum_{\ell =1}^{m_k(x)}\lambda _{j_\ell }^{k_\ell }a_{j_\ell }^{k_\ell }(y,t)=g(y,t),\quad (y,t)\in \Gamma (x).
%$$
%A diagonal argument allows us to conclude that there exists an increasing sequence $(M_k)_{k\in \mathbb{N}}\subset \mathbb{N}_+$ such that
%$$
%\lim_{k\rightarrow \infty}\sum_{\ell =1}^{M_k}\lambda _{j_\ell }^{k_\ell }a_{j_\ell }^{k_\ell }(y,t)=g(y,t),\quad (y,t)\in \Gamma \times \mathbb{N}_+.
%$$

Suppose now that $1<r<\infty$ and $f\in T_2^{p(\cdot)}(\Gamma)$.
Since $T_2^{p(\cdot)}(\Gamma)\cap T_2^{2}(\Gamma)$ is dense in
$T_2^{p(\cdot)}(\Gamma)$, for every $k\in \mathbb{N}_+$, there
exists $f_k\in T_2^{p(\cdot)}(\Gamma)\cap T_2^{2}(\Gamma)$ such that
$\|f_k-f\|_{T_2^{p(\cdot)}(\Gamma)}\le
2^{-k}\|f\|_{T_2^{p(\cdot)}(\Gamma)}$. Also, we take $f_0=0$. We
have that $f=\sum_{k=1}^\infty (f_k-f_{k-1})$, where the series
converges in $T_2^{p(\cdot)}(\Gamma)$.

According to the first part of this proof, for every $k\in \mathbb{N}_+$, we can write
$$
f_k-f_{k-1}=\sum_{j=1}^\infty \lambda_j^ka_j^k,
$$
where the series is absolutely pointwisely convergent and it
converges in $T_2^{p(\cdot)}(\Gamma)$ where, for each $j\in
\mathbb{N}_+$, $a_j^k$ is a $( T_2^{p(\cdot)},r)$-atom associated
with the ball $B_j^k$ and $\lambda_j^k>0$ satisfying that
$\sum_{j=1}^\infty
(\lambda_j^k)^\mathfrak{p}\chi_{B_j^k}\|\chi_{B_j^k}\|^{-\mathfrak{p}}_{p(\cdot)}$
converges in $L^{p(\cdot)/\mathfrak{p}}(\Gamma)$ and
$$
\Big\|\Big(\sum_{j=1}^\infty
(\lambda_j^k)^\mathfrak{p}\chi_{B_j^k}\|\chi_{B_j^k}\|^{-\mathfrak{p}}_{p(\cdot)}\Big)^{1/\mathfrak{p}}\Big\|_{p(\cdot)}\le
C\|f_k-f_{k-1}\|_{T_2^{p(\cdot)}(\Gamma)}.
$$
Here $C>0$ does not depend on $k$.

 Assume that $\{(k_\ell,j_\ell):\,\ell\in \mathbb{N}\}$ is an ordenation of $\mathbb{N}_+\times\mathbb{N}_+$. We are going to see that  $\sum_{\ell=0}^\infty \lambda_{j_\ell}^{k_\ell}a_{j_\ell}^{k_\ell}=f$ in the sense of convergence in $T_2^{p(\cdot)}(\Gamma)$. Indeed, let $\epsilon>0$. There exists $\delta\in \mathbb{N}_+$ such that $\|\sum_{k=1}^\delta (f_k-f_{k-1})-f\|_{T_2^{p(\cdot)}(\Gamma)}<\epsilon$. By Lemma \ref{LemaSum} we can write, for every $m\in \mathbb{N}_+$,
\begin{align*}
\Big\|\sum_{k=1}^\delta\sum_{j=m}^\infty \lambda_j^ka_j^k\Big\|_{T_2^{p(\cdot)}(\Gamma)}&\le \Big\|\sum_{k=1}^\delta\sum_{j=m}^\infty \lambda_j^k\mathcal{A}(a_j^k)\Big\|_{p(\cdot)}\\
&\le \Big\|\Big(\sum_{k=1}^\delta\sum_{j=m}^\infty ( \lambda_j^k\mathcal{A}(a_j^k))^{\mathfrak{p}}\Big)^{1/\mathfrak{p}}\Big\|_{p(\cdot)}\\
&\le C \Big\|\Big(\sum_{k=1}^\delta\sum_{j=m}^\infty \frac{|\lambda_j^k|^\mathfrak{p}}{\|\chi_{B_j^k}\|_{p(\cdot)}^\mathfrak{p}}\chi_{B_j^k}\Big)^{1/\mathfrak{p}}\Big\|_{p(\cdot)}\\
&\le C \Big\|\sum_{k=1}^\delta\sum_{j=m}^\infty \frac{|\lambda_j^k|^\mathfrak{p}}{\|\chi_{B_j^k}\|_{p(\cdot)}^\mathfrak{p}}\chi_{B_j^k}\Big\|_{p(\cdot)/\mathfrak{p}}^{1/\mathfrak{p}}\\
&\le C \Big(\sum_{k=1}^\delta\Big\|\sum_{j=m}^\infty
\frac{|\lambda_j^k|^\mathfrak{p}}{\|\chi_{B_j^k}\|_{p(\cdot)}^\mathfrak{p}}\chi_{B_j^k}\big\|_{p(\cdot)/\mathfrak{p}}\Big)^{1/\mathfrak{p}}.
\end{align*}

Since, for every $k\in \mathbb{N}_+$, the series $\sum_{j=1}^\infty
\frac{|\lambda_j^k|^\mathfrak{p}}{\|\chi_{B_j^k}\|_{p(\cdot)}^\mathfrak{p}}\chi_{B_j^k}$
converges in $L^{p(\cdot)/\mathfrak{p}}(\Gamma)$, there exists
$m_0\in \mathbb{N}_+$ such that
$$
\Big\|\sum_{j=m_0+1}^\infty
\frac{|\lambda_j^k|^\mathfrak{p}}{\|\chi_{B_j^k}\|_{p(\cdot)}^\mathfrak{p}}\chi_{B_j^k}\Big\|_{p(\cdot)/\mathfrak{p}}<\frac{\epsilon^\mathfrak{p}}{\delta},
$$
for every $k\in \mathbb{N}_+$, $k=1,...,\delta$.

We now choose $L_0\in \mathbb{N}$ such that
$$
\{(k,j)\in \mathbb{N}_+\times\mathbb{N}_+:k=1,...,\delta,\,j=1,...,m_0\}\subset \{(k_\ell,j_\ell):\ell=0,...,L_0\}.
$$

We have that, for every $L>L_0$,
\begin{align*}
&\Big\|\sum_{\ell=0}^L \lambda_{j_\ell}^{k_\ell}a_{j_\ell}^{k_\ell}-f\Big\|_{T_2^{p(\cdot)}(\Gamma)}\\
&\le \Big\|\sum_{\ell=0}^L
\lambda_{j_\ell}^{k_\ell}a_{j_\ell}^{k_\ell}-\sum_{k=1}^\delta
(f_k-f_{k-1})\Big\|_{T_2^{p(\cdot)}(\Gamma)}+
\Big\|\sum_{k=1}^\delta (f_k-f_{k-1})-f\Big\|_{T_2^{p(\cdot)}(\Gamma)}\\
&\le\Big\|\sum_{\ell=0}^L \lambda_{j_\ell}^{k_\ell}a_{j_\ell}^{k_\ell}-\sum_{k=1}^\delta\sum_{j=1}^\infty \lambda_j^ka_j^k\Big\|_{T_2^{p(\cdot)}(\Gamma)}+\epsilon\\
&\le\Big\|\sum_{k=1}^\delta\sum_{j=m_0+1}^\infty
\lambda_j^k\mathcal{A}(a_j^k)\Big\|_{p(\cdot)}+\epsilon\le
C\epsilon.
\end{align*}
Thus, we proved that $\sum_{\ell=0}^\infty
\lambda_{j_\ell}^{k_\ell}a_{j_\ell}^{k_\ell}=f$ in
$T_2^{p(\cdot)}(\Gamma)$.

$(ii)$ Suppose now that $1<q<\infty$ and that, for every $j\in \mathbb{N}$, $\lambda _j\in \mathbb{C}$ and $a_j$ is a $(T^{p(\cdot )}_2,q)$-atom associated with the ball $B_j$, satisfying that $\mathcal{A}(\{\lambda _j\},\{B_j\})<\infty.$
%$$
%\int_\Gamma \Big(\sum_{j\in \mathbb{N}}\Big(\frac{|\lambda _j|\chi _{B_j}}{\|\chi _{B_j}\|_{p(\cdot )}}\Big)^\mathfrak{p} \Big)^{p(x)/\mathfrak{p}}d\mu (x)<\infty.
%$$
%According to \cite[Version of Lemma 2.2]{Ne}
%$$
%\Big(\sum_{j\in \mathbb{N}}\Big(\frac{|\lambda _j|\chi _{B_j}}{\|\chi _{B_j}\|_{p(\cdot )}}\Big)^\mathfrak{p} \Big)^{1/\mathfrak{p}}\in L^{p(\cdot )}(\Gamma ).
%$$
By proceeding as above we get that $\sum_{j\in \mathbb{N}}\lambda_ja_j$ converges in $T_2^{p(\cdot )}(\Gamma )$ and if $f=\sum_{j\in \mathbb{N}}\lambda _ja_j$ in $T_2^{p(\cdot )}(\Gamma )$, we have that
$$
\|f\|_{T_2^{p(\cdot )}(\Gamma )}\leq C\Big\|\Big(\sum_{j\in \mathbb{N}}\frac{|\lambda _j|^\mathfrak{p}}{\|\chi _{B_j}\|_{p(\cdot )}^\mathfrak{p}}\chi _{B_j}\Big)^{1/\mathfrak{p}}\Big\|_{p(\cdot )}=C\mathcal{A}(\{\lambda _j\},\{B_j\}).
$$

On the other hand, for every $j\in \mathbb{N}$, $|a_j|$ is $(T^{p(\cdot )}_2,q)$-atom associated with the ball $B_j$, and $\mathcal{A}(\{|\lambda _j|\},\{B_j\})<\infty$. Then, the series $\sum_{j\in \mathbb{N}}|\lambda_j||a_j|$ converges in $T_2^{p(\cdot )}(\Gamma )$. By taking in mind \cite[Lemma 2.4]{CrW} that also holds in our setting, we deduce that $\sum_{j\in \mathbb{N}}|\lambda_j||a_j(x,k)|<\infty$, for every $(x,k)\in \Gamma\times \mathbb{N}_+$.

$(iii)$ Let $0<r<\infty$. Assume that $f\in T_2^{p(\cdot )}(\Gamma )\cap T_2^r(\Gamma )\cap T_2^2(\Gamma)$. We are going to see that the equality in (\ref{T2}) also holds in $T_2^r(\Gamma )$. As in (\ref{T4}) we write
$$
f(y,t)=\sum_{\ell \in \mathbb{N}}\lambda _{j_\ell }^{k_\ell }a_{j\ell
}^{k_\ell }(y,t),\,\,\;(y,t)\in \Gamma \times \mathbb{N}_+,
$$
where $\{(k_\ell ,j_\ell ):\ell \in \mathbb{N}\}$ represents an ordenation in the set $\{(k,j):k\in \mathbb{Z}, j\in I_k\}$.\\

Suppose firstly that $0<r\leq 1$. We define, for every $k\in
\mathbb{Z}$ and $j\in I_k$,
$$
\widetilde{\lambda _j^k}=2^k\mu (B(x_j^k, C_\eta r_j^k))^{1/r}.
$$
and
$$
\widetilde{a_j^k}=\frac{1}{\widetilde{\lambda _j^k}}f(y,t)\varphi
_j^k(y)\Big(\chi _{T_{1-\eta}(\Omega_k)}-\chi
_{T_{1-\eta}(\Omega_{k+1})}\Big)(y,t).
$$

It is clear that $\widetilde{a_j^k}\widetilde{\lambda
_j^k}=a_j^k\lambda _j^k$, $k\in \mathbb{Z}$, $j\in I_k$. By
proceeding as in the proof of (\ref{Tr}) we can see that there
exists $C>0$ such that for every $k\in \mathbb{Z}$ and $j\in I_k$,
$C\widetilde{a_j^k}$ is a $(T_2^q(\Gamma ),r)$-atom associated with
the ball $B(x_j^k,C_\eta r_j^k)$. Also, we have that
\begin{align*}
\sum_{\ell =0}^\infty |\widetilde{\lambda_{j_\ell}^{k_\ell}}|^r&=\sum_{k\in \mathbb{Z}}\sum_{j\in I_k}|\widetilde{\lambda_j^k}|^r\\
&=\sum_{k\in \mathbb{Z}}\sum_{j\in I_k}2^{kr}\mu (B(x_j^k,C_\eta r_j^k))\\
&\le C\sum_{k\in \mathbb{Z}}\sum_{j\in I_k}2^{kr}\mu (B(x_j^k, r_j^k/4))\\
%&=\sum_{k\in \mathbb{Z}}\sum_{j\in I_k}2^{kq}\sum_{x\in\Gamma }\chi _{B(x_j^k,(2+12/(1-\eta))r_j^k)}(x)\mu (x)\\
%&=\sum_{x\in\Gamma} \sum_{k\in \mathbb{Z}}\sum_{j\in I_k}\Big(2^{kq/2}\chi _{B(x_j^k,(2+12/(1-\eta))r_j^k)}(x)\Big)^2\mu (x)\\
%%&\leq \sum_{x\in \Gamma} \Big(\sum_{k\in \mathbb{Z}}\sum_{j\in I_k}2^{kq/2}\chi _{B(x_j^k,(2+12/(1-\eta))r_j^k)}(x)\Big)^2d\mu (x)\\
%&\leq C\sum_{x\in \Gamma}  \sum_{k\in \mathbb{Z}}\sum_{j\in I_k}\Big(\mathcal{M}(2^{kq/2}\chi _{B(x_j^k,r_j^k)})(x)\Big)^2\mu (x)\\
%&\leq C\sum_{x\in \Gamma}  \sum_{k\in \mathbb{Z}}\sum_{j\in I_k}2^{kq}\chi _{B(x_j^k,r_j^k)}(x)\mu (x)\\
&\leq C \sum_{k\in \mathbb{Z}}2^{kr}\mu(\Omega _k)\\
&\leq C \sum_{k\in \mathbb{Z}}2^{kr}\mu(O_k)\\
&= C \sum_{x\in O_k}\sum_{k\in \mathbb{Z}}2^{kr}\chi_{O_k}(x)\mu(x)\\
&\le C \sum_{x\in O_k}(\mathcal{A}(f)(x))^r\mu(x)\\
&= C\|\mathcal{A}(f)\|_r^r=C\|f\|_{T_2^q(\Gamma )}^r.
\end{align*}
From Lemma \ref{LemaSum} being $p(x)=r$, $x\in \Gamma$, we deduce that
\begin{align*}
\Big\|\mathcal{A}\Big(\sum_{\ell =\alpha }^\beta \lambda_{j_\ell}^{k_\ell }a_{j_\ell }^{k_\ell }\Big)\Big\|_r&\leq \Big\|\sum_{\ell =\alpha }^\beta|\widetilde{\lambda _{j_\ell }^{k_\ell }}|\mathcal{A}(\widetilde{a_{j_\ell }^{k_\ell }})\Big\|_r\\
&\leq \Big\|\Big(\sum_{\ell =\alpha }^\beta (|\widetilde{\lambda _{j_\ell }^{k_\ell}}|\mathcal{A}(\widetilde{a_{j_\ell }^{k_\ell }}))^r\Big)^{1/r}\Big\|_r\\
&\leq C\Big\|\Big(\sum_{\ell =\alpha }^\beta \frac{|\widetilde{\lambda _{j_\ell }^{k_\ell}}|^r}{\mu (B(x_{j_\ell }^{k_\ell }, C_\eta r_{j_\ell }^{k_\ell }))}\chi _{B(x_{j_\ell }^{k_\ell }, C_\eta r_{j_\ell }^{k_\ell })}\Big)^{1/r}\Big\|_r\\
&=C\Big\|\sum_{\ell =\alpha }^\beta \frac{|\widetilde{\lambda _{j_\ell }^{k_\ell}}|^r}{\mu (B(x_{j_\ell }^{k_\ell }, C_\eta r_{j_\ell }^{k_\ell }))}\chi _{B(x_{j_\ell }^{k_\ell }, C_\eta r_{j_\ell }^{k_\ell })}\Big\|^{1/r}_1\\
&\leq C\Big(\sum_{\ell =\alpha }^\beta \frac{|\widetilde{\lambda _{j_\ell }^{k_\ell}}|^r}{\mu (B(x_{j_\ell }^{k_\ell }, C_\eta r_{j_\ell }^{k_\ell }))}\|\chi _{B(x_{j_\ell }^{k_\ell }, C_eta r_{j_\ell }^{k_\ell })}\|_1\Big)^{1/r}\\
&=C\Big(\sum_{\ell =\alpha }^\beta |\widetilde{\lambda _{j_\ell }^{k_\ell }}|^r\Big)^{1/r},\quad \alpha ,\beta \in \mathbb{N},\;\beta >\alpha .
\end{align*}
Hence, the series $\sum_{\ell =0}^\infty \lambda _{j_\ell }^{k_\ell }a_{j_\ell }^{k_\ell }$ converges in $T_2^r(\Gamma )$ and $f=\sum_{\ell =0}^\infty \lambda_{j_\ell }^{k_\ell }a_{j_\ell }^{k_\ell }$ in $T_2^r(\Gamma )$.\\

Suppose now $r\in (1,\infty )$. We have that
$$
F_k=\{x\in \Gamma : \mathcal{A}(f)(x)\leq 2^k\},\quad k\in \mathbb{Z}.
$$
Then, $F_{k-1}\subset F_k$, $k\in \mathbb{Z}$, and $\cap_{k\in \mathbb{Z}}F_k=\{x\in \Gamma :\mathcal{A}(f)(x)=0\}$. Hence, since $f\in T_2^r(\Gamma )$, the monotone convergence theorem leads to
\begin{equation}\label{T4'}
\lim_{k\rightarrow -\infty }\sum_{x\in F_k}(\mathcal{A}(f)(x))^r\mu (x)=0.
\end{equation}

On the other hand, we recall that, for every $k\in \mathbb{Z}$, $\Omega_k =((F_k)^*_\gamma )^c$, and, by (\ref{A1}),
$$
\mu(\Omega_k )\leq C\mu (O_k), \quad k\in \mathbb{Z}.
$$
Also, $O_{k+1}\subset O_k$, $k\in \mathbb{Z}$, and $\cap _{k\in
\mathbb{Z}}O_k=\{x\in \Gamma : \mathcal{A}(f)(x)=+\infty
\}=\emptyset$. Hence, $\lim_{k\rightarrow +\infty }\mu (O_k)=0$, and
then $\lim_{k\rightarrow +\infty }\mu (\Omega_k )=0$. Since $f\in
T_2^r(\Gamma )$, the dominated convergence theorem implies that
\begin{equation}\label{T5}
\lim_{k\rightarrow +\infty }\sum_{x\in \Omega_k}(\mathcal{A}(f)(x))^r\mu (x)=0.
\end{equation}

Let $\varepsilon >0$. By (\ref{T4'}) and (\ref{T5}), there exists $m_0\in \mathbb{N}$ such that
$$
\Big(\sum_{x\in F_{-m_0}}(\mathcal{A}(f)(x))^r\mu (x)\Big)^{1/r}<
\varepsilon,
$$
and
$$
\Big(\sum_{x\in \Omega_{m_0} }(\mathcal{A}(f)(x))^r\mu (x)\Big)^{1/r}<\varepsilon.
$$

For $k\in \mathbb{Z}$ and $j\in I_k$, we define $A_{k,j}=\supp
(a_k^j)$. Then,
$$
A_{k,j}\subset T_1(B(x_j^k, C_\eta r_j^k))\cap (T_{1-\eta }(\Omega_k)\setminus T_{1-\eta }(\Omega_{k+1})),\quad k\in \mathbb{Z},\;j\in I_k.
$$
For every $k\in \mathbb{Z}$, by taking into account the Whitney
covering $\{B(x_j^k,r_j^k)\}_{j\in I_k}$ it follows that
$$
\sum_{k\in \mathbb{Z}}\sum_{j\in I_k}|f\chi _{A_{k,j}}|\varphi
_j^k\leq C|f|.
$$
We define
$$
H_0(x,t)=\sum_{k\in \mathbb{Z}, k<-m_0,\, j\in I_k}|f(x,t)|\varphi
_j^k(x)\chi _{A_{k,j}}(x,t),\quad (x,t)\in \Gamma \times
\mathbb{N}_+.
$$
We have that, for $k\in \mathbb{Z}$, $k<-m_0$ and $j\in I_k$,
$$
A_{k,j}\subset [T_{1-\eta}(\Omega_{k+1})]^c= \bigcup_{x\in
(F_{k+1})^*_\gamma }\Upsilon _{1-\eta}(x)\subset \bigcup_{x\in
(F_{-m_0})^*_\gamma }\Upsilon _{1-\eta}(x).
$$
By using \cite[Lemma 2.1]{Ru} we get, for every $h\in
T_2^{r'}(\Gamma )$ such that $\|h\|_{T_2^{r'}(\Gamma )}\leq 1$,
\begin{align*}
& \Big|\sum_{(x,t)\in \Gamma \times \mathbb{N}_+}H_0(x,t)h(x,t)\frac{\mu(x)}{t}\Big| \\
& \qquad \leq C\sum_{x\in F_{-m_0}}\sum_{(y,t)\in \Upsilon (x)}\sum_{k\in \mathbb{Z},k<-m_0}\sum_{j\in I_k}|f(y,t)|\varphi _j^k(y)\chi _{A_{k,j}}(y,t)|h(y,t)|\frac{\mu (y)}{t\mu(B(y,t))}\mu (x)\\
& \qquad \leq C\sum_{x\in F_{-m_0}}\sum_{(y,t)\in \Upsilon (x)}|f(y,t)||h(y,t)|\frac{\mu (y)}{t\mu(B(y,t))}\mu (x)\\
& \qquad \leq C\sum_{x\in F_{-m_0}}\mathcal{A}(f)(x)\mathcal{A}(h)(x)\mu (x)\\
&\qquad \leq C\Big(\sum_{x\in F_{-m_0}}(\mathcal{A}(f)(x))^r\mu (x)\Big)^{1/r}\|h\|_{T_2^{r'}(\Gamma )}<C\varepsilon.
\end{align*}
We obtain that
\begin{equation}\label{T6}
\|H_0\|_{T_2^r(\Gamma )}<C\varepsilon.
\end{equation}
We define
$$
\mathcal{H}_0(x,t)=\sum_{k\in \mathbb{Z}, k>m_0, j\in
I_k}|f(x,t)|\varphi _j^k(x)\chi _{A_{k,j}}(x,t),\quad (x,t)\in
\Gamma \times \mathbb{N}_+.
$$
Note that
$$
A_{k,j}\subset T_{1-\eta}(\Omega_k)\subset T_{1-\eta}(\Omega_{m_0}),\quad k\in \mathbb{Z}, \;k\geq m_0,\;j\in I_k.
$$
Then, $\supp \mathcal{A}(\mathcal{H}_0)\subset \Omega_{m_0} $. We can write
\begin{equation}\label{T7}
\|\mathcal{H}_0\|_{T_2^r(\Gamma)}^r=\sum_{x\in \Gamma}
(\mathcal{A}(\mathcal{H}_0)(x))^r\mu (x)\leq C\sum_{x\in
\Omega_{m_0}}(\mathcal{A}(f)(x))^r\mu (x)<C\varepsilon^r.
\end{equation}

By (\ref{T6}) and (\ref{T7}) we obtain, for every $m_1\in \mathbb{N}$,
\begin{align*}
& \Big\|\sum_{\stackrel{k\in \mathbb{Z}, j\in I_k}{|k|+j>m_0+m_1}}|f(y,t)|\varphi _j^k(y)\chi _{A_{k,j}}(y,t)\Big\|_{T_2^r(\Gamma )}
\leq \Big\|\sum_{\stackrel{k\in \mathbb{Z}, j\in I_k}{k>m_0, j>m_0+m_1-k}}|f(y,t)|\varphi _j^k(y)\chi _{A_{k,j}}(y,t)\Big\|_{T_2^r(\Gamma )}\\
&\qquad \qquad +\Big\|\sum_{\stackrel{k\in \mathbb{Z}, j\in I_k}{k<-m_0, j>m_0+m_1+k}}|f(y,t)|\varphi _j^k(y)\chi _{A_{k,j}}(y,t)\Big\|_{T_2^r(\Gamma )}\\
& \qquad \qquad +\Big\|\sum_{\stackrel{k\in \mathbb{Z}, j\in I_k}{-m_0\leq k\leq m_0, |k|+j>m_0+m_1}}|f(y,t)|\varphi _j^k(y)\chi _{A_{k,j}}(y,t)\Big\|_{T_2^r(\Gamma )}\\
&\qquad \leq C\varepsilon+\Big\|\sum_{\stackrel{k\in \mathbb{Z}, j\in I_k}{-m_0\leq k\leq m_0, j>m_1}}|f(y,t)|\varphi _j^k(y)\chi _{A_{k,j}}(y,t)\Big\|_{T_2^r(\Gamma )}.
\end{align*}
Let $m_1\in \mathbb{N}$. We define
$$
H_{0,m_1}(y,t)=\sum_{\stackrel{k\in \mathbb{Z}, j\in I_k}{-m_0\leq
k\leq m_0, j>m_1}}|f(y,t)|\varphi _j^k(y)\chi _{A_{k,j}}(y,t),\quad
(y,t)\in \Gamma\times\mathbb{N}_+.
$$
Since $\supp A_{k,j}\subset T_1(B(x_j^k, C_\eta r_j^k))$, we have that
$$
\supp \mathcal{A}(H_{0,m_1})\subset \mathfrak{W}_{m_1},
$$
where
$$
\mathfrak{W}_m=\bigcup_{\stackrel{k\in \mathbb{Z}, j\in
I_k}{-m_0\leq k\leq m_0, j>m}}B(x_j^k, C_\eta r_j^k),\,\,\,m\in \mathbb{N},
$$
and we get
\begin{align*}
\|H_{0,m_1}\|_{T_2^r(\Gamma )}^r&=\sum_{x\in\Gamma}\mathcal{A}\Big(\sum_{\stackrel{k\in \mathbb{Z}, j\in I_k}{-m_0\leq k\leq m_0, j>m_1}}|f(y,t)|\varphi _j^k(y)\chi _{A_{k,j}}(y,t)\Big)^r(x)\mu (x)\\
&\leq C\sum_{x\in\mathfrak{W}_{m_1}}[\mathcal{A}(f)(x)]^q\mu (x).
\end{align*}
By using (\ref{A1}), the doubling property of $\mu$ and the Whitney covering properties we obtain, for each $k\in \mathbb{Z}$,
$$
\sum_{j\in I_k}\mu (B(x_j^k,C_\eta r_j^k))\leq C\sum_{j\in I_k}\mu (B(x_j^k,r_j^k))\leq C\mu (\Omega _k)\leq C\mu (O_k)<\infty.
$$
Then, we have that
$$
\mu (\mathfrak{W}_{\ell}) \leq C\sum_{\stackrel{k\in
\mathbb{Z}}{-m_0\leq k\leq m_0}}\sum_{\stackrel{j\in
I_k}{j>\ell}}\mu (B(x_j^k,r_j^k))\longrightarrow 0,\quad \mbox{ as
}\;\ell \rightarrow +\infty,
$$
and since $f\in T_2^r(\Gamma )$ it follows that
$$
\lim_{\ell \rightarrow +\infty
}\sum_{x\in\mathfrak{W}_\ell}[\mathcal{A}(f)(x)]^r\mu (x)=0.
$$
Then, there exists $m_1\in \mathbb{N}$ such that
$$
\|H_{0,m_1}\|_{T_2^r(\Gamma )}<C\varepsilon.
$$
We conclude that
$$
\Big\|\sum_{\stackrel{k\in \mathbb{Z}, j\in I_k}{|k|+j>m_0+m_1}}|f(y,t)|\varphi _j^k(y)\chi _{A_{k,j}}(y,t)\Big\|_{T_2^r(\Gamma )}<C\varepsilon.
$$
There exists $\ell _0\in \mathbb{N}$ such that $|k_\ell|+j_\ell >m_0+m_1$ provided that $\ell >\ell_0$. We obtain that
$$
\Big\|\sum_{\ell >\ell _0}\lambda_{j_\ell }^{k_\ell }a_{j_\ell }^{k_\ell }\Big\|_{T_2^r(\Gamma )}\leq \Big\|\sum_{\stackrel{k\in \mathbb{Z},j\in I_k}{|k|+j>m_0+m_1}}|f(y,t)|\varphi _j^k(y)\chi _{A_{k,j}}(y,t)\Big\|_{T_2^r(\Gamma )}<C\varepsilon.
$$
Thus we establish that the series $\sum_{\ell\in \mathbb{N}}\lambda_{j_\ell }^{k_\ell }a_{j_\ell }^{k_\ell }$ converges in $T_2^r(\Gamma )$. Also, we have that $f=\sum_{\ell \in \mathbb{N}}\lambda_{j_\ell }^{k_\ell }a_{j_\ell }^{k_\ell }$ in $T_2^r(\Gamma )$.

We define $T_c(\Gamma )$ the space of complex functions $f$ defined
on $\Gamma \times \mathbb{N}_+$ such that $\supp f$ is finite. Note
that in $\Gamma \times \mathbb{N}_+$ the compact sets are the finite
sets. $T_c(\Gamma )$ is a dense subspace of $T_2^{p(\cdot )}(\Gamma
)$ and in $T_2^r(\Gamma)$. Suppose now that $f\in T_2^{p(\cdot)}(\Gamma )\cap T_2^q(\Gamma
)$. There exists a sequence $\{f_k\}_{k\in \mathbb{N}}\subset
T_c(\Gamma)$ such that $f_0=0$, and
$\|f_k-f\|_{T_2^{p(\cdot)}(\Gamma)}\le
2^{-k}\|f\|_{T_2^{p(\cdot)}(\Gamma)}$ and
$\|f_k-f\|_{T_2^{r}(\Gamma)}\le 2^{-k}\|f\|_{T_2^r(\Gamma)}$. By the
above arguments, for every $k\in \mathbb{N}_+$, we can
write
$$
f_k-f_{k-1}=\sum_{j=1}^\infty \lambda_j^ka_j^k,
$$
where the series is absolutely pointwisely convergent, and also in
both $T_2^{p(\cdot)}(\Gamma)$ and in $T^r_2(\Gamma)$, and, for each
$j\in \mathbb{N}_+$, $a_j^k$ is a $ (T_2^{p(\cdot)},r)$-atom
associated with the ball $B_j^k$ and $\lambda_j^k>0$ satisfying that
$\sum_{j=1}^\infty
(\lambda_j^k)^\mathfrak{p}\chi_{B_j^k}\|\chi_{B_j^k}\|^{-\mathfrak{p}}_{p(\cdot)}$
converges in $L^{p(\cdot)/\mathfrak{p}}(\Gamma)$ and
$$
\Big\|\Big(\sum_{j=1}^\infty
(\lambda_j^k)^\mathfrak{p}\chi_{B_j^k}\|\chi_{B_j^k}\|^{-\mathfrak{p}}_{p(\cdot)}\Big)^{1/\mathfrak{p}}\Big\|_{p(\cdot)}\le
C\|f_k-f_{k-1}\|_{T_2^{p(\cdot)}(\Gamma)}.
$$
Here $C>0$ does not depend on $k$.

By proceeding as in the end of part (i)  we can prove
that if $\{(k_\ell,j_\ell):\,\ell\in \mathbb{N}\}$ is an ordenation
of $\mathbb{Z}\times\mathbb{N}_+$, then $\sum_{\ell=0}^\infty
\lambda_{j_\ell}^{k_\ell}a_{j_\ell}^{k_\ell}=f$ in the sense of
convergence in $T_2^{p(\cdot)}(\Gamma)$ and in $T^q_2(\Gamma)$.

\end{proof}

%%%%%%%%%%%%%%%%%%%%%%%%%%%%%%%%%%%%%%%%%%%%%%%%%%%%%%%%%%%%%%%%%%%
%\AJC{Notas Jorge: pp. H1 $\to$ H12}
\section{Variable exponents Hardy spaces on graphs}
\label{Sect:H}
%%%%%%%%%%%%%%%%%%%%%%%%%%%%%%%%%%%%%%%%%%%%%%%%%%%%%%%%%%%%%%%%%%%

We recall the definition of our Hardy spaces. We define, for every
$f\in L^2(\Gamma )$,
$$
S_L(f)(x)=\Big(\sum_{k=1}^\infty \sum_{d(x,y)<k}\frac{|k(I-P)P^{[k/2]}(f)(y)|^2}{k\mu(B(y,k))}\mu (y)\Big)^{1/2},\quad x\in \Gamma .
$$
The Hardy space $H^{p(\cdot )}_L(\Gamma)$ is defined as the
completion of
$$
\mathbb{H}^{p(\cdot )}_L(\Gamma)=\{f\in L^2(\Gamma ): S_L(f)\in
L^{p(\cdot )}(\Gamma )\},
$$
with respect to the quasinorm $\|\cdot \|_{H^{p(\cdot )}_L(\Gamma
)}$ given by
$$
\|f\|_{H^{p(\cdot )}_L(\Gamma )}=\|S_L(f)\|_{p(\cdot )},\quad f\in
\mathbb{H}^{p(\cdot )}_L(\Gamma).
$$

The following result was established in \cite[Theorem 3.7]{BD} (see
also \cite[Proposition 2.1]{Fe1}).
\begin{Prop}\label{Representation}
Let $M\in \mathbb{N}_+$. Then, for every $f\in L^2(\Gamma )$ we have
\begin{equation}\label{B1}
f=\sum_{k=0}^\infty c_{k,M}(I-P)^MP^kf,
\end{equation}
on $L^2(\Gamma )$, where the coefficients $c_{k,N}$, $N\in \mathbb{N}_+$, are defined as follows:

(i) $c_{k,1}=1$, $k\in \mathbb{N}$,

(ii) $c_{k,N+1}=\sum_{j=0}^kc_{j,N}$, $k\in \mathbb{N}$.
\end{Prop}

As above we denote by $T_c(\Gamma )$ the space of complex functions
$f$ defined on $\Gamma \times \mathbb{N}_+$ such that $\supp f$ is
finite.

Motivated by the representation (\ref{B1}) and the definition of
Hardy spaces we introduce the operator $\Pi _M$, $M\in
\mathbb{N}_+$, as follows
$$
\Pi _M(f)(x)=\sum_{k=0}^\infty \frac{c_{k,M+1}}{k+1}(I-P)^MP^{[k/2]}(f(\cdot , k+1))(x),\quad x\in \Gamma ,
$$
for every $f\in T_c(\Gamma )$. These operators were considered in
\cite{BD}. Next we generalize \cite[Proposition 3.14]{BD}.
\begin{Prop}\label{AD}
Let $M\in \mathbb{N}_+$. The operator $\Pi _M$ can be extended from
$T_c(\Gamma )$ to:

(i) $T_2^q(\Gamma )$ as a bounded operator from $T_2^q(\Gamma )$
into $L^q(\Gamma )$, for every $1<q<\infty $. Moreover, if we also
continue denoting $\Pi_M$ to the extension operator, for every $f\in
T_2^q(\Gamma )$, with $1<q<\infty$, we have that
\begin{equation}\label{XXX1}
\Pi _M(f)(x)=\sum_{k=0}^\infty
\frac{c_{k,M+1}}{k+1}(I-P)^MP^{[k/2]}(f(\cdot , k+1))(x),\quad
x\in \Gamma .
\end{equation}

(ii) $T_2^{p(\cdot )}(\Gamma )$ as a bounded operator from
$T_2^{p(\cdot )}(\Gamma )$ into $H^{p(\cdot )}_L(\Gamma )$, provided
that $p\in \mathcal{P}^{\log}(\Gamma)$ and $M>2D/p_-$.

\end{Prop}

\begin{proof} We recall that $T_c(\Gamma )$ is a dense subspace of $T_2^{p(\cdot )}(\Gamma )$, and in particular of $T_2^q(\Gamma )$, $0<q<\infty$.

$(i)$ Let $1<q<\infty$. Assume that $f\in T_c(\Gamma )$ and $h\in L^{q'}(\Gamma )$. We have that
\begin{align*}
\sum_{x\in \Gamma}\Pi _M(f)(x)h(x)\mu (x)&=\sum_{x\in \Gamma }h(x)\sum_{k=0}^\infty \frac{c_{k,M+1}}{k+1}(I-P)^MP^{[k/2]}(f(\cdot , k+1))(x)\mu (x)\\
&=\sum_{k=0}^\infty \frac{c_{k,M+1}}{k+1}\sum_{x\in
\Gamma}h(x)(I-P)^MP^{[k/2]}(f(\cdot , k+1))(x)\mu (x).
\end{align*}
Note that the series appearing above are actually finite sums because $f\in T_c(\Gamma )$.\\

We can write
\begin{align*}
& \sum_{x\in \Gamma}h(x)(I-P)^MP^{[k/2]}(f(\cdot , k+1))(x)\mu (x)\\
& \qquad =\sum_{y\in \Gamma }f(y,k+1)(I-P)^MP^{[k/2]}(h)(y)\mu
(y),\quad k\in \mathbb{N},
\end{align*}
because $p_n(x,y)\mu (x)=p_n(y,x)\mu (y)$, $x,y\in \Gamma $ and
$n\in \mathbb{N}$.

 Since
$0\le c_{k,M+1}\leq (k+1)^M$, $k\in \mathbb{N}$, by using
\cite[Proposition 3.2, (a)]{BD}, we get
\begin{align*}
\Big|\sum_{x\in \Gamma}\Pi _M(f)(x)h(x)\mu (x)\Big|&=\Big|\sum_{y\in \Gamma }\sum_{k=0}^\infty  \frac{c_{k,M+1}}{k+1}f(y,k+1)(I-P)^MP^{[k/2]}(h)(y)\mu (y)\Big|\\
&\leq \sum_{y\in \Gamma }\sum_{k=0}^\infty \frac{c_{k,M+1}}{k+1}|f(y,k+1)||(I-P)^MP^{[k/2]}(h)(y)|\mu (y)\\
&\leq \sum_{y\in \Gamma }\sum_{k=0}^\infty |(k+1)^M (I-P)^MP^{[k/2]}(h)(y)||f(y,k+1)|\frac{\mu (y)}{k+1}\\
&\leq \sum_{y\in \Gamma }\sum_{k=1}^\infty |k^M (I-P)^MP^{[(k-1)/2]}(h)(y)||f(y,k)|\frac{\mu (y)}{k}\\
&\leq \sum_{y\in \Gamma }\mathcal{A}(f)(y)\mathcal{A}(g)(y)\mu (y),
\end{align*}
where
$$
g(x,k)=k^M(I-P)^MP^{[(k-1)/2]}(h)(x),\quad x\in \Gamma, \;k\in
\mathbb{N}_+.
$$

Note that, since
\begin{equation}\label{YY1}
\mathcal{A}(g)(x)=\Big(\sum_{(y,k)\in \Upsilon
(x)}\frac{|k^M(I-P)^MP^{[(k-1)/2]}(h)(y)|^2}{k\mu(B(x,k))}\mu
(y)\Big)^{1/2}=S_{M,L}(h)(x),\quad x\in \Gamma ,
\end{equation}
according to \cite[Proposition 4.6]{B} we obtain that $\mathcal{A}(g)\in L^{q'}(\Gamma)$ and by using H\"older's inequality we
obtain
$$
\Big|\sum_{x\in \Gamma }\Pi _M(f)(x)h(x)\mu (x)\Big|\leq \|\mathcal{A}(f)\|_q\|S_{M,L}(h)\|_{q'}\leq C\|\mathcal{A}(f)\|_q\|h\|_{q'}.
$$
Hence,
$$
\|\Pi_M(f)\|_q\leq C\|f\|_{T_2^q(\Gamma )}.
$$
Thus we prove that $\Pi_M$ can be extended from $T_c(\Gamma)$ to $T_2^q(\Gamma)$ as a bounded operator from $T_2^q(\Gamma)$ into $L^q(\Gamma)$. We continue denoting by $\Pi_M$ to a such extension.\\

Let now $f\in T_2^q(\Gamma)$. For every $\ell, m\in \mathbb{N}_+$,
$$
f_\ell(x,t)=f(x,t)\chi_{(0,\ell]}(t),\,\,\,(x,t)\in \Gamma\times \mathbb{N}_+,
$$
and
$$
f_{\ell,m}(x,t)=f(x,t)\chi_{(0,\ell]}(t)\chi_{B(x_0,m)}(x),\,\,\,(x,t)\in \Gamma\times \mathbb{N}_+.
$$
Here $x_0\in \Gamma$ is fixed.

Let $\ell\in \mathbb{N}_+$. By using dominated convergence theorem we deduce that $f_{\ell,m}\to f_\ell$, as $m\to\infty$, in $T_2^q(\Gamma)$. Then,
$$
\Pi_M(f_{\ell,m})=\sum_{k=0}^{\ell-1}
\frac{c_{k,M+1}}{k+1}(I-P)^MP^{[k/2]}(f_{\ell,m}(\cdot ,
k+1))\longrightarrow \Pi_M(f_\ell),\,\,\,\rm{as}\,\,\,m\to\infty,
$$
in the sense of convergence in $L^q(\Gamma)$. For every $t\in
\mathbb{N}_+$, $f_{\ell,m}(.,t)\to f_\ell(.,t)$, as $m\to\infty$, in
$L^q(\Gamma)$. Indeed, let $t\in \mathbb{N}_+$. By taking into
account that $\mu$ is doubling (see (\ref{doubling})) we have that,
for every $m,\ell\in \mathbb{N}_+$,
\begin{align*}
&\Big(\sum_{x\in \Gamma}\Big(\sum_{(y,s)\in \Upsilon(x)}\frac{|f_{\ell,m}(y,s)-f_\ell(y,s)|^2}{s\mu(B(y,s))}\mu(y)\Big)^{q/2}\mu(x)\Bigg)^{1/q}\\
&\qquad\qquad\ge \Big(\sum_{x\in \Gamma}|f_{\ell,m}(x,t)-f_\ell(x,t)|^q\Big(\frac{\mu(x)}{t\mu(B(x,t))}\Big)^{q/2}\mu(x)\Bigg)^{1/q}\\
&\qquad\qquad\ge Ct^{-(D+1)/2}\Big(\sum_{x\in
\Gamma}|f_{\ell,m}(x,t)-f_\ell(x,t)|^q\mu(x)\Bigg)^{1/q}.
\end{align*}
Since $P$ is a bounded operator in $L^q(\Gamma)$, it follows that
$$
\Pi_M(f_\ell)(x)=\sum_{k=0}^{\ell-1}
\frac{c_{k,M+1}}{k+1}(I-P)^MP^{[k/2]}(f_\ell(\cdot , k+1))(x),\,\,\,x\in
\Gamma.
$$
Also, $f_\ell\to f$, as $\ell\to\infty$, in $T_2^q(\Gamma)$. Then,
$\Pi_M(f_\ell)\to \Pi_M(f)$, as $\ell\to\infty$, in $L^q(\Gamma)$.
Hence, we obtain (\ref{XXX1}).

$(ii)$ Let $f\in T_c(\Gamma )$. We take $q=2\max\{1,p_+\}$.
According to Theorem \ref{Th1.1}, for every $j\in \mathbb{N}$, there
exist $\lambda _j\in (0,\infty)$ and a $(T_2^{p(\cdot )},q)$-atom
$a_j$ associated to a ball $B_j=B(x_{B_j},r_{B_j})$, with
$x_{B_j}\in \Gamma$ and $r_{B_j}\ge 1$, such that
$$
f=\sum_{j\in \mathbb{N}}\lambda _ja_j,
$$
where the series converges in both $T_2^{p(\cdot )}(\Gamma )$ and
$T_2^2(\Gamma )$, and $\mathcal{A}(\{\lambda _j\},\{B_j\})\leq
C\|f\|_{T_2^{p(\cdot )}(\Gamma )}$, where $C$ does not depend on
$f$. Note that as it was shown in (\ref{Tr}), for every $j\in
\mathbb{N}$, $a_j$ is a $(T_2^{p(\cdot)},r)$-atom for every
$1<r<\infty$.

By Proposition \ref{AD}, $(i)$ $\Pi _M$ can be extended to
$T_2^2(\Gamma )$ as a bounded operator from $T_2^2(\Gamma )$ to
$L^2(\Gamma )$. We have that
$$
\Pi _M(f)=\sum_{j=0}^\infty \lambda _j\Pi _M(a_j),\quad \mbox{ in
}\;L^2(\Gamma ).
$$

We are going to see that $S_L(\Pi _M(f))\in L^{p(\cdot )}(\Gamma )$ and that
$$\|S_L(\Pi _M(f))\|_{p(\cdot )}\leq C\|f\|_{T_2^{p(\cdot )}(\Gamma )},$$ for a certain $C>0$ that does not depend on $f$.\\

For every $j\in \mathbb{N}$, we write $\alpha _j=\Pi _M(a_j)$. The operator $P$ is bounded in $L^2(\Gamma )$. Then, for every $k\in \mathbb{N}_+$,
$$
(I-P)P^{[k/2]}(\Pi _M(f))= \sum_{j=0}^{\infty}\lambda
_j(I-P)P^{[k/2]}(\alpha _j), \,\,\rm{in}\,\, L^2(\Gamma).
$$
Hence, for every $k\in \mathbb{N}$,
$$
(I-P)P^{[k/2]}(\Pi _M(f))(x)=\lim_{\ell \rightarrow \infty}
\sum_{j=0}^{\ell}\lambda _j(I-P)P^{[k/2]}(\alpha _j)(x),\quad x\in
\Gamma .
$$
By using Fatou's lemma we get
\begin{align*}
S_L(\Pi _M(f))(x)&=\Big(\sum_{k=1}^\infty \sum_{d(x,y)<k}\lim_{\ell \rightarrow \infty}\frac{|k\sum_{j=0}^{\ell }\lambda _j(I-P)P^{[k/2]}(\alpha _j)(y)|^2}{k\mu(B(y,k))}\mu (y)\Big)^{1/2}\\
&\leq \lim_{\ell \rightarrow \infty }\sum_{j=0}^{\ell }S_L(\lambda
_j\alpha _j)(x)\\
&\leq \sum_{j=0}^{\infty }|\lambda _j|S_L(\alpha _j)(x),\quad x\in
\Gamma .
\end{align*}
We recall that, according to \cite[Proposition 4.6]{B}, $S_L$ is a bounded operator from $L^q(\Gamma )$ into itself.\\

If $B=B(x_B,r_B)$, with $x_B\in \Gamma$ and $r_B\ge 1$, we define $S_0(B)=B(x_B,8Mr_B)$, and, for every $i\in \mathbb{N}_+$, $S_i(B)=B(x_B,2^{i+3}Mr_B)\setminus B(x_B,2^{i+2}Mr_B)$. As above we denote $\mathfrak{p} =\min\{1,p_-\}$. We can write
\begin{align}\label{RR1}
\|S_L(\Pi _M(f))\|_{p(\cdot )}&\leq \Big\|\sum_{i=0}^\infty \sum_{j=0}^\infty |\lambda _j|S_L(\alpha _j)\chi _{S_i(B_j)}\Big\|_{p(\cdot )}\nonumber\\
&\leq \Big\|\sum_{i=0}^\infty \Big(\sum_{j=0}^\infty \Big(|\lambda _j|S_L(\alpha _j)\chi _{S_i(B_j)}\Big)^\mathfrak{p} \Big)^{1/\mathfrak{p}}\Big\|_{p(\cdot )}\nonumber\\
&\leq \Big(\sum_{i=0}^\infty \Big\|\Big(\sum_{j=0}^\infty
\Big(|\lambda _j|S_L(\alpha _j)\chi
_{S_i(B_j)}\Big)^\mathfrak{p}\Big)^{1/\mathfrak{p}}\Big\|_{p(\cdot
)}^\mathfrak{p} \Big)^{1/\mathfrak{p}}.
\end{align}

By using $(i)$ we get
$$
\|S_L(\alpha _j)\|_{L^q(S_0(B_j))}\leq C\|\alpha _j\|_q\leq C\|a_j\|_{T_2^q(\Gamma )}\leq C\mu(B_j)^{1/q}\|\chi _{B_j}\|_{p(\cdot )}^{-1},\quad j\in \mathbb{N}.
$$

Since $q\in [1,\infty)\cap (p_+,\infty)$, by Lemmas \ref{LemaSum}, \ref{LemaCoc}, (i),  and
\ref{LemaSumCoc}, and Theorem \ref{Th1.1} it follows that
\begin{align}\label{RR2}
\Big\|\Big(\sum_{j=0}^\infty \Big(|\lambda _j|S_L(\alpha _j)\chi _{S_0(B_j)}\Big)^\mathfrak{p}\Big)^{1/\mathfrak{p}}\Big\|_{p(\cdot )}&\leq C\Big\|\Big(\sum_{j=0}^\infty \Big(\frac{|\lambda _j|}{\|\chi _{B_j}\|_{p(\cdot )}}\chi _{S_0(B_j)}\Big)^\mathfrak{p}\Big)^{1/\mathfrak{p}}\Big\|_{p(\cdot )}\nonumber\\
&\leq C\mathcal{A}(\{\lambda _j\},\{B_j\})\leq C\|f\|_{T_2^{p(\cdot )}(\Gamma )}.
\end{align}

Let $j\in \mathbb{N}_+$. In order to study $S_L(\alpha _j)$ in $\Gamma \setminus S_0(B_j)$ we write
$$
S_L(\alpha_j)(x)=\Big(\sum_{(y,k)\in \Upsilon(x)}\frac{|k(I-P)^{M+1}P^{[k/2]}(\beta _j)(y)|^2}{k\mu(B(y,k))}\mu (y)\Big)^{1/2},\quad x\in \Gamma ,
$$
where
$$
\beta _j(y)=\sum_{k=0}^{r_{B_j}-1}\frac{c_{k,M+1}}{k+1}P^{[k/2]}(a_j(\cdot , k+1))(y),\quad y\in \Gamma.
$$
Note that $\supp(\beta_j)\subset B(x_{B_j},2r_{B_j})$. We have that, for every $x\in \Gamma$,
\begin{align*}
(S_L(\alpha_j)(x))^2&=|(I-P)^{M+1}(\beta_j)(x)|^2\\
&+ \sum_{k=2}^\infty\sum_{d(y,x)<k}\Big|\sum_{z\in B(x_{B_j},2r_{B_j})}S_{M,k}(y,z)\beta_j(z)\Big|^2\frac{k\mu(y)}{\mu(B(y,k))}.
\end{align*}
Here $S_{M,k}$ denotes the kernel of the operator $(I-P)^{M+1}P^{[k/2]}$, for every $k\in \mathbb{N}$, $k\ge 2$.

Since $\supp(\beta_j)\subset B(x_{B_j},2r_{B_j})$, we obtain
$$
\supp((I-P)^{M+1}\beta_j)\subset \bigcup_{\ell=0}^{M+1}\supp(P^\ell(\beta_j))\subset \bigcup_{\ell=0}^{M+1}B(x_{B_j},2r_{B_j}+\ell+1)\subset B(x_{B_j},(M+4)r_{B_j}).
$$
Also, $d(x,x_{B_j})\ge 8Mr_{B_j}\ge (M+4)r_{B_j}$, provided that $x\in \Gamma\setminus S_0(B_j)$. Then, $(I-P)^{M+1}(\beta_j)(x)=0$, for every $x\in \Gamma\setminus S_0(B_j)$. We can write
\begin{align*}
(S_L(\alpha_j)(x))^2&=\sum_{k=2}^\infty\sum_{d(y,x)<k}\Big|\sum_{z\in B(x_{B_j},2r_{B_j})}S_{M,k}(y,z)\beta_j(z)\Big|^2\frac{k\mu(y)}{\mu(B(y,k))}\\
&=\Big(\sum_{k=2}^{d(x,x_{B_j})/2}+\sum_{k\in \mathbb{N},\,k>d(x,x_{B_j})/2}\Big)\sum_{d(y,x)<k}\Big|\sum_{z\in B(x_{B_j},2r_{B_j})}S_{M,k}(y,z)\beta_j(z)\Big|^2\frac{k\mu(y)}{\mu(B(y,k))}\\
&=I_{j,1}(x)+I_{j,2}(x), \,\,\,x\in \Gamma\setminus S_0(B_j).
\end{align*}
If $z\in B(x_{B_j},2r_{B_j})$, $y\in B(x,k)$, $k\le d(x,x_{B_j})/2$ and $x\in\Gamma\setminus S_0(B_j)$, then $d(y,z)\ge d(x,x_{B_j})/4$. Also, $B(x,d(x,x_{B_j}))\subset B(y,2d(x,x_{B_j}))$, provided that $y\in B(x,k)$, $k\le d(x,x_{B_j})/2$. By using (\ref{doubling}) and (\ref{compuesto}) we deduce that
\begin{align}\label{RR3}
I_{j,1}(x)&\le C\sum_{k=2}^{d(x,x_{B_j})/2}\sum_{d(y,x)<k}e^{-cd(x,x_{B_j})^2/k}\frac{\mu(y)}{k^{2M+1}\mu(B(y,k))(\mu(B(y,\sqrt{k})))^2}\nonumber\\
&\hspace{3cm}\times\Big(\sum_{z\in B(x_{B_j},2r_{B_j})}|\beta_j(z)|\mu(z)\Big)^2\nonumber\\
&\le Ce^{-cd(x,x_{B_j})}\sum_{k=2}^{d(x,x_{B_j})/2}e^{-cd(x,x_{B_j})^2/k}\Big(\frac{d(x,x_{B_j})}{\sqrt{k}}\Big)^{2D}\frac{1}{\mu(B(x,k))}\nonumber\\
&\hspace{3cm}\times\sum_{d(y,x)<k}\frac{\mu(y)}{(\mu(B(y,d(x,x_{B_j}))))^2}\Big(\sum_{z\in B(x_{B_j},2r_{B_j})}|\beta_j(z)|\mu(z)\Big)^2\nonumber\\
&\le C \frac{e^{-cd(x,x_{B_j})}d(x,x_{B_j})}{(\mu(B(x,d(x,x_{B_j}))))^2} \Big(\sum_{z\in B(x_{B_j},2r_{B_j})}|\beta_j(z)|\mu(z)\Big)^2\nonumber\\
&\le C \frac{e^{-cd(x,x_{B_j})}}{(\mu(B(x,d(x,x_{B_j}))))^2} \Big(\sum_{z\in B(x_{B_j},2r_{B_j})}|\beta_j(z)|\mu(z)\Big)^2\\
&\le C \frac{e^{-cd(x,x_{B_j})}}{(\mu(B(x,d(x,x_{B_j}))))^2}\mu(B_j)\|\beta_j\|_2^2, \,\,\,x\in\Gamma\setminus S_0(B_j).\nonumber
\end{align}

On the other hand, by invoking again (\ref{compuesto}), we get
\begin{align*}
I_{j,2}(x)&\le C\sum_{k\in \mathbb{N},\,k>d(x,x_{B_j})/2}\frac{1}{k^{2M+1}\mu(B(x,k))}\\&
\times\sum_{d(y,x)<k}\mu(y)\Big(\sum_{z\in B(x_{B_j},2r_{B_j})}e^{-cd(y,z)^2/k}\frac{|\beta_j(z)|}{\mu(B(y,\sqrt{k}))}\mu(z)\Big)^2,\,\,\,x\in\Gamma\setminus S_0(B_j).
\end{align*}
We have that
\begin{align}\label{max1}
\sum_{y\in \Gamma} &\frac{e^{-cd(x,y)^2/k}}{\mu(B(x,\sqrt{k}))}\mu (y)\nonumber\\
&= \frac{1}{\mu(B(x,\sqrt{k}))}\Big(\sum_{d(x,y)<\sqrt{k}}+\sum_{i=1}^\infty \sum_{d(x,y)\in [i\sqrt{k}, (i+1)\sqrt{k})}\Big)e^{-cd(x,y)^2/k}\mu (y)\nonumber\\
&\leq \frac{1}{\mu(B(x,\sqrt{k}))}\Big(\mu(B(x,\sqrt{k}))+\sum_{i=1}^\infty e^{-ci^2}\mu(B(x,(i+1)\sqrt{k}))\Big)\nonumber\\
&\hspace{3cm}\leq C\sum_{i=0}^\infty (i+1)^De^{-ci^2}\leq C,\quad x\in \Gamma ,
\end{align}
and
\begin{align}\label{max2}
\sum_{y\in \Gamma }&\frac{1}{\mu(B(y,\sqrt{k}))}e^{-cd(y,z)^2/k}\Big)\mu (y)\nonumber\\
&\hspace{3mm}=\Big(\sum_{d(y,z)<\sqrt{k}}+\sum_{i=1}^\infty \sum_{d(y,z)\in [i\sqrt{k}, (i+1)\sqrt{k})}\Big)\frac{1}{\mu(B(y,\sqrt{k}))}e^{-cd(y,z)^2/k}\mu (y)\nonumber\\
&\hspace{3mm}\leq C\Big(\frac{1}{\mu(B(z,\sqrt{k}))}\sum_{d(y,z)<\sqrt{k}}\mu(y)\nonumber\\
&\hspace{1cm}+\sum_{i=1}^\infty \frac{(i+2)^D}{\mu(B(z,\sqrt{k}))}\sum_{d(y,z)\in [i\sqrt{k}, (i+1)\sqrt{k})}e^{-cd(y,z)^2/k}\mu (y)\Big)\nonumber\\
&\hspace{3mm}\leq C\sum_{i=0}^\infty (i+2)^{2D}e^{-i^2}\leq C,\quad z\in \Gamma ,\;k\in \mathbb{N}_+.
\end{align}

Estimates (\ref{max1}), (\ref{max2}) and Jensen's inequality lead to
\begin{align*}
I_{j,2}(x)&\le C\sum_{k\in \mathbb{N},\,k>d(x,x_{B_j})/2}\frac{1}{k^{2M+1}\mu(B(x,k))}\\&
\times\sum_{d(y,x)<k}\mu(y)\sum_{z\in B(x_{B_j},2r_{B_j})}e^{-cd(y,z)^2/k}\frac{|\beta_j(z)|^2}{\mu(B(y,\sqrt{k}))}\mu(z)\\
&\le C\sum_{k\in \mathbb{N},\,k>d(x,x_{B_j})/2}\frac{1}{k^{2M+1}\mu(B(x,d(x,x_{B_j})))}\sum_{z\in B(x_{B_j},2r_{B_j})}|\beta_j(z)|^2\mu(z)\\
&\le \frac{C}{d(x,x_{B_j})^{2M}\mu(B(x,d(x,x_{B_j})))}\|\beta_j\|_2^2,\,\,\,x\in\Gamma\setminus S_0(B_j).
\end{align*}

If $x\in\Gamma\setminus S_0(B_j)$, $B_j\subset B(x_{B_j}, d(x,x_{B_j}))$, and then
$$\mu(B_j)\le \mu(B(x_{B_j}, d(x,x_{B_j}))\le \mu(B(x,2d(x,x_{B_j})))\le C\mu(B(x,d(x,x_{B_j})).$$
By combining the above estimates we obtain
\begin{align*}
S_L(\alpha_j)(x)&\le C\Big(e^{-cd(x,x_{B_j})}\frac{(\mu(B_j))^{1/2}}{\mu(B(x,d(x,x_{B_j})))}+\frac{1}{d(x,x_{B_j})^{M}(\mu(B(x,d(x,x_{B_j}))))^{1/2}}\Big)\|\beta_j\|_2\\
&\le C\frac{\|\beta_j\|_2}{d(x,x_{B_j})^{M}(\mu(B(x,d(x,x_{B_j}))))^{1/2}}\Big(\Big(\frac{\mu(B_j)}{\mu(B(x,d(x,x_{B_j})))}\Big)^{1/2}+1\Big)\\
&\le C\frac{\|\beta_j\|_2}{d(x,x_{B_j})^{M}(\mu(B(x,d(x,x_{B_j}))))^{1/2}},\,\,\,x\in\Gamma\setminus S_0(B_j).
\end{align*}
Hence,
\begin{align*}
\|S_L(\alpha_j)\|_{L^q(S_i(B_j)}&\le C\|\beta_j\|_2\Big(\sum_{x\in S_i(B_j)}\frac{\mu(x)}{d(x,x_{B_j})^{Mq}(\mu(B(x,d(x,x_{B_j}))))^{q/2}}\Big)^{1/q}\\
&\le C\|\beta_j\|_2\frac{1}{(2^ir_{B_j})^M}\frac{(\mu(B(x_{B_j},2^{i+3}Mr_{B_j})))^{1/q}}{(\mu(B(x_{B_j},2^{i+2}Mr_{B_j})))^{1/2}}\\
&\le C \|\beta_j\|_2\frac{(\mu(B(x_{B_j},2^{i}r_{B_j})))^{1/q-1/2}}{(2^ir_{B_j})^M}.
\end{align*}

We now study $\|\beta_j\|_2$. We recall that
$$
\beta_j(y)=\sum_{k=0}^{r_{B_j}-1}\frac{c_{k,M+1}}{k+1}P^{[k/2]}(a_j(\cdot , k+1))(y),\quad y\in \Gamma .
$$
Assume that $h\in L^2(\Gamma )$. We have that
$$
\Big(\sum_{y\in \Gamma }\sum_{k=1}^{r_{B_j}}\frac{|a_j(y,k)|^2}{k}\mu (y)\Big)^{1/2}\le C\|a_j\|_{T^2_2(\Gamma)}.
$$
Since $P$ is a selfadjoint and contractive operator in $L^2(\Gamma )$ and $c_{k,M+1}\leq (k+1)^M$, $k\in \mathbb{N}$, we can write
\begin{align*}
& \Big|\sum_{y\in \Gamma}\beta _j(y)h(y)\mu (y)\Big|
=\Big|\sum_{y\in \Gamma }\sum_{k=0}^{r_{B_j}-1}\frac{c_{k,M+1}}{k+1}P^{[k/2]}(h)(y)a_j(y, k+1)\mu (y)\Big|\\
& \qquad \leq \sum_{y\in \Gamma }\sum_{k=0}^{r_{B_j}-1}(k+1)^{M-1}|P^{[k/2]}(h)(y)||a_j(y,k+1)|\mu (y)\\
& \qquad \leq \Big(\sum_{y\in \Gamma }\sum_{k=1}^{r_{B_j}}|a_j(y,k)|^2k^{M-1}\mu (y)\Big)^{1/2}\Big(\sum_{y\in \Gamma }\sum_{k=0}^{r_{B_j}-1}|P^{[k/2]}(h)(y)|^2(k+1)^{M-1}\mu (y)\Big)^{1/2}\\
&\qquad \le C (r_{B_j})^{M/2}\|a_j\|_{T_2^2(\Gamma)}\Big(\sum_{k=0}^{r_{B_j}-1}(k+1)^{M-1}\|P^{[k/2]}(h)\|_2^2\Big)^{1/2}\\
&\qquad \le  C(r_{B_j})^{M}\|a_j\|_{T_2^2(\Gamma)}\|h\|_2.
\end{align*}
Since $a_j$ is a $(T_2^{p(\cdot )},2)$-atom associated with $B_j$ it
follows that
$$
\|\beta_j\|_2\le C(r_{B_j})^{M}\|a_j\|_{T_2^2(\Gamma)}\le
C(r_{B_j})^{M} (\mu(B_j))^{1/2}\|\chi_{B_j}\|_{p(\cdot)}^{-1}.
$$
We get
\begin{align*}
\|S_L(\alpha_j)\|_{L^q(S_i(B_j))}&\le C(r_{B_j})^{M} (\mu(B_j))^{1/2}\|\chi_{B_j}\|_{p(\cdot)}^{-1}\frac{(\mu(B(x_{B_j},2^{i}r_{B_j})))^{1/q-1/2}}{(2^ir_{B_j})^M}\\
&\le C2^{-iM}(\mu(B(x_{B_j},2^{i+3}Mr_{B_j})))^{1/q}\|\chi_{B_j}\|_{p(\cdot)}^{-1}.
\end{align*}

We choose $0<w<p_-$. Since $q\in [1,\infty)\cap (p_+,\infty)$, according to Lemmas \ref{LemaSum}, \ref{LemaCoc}, (i), and \ref{LemaSumCoc}, we obtain
\begin{align}\label{RR4}
\Big\|\Big(\sum_{j\in \mathbb{N}}(|\lambda _j|&S_L(\alpha _j)\chi _{S_i(B_j)})^\mathfrak{p} \Big)^{1/\mathfrak{p}} \Big\|_{p(\cdot )}
\leq C2^{-i(M-D/w)}\mathcal{A}(\{\lambda_j\},\{B(x_{B_j},2^{i+3}Mr_{B_j})\})\nonumber\\
&\leq C2^{-i(M-2D/w)}\mathcal{A} (\{\lambda _j\},\{B_j\})\nonumber\\
&\le  C2^{-i(M-2D/w)}\|f\|_{T_2^{p(\cdot)}(\Gamma)}.
\end{align}
Since $M>2D/p_-$, by choosing above $0<w<p_-$ such that $M>2D/w$, it follows from (\ref{RR1}), (\ref{RR2}) and (\ref{RR4}) that
$$
\|S_L(\Pi _M(f)\|_{p(\cdot )}\leq \Big(\sum_{i=0}^\infty \Big\|\Big(\sum_{j=1}^\infty (|\lambda_j|S_L(\alpha _j)\chi _{S_i(B_j)})^\mathfrak{p}\Big)^{1/\mathfrak{p}}\Big\|_{p(\cdot) }^\mathfrak{p} \Big)^{1/\mathfrak{p} }\leq C\|f\|_{T_2^{p(\cdot )}(\Gamma )}.
$$

\end{proof}

We recall the definitions of atoms in $H_{L}^{p(\cdot )}(\Gamma)$. Let $M\in \mathbb{N}_+$ and $1<r<\infty$. We say that $a\in
L^r(\Gamma )$ is a $(r,p(\cdot ), M)$-atom associated with a ball
$B=B(x_B, r_B)$, with $x_B\in \Gamma $ and $r_B\geq 1$, when there
exists $b\in L^r(\Gamma )$ satisfying that:

$(i)$ $a=L^Mb$;

$(ii)$ $\supp L^kb\subset B$, $k=0,...,M$;

$(iii)$ $\|L^kb\|_r\leq (r_B)^{M-k}(\mu(B))^{1/r}\|\chi _B\|_{p(\cdot )}^{-1}$, $k=0,...,M$.

If $f\in L^2(\Gamma )$, we say that $f$ has a $(2,p(\cdot ),
M)$-atomic representation when
$$
f=\sum_{j=1}^\infty \lambda _ja_j,\quad \mbox{ in }L^2(\Gamma ),
$$
where, for every $j\in \mathbb{N}$, $\lambda _j\in \mathbb{C}$ and
$a_j$ is a $(2,p(\cdot ), M)$-atom associated with the ball $B_j$, satisfying
that $\mathcal{A}(\{\lambda_j\},\{B_j\})<\infty$.
%$$
%\sum_{x\in \Gamma }\Big(\sum_{j\in \mathbb{N}}\Big(\frac{|\lambda _j|\chi _{B_j}(x)}{\|\chi _{B_j}\|_{p(\cdot )}}\Big)^{\mathfrak{p}}\Big)^{p(x)/\mathfrak{p}}\mu (x)<\infty.
%$$

We define the atomic Hardy space $H_{L,M,at}^{p(\cdot )}(\Gamma )$
as follows: $f\in L^2(\Gamma )$ is in
$\mathbb{H}_{L,M,at}^{p(\cdot )}(\Gamma )$ if and only if $f$ has
a $(2,p(\cdot),M)$-atomic representation. On
$\mathbb{H}_{L,M,at}^{p(\cdot )}(\Gamma )$ as usual we consider
the quasinorm $\|\cdot \|_{H_{L,M,at}^{p(\cdot )}(\Gamma)}$ given by
$$
\|f\|_{H_{L,M,at}^{p(\cdot )}(\Gamma)}=\inf \mathcal{A} (\{\lambda
_j\},\{B_j\}),
$$
where
$$
 \mathcal{A}(\{\lambda _j\},\{B_j\})=\Big\|\Big(\sum_{j\in \mathbb{N}}\Big(\frac{|\lambda _j|\chi _{B_j}}{\|\chi _{B_j}\|_{p(\cdot )}}\Big)^{\mathfrak{p}}\Big)^{1/\mathfrak{p}}\Big\|_{p(\cdot )},
$$
for every sequence $\{\lambda _j\}_{j\in \mathbb{N}}$ of complex
numbers and $\{B_j\}_{j\in \mathbb{N}}$ of balls, and where the
infimum is taken over all the sequences $\{\lambda _j\}_{j\in
\mathbb{N}}$ and $\{B_j\}_{j\in \mathbb{N}}$ such that $f=\sum_{j\in
\mathbb{N}}\lambda _ja_j$, being, for every $j\in \mathbb{N}$, $a_j$
a $(2, p(\cdot ), M)$-atom associated with $B_j$.

The atomic Hardy space $H_{L,M,at}^{p(\cdot )}(\Gamma )$ is
defined as the completion of $\mathbb{H}_{L,M,at}^{p(\cdot
)}(\Gamma )$ with respect to the quasinorm
$\|\cdot\|_{H_{L,M,at}^{p(\cdot )}(\Gamma)}$.

We now prove Theorem \ref{Th1.2} by establishing the next two propositions.
\begin{Prop}\label{H6}
 Assume that $p\in \mathcal{P}^{\log}(\Gamma)$, $r\ge 2$, $r>p_+$ and $M\in \mathbb{N}_+$, $M>2D/p_-$.
 There exists $C>0$ satisfying that: if, for every $j\in \mathbb{N}$, $\lambda _j\in \mathbb{C}$ and $a_j$ is a $(r,p(\cdot ),M)$-atom
 associated with the ball $B_j$ such that $\mathcal{A}(\{\lambda_j\},\{B_j\})<\infty$, then the series $\sum_{j\in \mathbb{N}}\lambda _ja_j$ converges in $H^{p(\cdot )}_L(\Gamma )$ and
$$
\|f\|_{H^{p(\cdot )}_L(\Gamma )}\leq
C\mathcal{A}(\{\lambda_j\},\{B_j\}),
$$
where $f=\sum_{j\in \mathbb{N}}\lambda _ja_j$.
\end{Prop}
\begin{proof}
In order to proof this fact we proceed as in the proof of
Proposition \ref{AD}, (ii). If $B=B(x_B,r_B)$ is a ball we recall that $S_0(B)=B(x_B,8Mr_B)$, and, for every $i\in \mathbb{N}_+$,
$S_i(B)=B(x_B,2^{i+3}Mr_B)\setminus B(x_B,2^{i+2}Mr_B)$. Assume that, for every $j\in \mathbb{N}$, $\lambda _j\in \mathbb{C}$ and $a_j$ is a $(r,p(\cdot ),M)$-atom
 associated with the function $b_j$ and  the ball $B_j=B(x_{B_j},r_{B_j})$, with $x_{B_j}\in \Gamma$ and $r_{B_j}\ge 1$,  such that $\mathcal{A}(\{\lambda_j\},\{B_j\})<\infty$. Let
$\ell_1,\ell_2\in \mathbb{N}$, $\ell_1<\ell_2$. We have that
$$
\Big\|S_L\Big(\sum_{j=\ell_1}^{\ell_2}\lambda_ja_j\Big)\Big\|_{p(\cdot)}\le
\Big(\sum_{i=0}^\infty\Big\|\Big(\sum_{j=\ell_1}^{\ell_2}\Big(|\lambda_j|S_L(a_j)\chi_{S_i(B_j)}\Big)^\mathfrak{p}\Big)^{1/\mathfrak{p}}\Big\|_{p(\cdot)}^\mathfrak{p}\Big)^{1/\mathfrak{p}}.
$$
Since $S_L$ is a bounded operator in $L^r(\Gamma)$, the properties
of the $(r,p(\cdot),M)$-atoms lead to
$$
\|S_L(a_j)\|_{L^r(S_0(B_j))}\le C\|a_j\|_r\le C\mu(B_j)^{1/r}\|\chi_{B_j}\|_{p(\cdot)}^{-1},\,\,\,j\in \mathbb{N}.
$$
Since $r>\max\{1,p_+\}$, by using Lemmas \ref{LemaSum},
\ref{LemaCoc}, (i), and \ref{LemaSumCoc} we get
\begin{align*}
\Big\|\Big(\sum_{j=\ell_1}^{\ell_2}\Big(|\lambda_j|S_L(a_j)\chi_{S_0(B_j)}\Big)^\mathfrak{p}\Big)^{1/\mathfrak{p}}\Big\|_{p(\cdot)}&\le
C\Big\|\Big(\sum_{j=\ell_1}^{\ell_2}\Big(\frac{|\lambda_j|}{\|\chi_{S_0(B_j)}\|_{p(\cdot)}}\chi_{S_0(B_j)}\Big)^\mathfrak{p}\Big)^{1/\mathfrak{p}}\Big\|_{p(\cdot)}\nonumber\\
&\le
C\Big\|\Big(\sum_{j=\ell_1}^{\ell_2}\Big(\frac{|\lambda_j|}{\|\chi_{B_j}\|_{p(\cdot)}}\chi_{B_j}\Big)^\mathfrak{p}\Big)^{1/\mathfrak{p}}\Big\|_{p(\cdot)}.
\end{align*}

Let $j\in \mathbb{N}_+$. We can write
\begin{align*}
(S_L(a_j)(x))^2&=|(I-P)^{M+1}(b_j)(x)|^2\\
&+ \sum_{k=2}^\infty\sum_{d(y,x)<k}\Big|\sum_{z\in B(x_{B_j},r_{B_j})}S_{M,k}(y,z)b_j(z)\Big|^2\frac{k\mu(y)}{\mu(B(y,k))}, \,\,\,x\in \Gamma.
\end{align*}
Here $S_{M,k}$ denotes as above the kernel of the operator $(I-P)^{M+1}P^{[k/2]}$, for every $k\in \mathbb{N}$, $k\ge 2$.

Since $\supp(b_j)\subset B_j$, then $\supp((I-P)^{M+1}(b_j))\subset
B(x_{B_j},(M+2)r_{B_j})$, so $(I-P)^{M+1}(b_j)(x)=0$, $x\in
\Gamma\setminus S_0(B_j)$. Hence,
\begin{align*}
(S_L(a_j)(x))^2&=
\Big( \sum_{k=2}^{d(x,x_{B_j})/2}+\sum_{k\in \mathbb{N},\,k>d(x,x_{B_j})/2}\Big)\sum_{d(y,x)<k}\Big|\sum_{z\in B(x_{B_j},2r_{B_j})}S_{M,k}(y,z)b_j(z)\Big|^2\frac{k\mu(y)}{\mu(B(y,k))}\\
&=I_{j,1}(x)+I_{j,2}(x), \,\,\,x\in \Gamma\setminus S_0(B_j).
\end{align*}

We have that (see (\ref{RR3}))
$$
I_{j,1}(x)\le C\frac{e^{-cd(x,x_{B_j})}}{(\mu(B(x,d(x,x_{B_j})))^2}(\mu(B_j))^{2/r'}\|b_j\|_r^2, \,\,\,x\in\Gamma\setminus S_0(B_j).
$$
Also, since $r\ge 2$, by using Jensen's inequality we get
\begin{align*}
I_{j,2}(x)&\le C\sum_{k\in \mathbb{N},\,k>d(x,x_{B_j})/2}\sum_{d(y,x)<k}\Big(\sum_{z\in \Gamma}\frac{e^{-cd(y,z)^2/k}}{\mu(B(y,\sqrt{k}))}|b_j(z)|^r\mu(z)\Big)^{2/r}\frac{\mu(y)}{k^{2M+1}\mu(B(y,k))}\\
&\le C\sum_{k\in \mathbb{N},\,k>d(x,x_{B_j})/2}\frac{1}{k^{2M+1}\mu(B(x,k))}\sum_{d(y,x)<k}\Big(\sum_{z\in \Gamma}\frac{e^{-cd(y,z)^2/k}}{\mu(B(y,\sqrt{k}))}|b_j(z)|^r\mu(z)\Big)^{2/r}\mu(y)\\
&\le C\sum_{k\in \mathbb{N},\,k>d(x,x_{B_j})/2}\frac{1}{k^{2M+1}}\Big(\sum_{d(y,x)<k}\sum_{z\in \Gamma}\frac{e^{-cd(y,z)^2/k}}{\mu(B(y,\sqrt{k}))}|b_j(z)|^r\mu(z)\frac{\mu(y)}{\mu(B(x,k))}\Big)^{2/r}\\
&\le C\sum_{k\in \mathbb{N},\,k>d(x,x_{B_j})/2}\frac{1}{k^{2M+1}\mu(B(x,k))^{2/r}}\Big(\sum_{z\in \Gamma}|b_j(z)|^r\mu(z)\sum_{y\in \Gamma}\frac{e^{-cd(y,z)^2/k}}{\mu(B(y,\sqrt{k}))}\mu(y)\Big)^{2/r}\\
&\le C\sum_{k\in \mathbb{N},\,k>d(x,x_{B_j})/2}\frac{1}{k^{2M+1}\mu(B(x,k))^{2/r}}\|b_j\|_r^2\\
&\le \frac{C}{d(x,x_{B_j})^{2M}\mu(B(x,d(x,x_{B_j})))^{2/r}}\|b_j\|_r^2,\,\,\,x\in \Gamma\setminus S_0(B_j).
%&\le C\frac{1}{d(x,x_{B_j})^{2M}\mu(B(x,d(x,x_{B_j})))^{2/r}}r_{B_j}^{2M}(\mu(B_j))^{2/r}\|\chi_{B_j}\|_{p(\cdot)}^{-2},\,\,\,x\in \Gamma.
\end{align*}

The properties of $b_j$ and Lemma \ref{LemaCoc}, (i), lead to
\begin{align*}
\|S_L(a_j)\|_{L^r(S_i(B_j))}&\le C\|b_j\|_r\Big(\sum_{x\in S_i(B_j))} \frac{1}{d(x,x_{B_j})^{rM}\mu(B(x,d(x,x_{B_j})))}\\
&\qquad +(\mu(B_j))^{r/r'}\sum_{x\in S_i(B_j))}\frac{e^{-cd(x,x_{B_j})}}{(\mu(B(x,d(x,x_{B_j})))^r}\Big)^{1/r}\\
&\le C(r_{B_j})^{M}(\mu(B_j))^{1/r}\|\chi_{B_j}\|_{p(\cdot)}^{-1}\Big(\frac{\mu(B(x_{B_j},2^ir_{B_j}))}{(2^ir_{B_j})^{Mr}\mu(B(x_{B_j},2^ir_{B_j}))}\Big)^{1/r}\\
%&\qquad +\frac{\mu(B(x_{B_j},2^ir_{B_j}))}{(2^ir_{B_j})^{Mr}\mu(B(x_{B_j},2^ir_{B_j}))}\Big)^{1/r}\\
&\le C2^{-iM}(\mu(B_j))^{1/r}\|\chi_{B_j}\|_{p(\cdot)}^{-1}\\
&\le C2^{-i(M-D/R)}(\mu(B(x_{B_j},2^{i+3}Mr_{B_j})))^{1/r}\|\chi_{B(x_{B_j},2^{i+3}Mr_{B_j})}\|_{p(\cdot)}^{-1},
\end{align*}
where $0<R<p_-$.

According to Lemmas \ref{LemaSum} and \ref{LemaSumCoc}, since $r>\max\{p_+,1\}$, we can
write
$$
\Big\|\Big(\sum_{j=\ell_1}^{\ell_2}\Big(|\lambda_j|S_L(a_j)\chi_{S_i(B_j)}\Big)^\mathfrak{p}\Big)^{1/\mathfrak{p}}\Big\|_{p(\cdot)}\le
C2^{-i(M-2D/R)}\Big\|\Big(\sum_{j=\ell_1}^{\ell_2}\Big(\frac{|\lambda_j|\chi_{B_j}}{\|\chi_{B_j}\|_{p(\cdot)}}\Big)^\mathfrak{p}\Big)^{1/\mathfrak{p}}\Big\|_{p(\cdot)}.
$$

Then, we deduce that
$$
\Big\|S_L\Big(\sum_{j=\ell_1}^{\ell_2}\lambda_ja_j\Big)\Big\|_{p(\cdot)}\le
C\Big\|\Big(\sum_{j=\ell_1}^{\ell_2}\Big(\frac{|\lambda_j|\chi_{B_j}}{\|\chi_{B_j}\|_{p(\cdot)}}\Big)^\mathfrak{p}\Big)^{1/\mathfrak{p}}\Big\|_{p(\cdot)},
$$
because $M>2D/p_-$.

Since the series $\sum_{j\in \mathbb{N}}
\Big(\frac{|\lambda_j|\chi_{B_j}}{\|\chi_{B_j}\|_{p(\cdot)}}\Big)^\mathfrak{p}$
converges in $L^{p(\cdot)/\mathfrak{p}}(\Gamma)$, we conclude that
the series $\sum_{j\in \mathbb{N}} \lambda_ja_j$ converges in
$H^{p(\cdot)}_L(\Gamma)$, and
$$
\|f\|_{H^{p(\cdot)}_L(\Gamma)}\le
\mathcal{A}(\{\lambda_j\},\{B_j\}),
$$
where $f=\sum_{j\in \mathbb{N}} \lambda_ja_j$.

\end{proof}

Our next objective is to see that each $f\in H^{p(\cdot )}_L(\Gamma )$
admits $(r,p(\cdot ),M)$-representations.

\begin{Prop}\label{Hardy}
 Assume that $p\in \mathcal{P}^{\log}(\Gamma)$, $r\ge 2$, $r>p_+$, $M\in \mathbb{N}_+$ and $M>2D/p_-$. There exists $C>0$ such that, for every $f\in H^{p(\cdot )}_L(\Gamma )$ there exist, for every $j\in \mathbb{N}$, $\lambda _j\in \mathbb{C}$ and a $(r,p(\cdot ),M)$-atom $a_j$ associated with $B_j$ such that
$$
f=\sum_{j\in \mathbb{N}}\lambda _ja_j,\quad \mbox { in }H^{p(\cdot
)}_L(\Gamma ),
$$
and
$$
\mathcal{A}(\{\lambda _j\},\{B_j\})\leq C\|f\|_{H^{p(\cdot
)}_L(\Gamma)}.
$$
\end{Prop}

\begin{proof}
Let $f\in L^2(\Gamma )\cap H^{p(\cdot )}_L(\Gamma )$. According to
Proposition \ref{Representation}, we have that
$$
f=\sum_{k=0}^\infty c_{k,M+1}(I-P)^{M+1}P^kf,\quad \mbox{ in
}L^2(\Gamma ),
$$
where $c_{k,1}=1$ and $c_{k,N+1}=\sum_{j=0}^kc_{j,N}$, $N\in \mathbb{N}_+$ and $k\in \mathbb{N}$.\\

We can write
$$
f=\sum_{k=0}^\infty \frac{c_{k,M+1}}{k+1}(I-P)^MP^{[k/2]}((k+1)(I-P)P^{[(k+1)/2]}f),\quad \mbox{ in }L^2(\Gamma ).
$$

Since $f\in L^2(\Gamma)\cap H^{p(\cdot )}_L(\Gamma )$,
$k(I-P)P^{[k/2]}(f)\in T_2^2(\Gamma)\cap T_2^{p(\cdot )}(\Gamma )$,
$k\in \mathbb{N}_+$, and according to Theorem \ref{Th1.1} there
exist, for every $j\in \mathbb{N}$, $\lambda _j\in \mathbb{C}$ and a
$(T^{p(\cdot )}_2,r)$-atom $a_j$ associated to
$B_j=B(x_{B_j},r_{B_j})$, with $x_{B_j}\in \Gamma$ and $r_{B_j}\ge 1$, satisfying that
$$
\mathcal{A}(\{\lambda _j\},\{B_j\})\leq
C\|k(I-P)P^{[k/2]}(f)\|_{T_2^{p(\cdot )}(\Gamma
)}=C\|f\|_{H^{p(\cdot )}_L(\Gamma )},
$$
and
$$
k(I-P)P^{[k/2]}(f)(x)=\sum_{j\in \mathbb{N}}\lambda_ja_j(x,k),\quad (x,k)\in \Gamma \times \mathbb{N}_+,
$$
where the series converges in both $T_2^{p(\cdot )}(\Gamma )$ and $T_2^2(\Gamma )$.\\

Note that, for every $m\in \mathbb{N}$, $\zeta_m=\sum_{j=0}^m\lambda
_ja_j\in T_c(\Gamma )$. By Proposition \ref{AD} we have that
$$
\Pi _M(\zeta_m)=\sum_{j=0}^m\lambda _j\Pi_M(a_j)\longrightarrow
\Pi_M(k(I-P)P^{[k/2]}(f)(y)),\quad \mbox{ as }m\rightarrow \infty,
$$
in $L^2(\Gamma )$, and in $H^{p(\cdot )}_L(\Gamma )$. Also, for every $x\in \Gamma$,
\begin{align*}
& \sum_{j\in \mathbb{N}}\lambda _j\Pi _M(a_j)(x)
=\Pi _M(k(I-P)P^{[k/2]}(f)(y))(x)\\
& \qquad =\sum_{k=0}^\infty
\frac{c_{k,M+1}}{k+1}(I-P)^MP^{[k/2]}((k+1)(I-P)P^{[(k+1)/2]}(f)(y))(x).
\end{align*}

%In order to prove the last equality we can proceed as follows. Let$g\in T_2^2(\Gamma )$. We define, for every $\ell \in \mathbb{N}$, $g_\ell (x,k)=g(x,k)\chi _{(0,\ell )}(k)$, $x\in \Gamma$ and $k\in \mathbb{N}$. We have that $g_\ell \longrightarrow g$, as $\ell \rightarrow \infty$, in $T_2^2(\Gamma )$. Then, by \red{Proposition \ref{}} $(i)$, $\Pi _M(g_\ell )\longrightarrow \Pi _M(g)$, as $\ell \rightarrow \infty$, in $L^2(\Gamma )$. Hence
%
%\begin{align*}
%\Pi_M(g)(x)&=\lim_{\ell \rightarrow \infty}\Pi _M(g_\ell )(x)=\lim_{\ell \rightarrow \infty}\sum_{k=0}^{\ell -1}\frac{c_{k,M+1}}{k+1}(I-P)^MP^{k-[(k+1)/2]}(g(\cdot , k+1))(x)\\
%&=\sum_{k=0}^\infty \frac{c_{k,M+1}}{k+1}(I-P)^MP^{k-[(k+1)/2]}(g(\cdot , k+1))(x),\quad x\in \Gamma .
%\end{align*}
%
%We conclude that
%$$
%f=\sum_{j\in \mathbb{N}} \lambda _j\Pi _M(a_j),\quad \mbox{ in }L^2(\Gamma ),
%$$
%where
%$$
%\Pi _M(a_j)(x)=\sum_{k=0}^{r_{B_j}-1}\frac{c_{k,M+1}}{k+1}(I-P)^MP^{k-[(k+1)/2]}(a_j(\cdot , k+1))(x),\quad x\in \Gamma ,\;j\in \mathbb{N}.
%$$

We are going to see that there exists $C>0$ such that, for every $j\in \mathbb{N}$, $C\Pi _M(a_j)$ is a $(r,p(\cdot ),M)$-atom.\\

Let $j\in \mathbb{N}$. We write $\Pi _M(a_j)=(I-P)^Mb_j$, where
$$
b_j=\sum_{\ell=0}^{r_{B_j}-1}\frac{c_{\ell,M+1}}{\ell+1}P^{[\ell/2]}(a_j(\cdot , \ell+1)).
$$
Note that, for every $\ell\in \mathbb{N}_+$, $a_j(.,\ell)\in L^r(\Gamma)$,
because $\supp(a_j(.,\ell))$ is finite.

%Indeed, let $k\in \mathbb{N}_+$. Since $a_j\in T_2^r(\Gamma)$, we
%have that
%$$
%\sum_{x\in \Gamma}\Big(\frac{|a_j(x,k)|^2}{k\mu(B(x,k))}\mu(x)\Big)^{r/2}\mu(x)\le \sum_{x\in \Gamma}\Big(\sum_{(y,t)\in \Upsilon(x)}\frac{|a_j(y,t)|^2}{t\mu(B(y,t))}\mu(y)\Big)^{r/2}\mu(x)<\infty.
%$$
%It follows that
%$$
%\frac{\mu(x)}{\mu(B(x,k))}=\frac{\mu(B(x,1))}{\mu(B(x,k))}\ge Ck^{-D},\,\,\,x\in \Gamma.
%$$
%Then,
%$$
%\sum_{x\in \Gamma}|a_j(x,k)|^r\mu(x)<\infty.
%$$

Hence, since $P$ is a contraction in $L^r(\Gamma )$,  $b_j\in L^r(\Gamma )$.
As it was proved in \cite[p. 3463]{BD}, $\supp L^kb_j\subset B(x_{B_j}, M+2r_{B_j})\subset B(x_{B_j},(M+2)r_{B_j})=\mathbb{B}_j$, for every $k=0,...,M$.\\

In order to estimate $\|L^kb_j\|_r$, $k=0,...,M$, we proceed as in
\cite[p. 828-829]{B}. Let $k\in \mathbb{N}$, $0\leq k\leq M$. Assume
that $h\in L^{r'}(\Gamma )$. By taking into account that
$c_{R,M+1}\leq (R+1)^M$, $R \in \mathbb{N}$, the operator $P$ is
selfadjoint, \cite[Proposition 3.2, (a)]{BD}, and H\"older's inequality, we obtain
$$
\Big|\sum_{x\in \Gamma}L^k(b_j)(x)h(x)\mu (x)\Big|\le
Cr_{B_j}^{M-k}\|a_j\|_{T^r_2(\Gamma)}\|S_{k,L}(h)\|_{r'},
$$
%=\Big|\sum_{x\in \Gamma }\sum_{\ell =0}^{r_{B_j}-1}\frac{c_{\ell ,M+1}}{\ell +1}(I-P)^kP^{\ell-[(\ell+1)/2]}(g(\cdot , \ell +1))(x)h(x)\mu (x)\Big|\\
%& \qquad \leq \Big(\sum_{x\in \Gamma} \sum_{\ell =0}^{r_{B_j}-1}\frac{(c_{\ell ,M+1})^2|a_j(x,\ell +1)|^2}{(\ell +1)^{2k}}\mu (x)\Big)^{1/2}\\
%& \qquad \qquad \times \Big(\sum_{x\in \Gamma}\sum_{\ell =0}^{r_{B_j}-1}\frac{|(\ell +1)^k(I-P)^kP^{\ell-[(\ell+1)/2]}(h)(x)|^2}{\ell +1}\mu (x)\Big)^{1/2}\\
%& \qquad \leq Cr_{B_j}^{M-k}\Big(\sum_{x\in \Gamma} \sum_{\ell =0}^{r_{B_j}-1}\frac{|a_j(x,\ell +1)|^2}{\ell +1}\mu (x)\Big)^{1/2}\|h\|_2\\
%& \qquad \leq C(M+2r_{B_j})^{M-k}(V(B_j))^{1/2}\|\chi _{B_j}\|_{p(\cdot )}^{-1}\|h\|_2\\
%& \qquad \leq C(M+2r_{B_j})^{M-k}(V(\mathbb{B}_j))^{1/2}\|\chi _{\mathbb{B}_j}\|_{p(\cdot )}^{-1}\|h\|_2.
%\end{align*}
where, as in (\ref{YY1}) ,
$$
S_{k,L}(g)(x)=\Big(\sum_{(y,t)\in \Upsilon(x)}\frac{|t^k(I-P)^kP^{[(t-1)/2]}(g)(y)|^2}{t\mu(B(x,t))}\mu(y)\Big)^{1/2}.
$$
Since $\mathcal{S}_{k,L}$ defines a bounded operator from
$L^{r'}(\Gamma )$ into itself (\cite[Proposition 4.6]{B}), by using
Lemma \ref{LemaCoc}, (i), we get
$$
\Big|\sum_{x\in \Gamma}L^k(b_j)(x)h(x)\mu (x)\Big|\le
C(r_{\mathbb{B}_j})^{M-k}(\mu(\mathbb{B}_j))^{1/r}\|\chi_{\mathbb{B}_j}\|_{p(\cdot)}^{-1}\|h\|_{r'},
$$
where $r_{\mathbb{B}_j}=(M+2)r_{B_j}$. We deduce that
\begin{equation}\label{ABC}
\|L^k(b_j)\|_r\leq
C(r_{\mathbb{B}_j})^{M-k}(\mu(\mathbb{B}_j))^{1/r}\|\chi
_{\mathbb{B}_j}\|_{p(\cdot )}^{-1}.
\end{equation}
Thus our aim is proved.\\

On the other hand, by using Lemma \ref{LemaSumCoc}, we get
$$
 \mathcal{A}(\{\lambda _j\},\{\mathbb{B}_j\})\leq C\mathcal{A}(\{\lambda _j\},\{B_j\})\leq C\|f\|_{H^{p(\cdot )}_L(\Gamma )}.
$$

Suppose now that $f\in H^{p(\cdot )}_L(\Gamma )$. Since $H^{p(\cdot
)}_L(\Gamma ) \cap L^2(\Gamma )$ is a dense subspace of $H^{p(\cdot
)}_L(\Gamma )$, there exists a sequence $(f_k)_{k\in
\mathbb{N}}\subset H^{p(\cdot )}_L(\Gamma ) \cap L^2(\Gamma )$ such
that $f_0=0$, $f_k\longrightarrow f$, as $k\rightarrow \infty$, in
$H^{p(\cdot )}_L(\Gamma)$, and $\|f_k-f\|_{H^{p(\cdot )}_L(\Gamma
)}\leq 2^{-k}\|f\|_{H^{p(\cdot )}_L(\Gamma )}$, $k\in \mathbb{N}$.
According to the first part of this proof, for every $k\in
\mathbb{N}$,
$$
f_{k+1}-f_{k}=\sum_{j=0}^\infty \lambda _{j,k}a_{j,k},\quad \mbox{
in }L^2(\Gamma), \mbox{ and in }H^{p(\cdot )}_L(\Gamma ),
$$
where, for every $j\in \mathbb{N}$, $\lambda _{j,k}\in \mathbb{C}$ and $a_{j,k}$ is a $( r,p(\cdot ), M)$-atom associated with the ball $\mathbb{B}_{j,k}$, such that
$$
\mathcal{A} (\{\lambda _{j,k}\},\{\mathbb{B}_{j,k}\})\leq
C\|f_k-f_{k+1}\|_{H^{p(\cdot )}_L(\Gamma )}.
$$

We have that
$$
f=\sum_{k=0}^\infty (f_{k+1}-f_k),\quad \mbox{ in }H^{p(\cdot
)}_L(\Gamma ).
$$

Suppose that $\{(j_\ell ,k_\ell )\}_{\ell \in \mathbb{N}}$ is an ordenation of $\mathbb{N}\times \mathbb{N}$. By using Minkowski's inequality we can write
\begin{align*}
& \Big\|\Big(\sum_{\ell \in \mathbb{N}}\Big(\frac{|\lambda _{j_\ell ,k_\ell}|\chi _{\mathbb{B}_{j_\ell ,k_\ell}}}{\|\chi _{\mathbb{B}_{j_\ell ,k_\ell}}\|_{p(\cdot )}}\Big)^\mathfrak{p}\Big)^{1/\mathfrak{p}}\Big\|_{p(\cdot)}
=\Big\|\sum_{\ell \in \mathbb{N}}\Big(\frac{|\lambda _{j_\ell ,k_\ell}|\chi _{\mathbb{B}_{j_\ell ,k_\ell}}}{\|\chi _{\mathbb{B}_{j_\ell ,k_\ell}}\|_{p(\cdot )}}\Big)^\mathfrak{p}\Big\|_{p(\cdot)/\mathfrak{p}}^{1/\mathfrak{p}}\\
& \qquad \leq \Big\|\sum_{k=0}^\infty \sum_{j=0}^\infty
\Big(\frac{|\lambda _{j,k}|\chi _{\mathbb{B}_{j,k}}}{\|\chi
_{\mathbb{B}_{j,k}}\|_{p(\cdot
)}}\Big)^\mathfrak{p}\Big\|_{p(\cdot)/\mathfrak{p}}^{1/\mathfrak{p}}
\leq \Big(\sum_{k=0}^\infty \Big\|\sum_{j=0}^\infty \Big(\frac{|\lambda _{j,k}|\chi _{\mathbb{B}_{j,k}}}{\|\chi _{\mathbb{B}_{j,k}}\|_{p(\cdot )}}\Big)^\mathfrak{p}\Big\|_{p(\cdot)/\mathfrak{p}}\Big)^{1/\mathfrak{p}} \\
& \qquad =\Big(\sum_{k=0}^\infty \Big\|\Big(\sum_{j=0}^\infty
\Big(\frac{|\lambda _{j,k}|\chi _{\mathbb{B}_{j,k}}}{\|\chi
_{\mathbb{B}_{j,k}}\|_{p(\cdot
)}}\Big)^\mathfrak{p}\Big)^{1/\mathfrak{p}}\Big\|_{p(\cdot)}^\mathfrak{p}\Big)^{1/\mathfrak{p}}
 \leq C\|f\|_{H^{p(\cdot )}_L(\Gamma )}\Big(\sum_{k=0}^\infty 2^{-k\mathfrak{p}}\Big)^{1/\mathfrak{p}}
 \leq C\|f\|_{H^{p(\cdot )}_L(\Gamma )}.
\end{align*}

According to Proposition \ref{H6} we deduce that the series
$\sum_{\ell\in \mathbb{N}}\lambda _{j_\ell ,k_\ell }a_{j_\ell , k_\ell
}$ converges in $H^{p(\cdot )}_L(\Gamma )$.

We are going to see that $f=\sum_{\ell\in \mathbb{N}} \lambda _{j_\ell
,k_\ell }a_{j_\ell , k_\ell }$.

Let $\epsilon>0$. There exists $L_0\in \mathbb{N}$ such that if $\ell_0\in \mathbb{N}$, $\ell_0\ge L_0$, then
$$
\Big\|\sum_{\ell =\ell_0}^\infty
\Big(\frac{|\lambda_{j_\ell,k_\ell}|\chi_{\mathbb{B}_{j_\ell,k_\ell}}}{\|\chi_{\mathbb{B}_{j_\ell,k_\ell}}\|_{p(\cdot)}}\Big)^\mathfrak{p}\Big)^{1/\mathfrak{p}}\Big\|_{p(\cdot)}<\epsilon.
$$
Let $\ell_0\in \mathbb{N}$, $\ell_0\ge L_0$. We define
$\hat{k}_0=\max\{k_\ell,\,\,\ell=0,...,\ell_0\}$ and
$J_0=\max\{j_\ell,\,\,\ell=0,...,\ell_0\}$. There exists
$\hat{k}_1\in \mathbb{N}$ such that $\hat{k}_1>\hat{k}_0$ and
$$
\Big\|f-\sum_{k=0}^{\hat{k}_1}(f_k-f_{k+1})\Big\|_{H^{p(\cdot
)}_L(\Gamma )}<\epsilon.
$$
We choose $J_1\in \mathbb{N}$, $J_1>J_0$, and such that
$$
\Big\|f_{k+1}-f_{k}-\sum_{j=1}^{J_1}\lambda_{j,k}a_{j,k}\Big\|_{H^{p(\cdot
)}_L(\Gamma )}<\frac{\epsilon}{\hat{k}_1},
$$
for every $k=0,...,\hat{k}_1$.

Then, by defining $U_0=\{(j,k):j=1,...,J_1,\,k=0,...,\hat{k}_1\}$
and $W_0=\{(j_\ell,k_\ell):\,\,\ell\le \ell_0\}$ we have that
$W_0\subset U_0$ and
\begin{align*}
\Big\|\sum_{k=0}^{\hat{k}_1}\sum_{j=1}^{J_1}\lambda_{j,k}a_{j,k}-\sum_{\ell=0}^{\ell_0}\lambda_{j_\ell,k_\ell}a_{j_\ell,k_\ell}\Big\|_{H^{p(\cdot )}_L(\Gamma )}&=\Big\|\sum_{(j,k)\in U_0\setminus W_0}\lambda_{j,k}a_{j,k}\Big\|_{H^{p(\cdot )}_L(\Gamma )}\\
&\le C\Big\|\Big(\sum_{(j,k)\in U_0\setminus W_0}\Big(\frac{|\lambda_{j,k}|\chi_{\mathbb{B}_{j,k}}}{\|\chi_{\mathbb{B}_{j,k}}\|_{p(\cdot)}}\Big)^\mathfrak{p}\Big)^{1/\mathfrak{p}}\Big\|_{p(\cdot)}\\
&\le
C\Big\|\Big(\sum_{\ell=\ell_0}^\infty\Big(\frac{|\lambda_{j_\ell,k_\ell}|\chi_{\mathbb{B}_{j_\ell,k_\ell}}}{\|\chi_{\mathbb{B}_{j_\ell,k_\ell}}\|_{p(\cdot)}}\Big)^\mathfrak{p}\Big)^{1/\mathfrak{p}}\Big\|_{p(\cdot)}<C\epsilon.
\end{align*}
By combining the above estimates we conclude that
$$
\Big\|f-\sum_{\ell=0}^{\ell_0}\lambda_{j_\ell,k_\ell}a_{j_\ell,k_\ell}\Big\|_{H^{p(\cdot
)}_L(\Gamma )}\le C\epsilon.
$$
Note that the constant $C>0$ does not depend on $\ell_0$.

\end{proof}

As a consequence of Propositions \ref{H6} and \ref{Hardy} we deduce
the following result.

\begin{Cor}  Assume that $p\in \mathcal{P}^{\log}(\Gamma)$, $p_+<2$ and $M\in \mathbb{N}_+$, $M>2D/p_-$.
Then, $H^{p(\cdot )}_L(\Gamma )=H^{p(\cdot )}_{L,M,at}(\Gamma )$ and
the quasinorms are equivalent.
\end{Cor}

We now establish that $H^{p(\cdot )}_L(\Gamma
)=L^{p(\cdot)}(\Gamma)$ provided that $p_->1$. In order to prove
this property we need to see that $S_L$ defines a bounded sublinear
operator from $L^{p(\cdot)}(\Gamma)$ into itself when $p_->1$.

We recall the definitions of Muckenhoupt classes of weights $A_r(\Gamma)$, $1\le r\le\infty$, in our setting. If $1<r<\infty$ we say that a function $w:\,\Gamma\longrightarrow (0,\infty)$ is in $A_r(\Gamma)$ when there exists $C>0$ such that, for every ball $B\subset \Gamma$,
$$
\Big(\frac{1}{\mu(B)}\sum_{x\in B}w(x)\mu(x)\Big)\Big(\frac{1}{\mu(B)}\sum_{x\in B}w(x)^{1-r'}\mu(x)\Big)^{r-1}\le C,
$$
where $r'=\frac{r}{r-1}$ is the exponent conjugated of $r$.

A function $w:\Gamma\longrightarrow (0,\infty)$ is said to be in $A_1(\Gamma)$ when there is a constant $C>0$ such that, for every ball $B\subset \Gamma$,
$$
\frac{1}{\mu(B)}\sum_{x\in B}w(x)\mu(x)\le Cw(y),\,\,\,y\in B.
$$
By $A_\infty(\Gamma)$ is denoted, as usual, the union $\cup_{1\le p<\infty}A_p(\Gamma)$.

If $1\le r< \infty$ and $w:\Gamma\longrightarrow (0,\infty)$,
$L^r(\Gamma,w)$ denotes the weighted $L^p$-space on
$(\Gamma,\mu,d)$.

\begin{Prop}\label{SL}
(i) Let $1<q<\infty$ and $w\in A_q(\Gamma)$. Then, $S_L$ defines a bounded operator from $L^q(\Gamma,w)$ into itself.

(ii)  Let $p\in \mathcal{P}^{\log}(\Gamma)$ such that $p_->1$. Then, $S_L$ defines a bounded operator from $L^{p(\cdot)}(\Gamma)$ into itself.
\end{Prop}

\begin{proof} (i) We adapt to our context some arguments developed in \cite[Sections 3.1, 3.2, and 4.1.1]{MP}. We divide the proof in four steps and we sketch the proof of each of them.

$\mathbf{Step\,\, 1}$. For every $\beta\ge 1$ we define the operator $\mathcal{A}^\beta$ as follows
$$
\mathcal{A}^\beta(F)(x)=\Big(\sum_{(y,k)\in
\Upsilon_\beta(x)}\frac{|F(y,k)|^2}{k\mu(B(y,k))}\mu(y)\Big)^{1/2},\,\,\,x\in
\Gamma,
$$
where $F:\Gamma\times \mathbb{N}_+\longrightarrow \mathbb{C}$.

Let $\beta\ge 1$, $1<s<\infty$ and $w\in A_s(\Gamma)$. There exists $C>0$ such that
\begin{equation}\label{ZZ1}
\|\mathcal{A}^\beta(F)\|_{L^s(\Gamma,w)}\le C\|\mathcal{A}(F)\|_{L^s(\Gamma,w)}.
\end{equation}
 Note that (\ref{ZZ1}) is non trivial when $\mathcal{A}(F)\in L^s(\Gamma,w)$.

To prove (\ref{ZZ1}) we can show firstly that (\ref{ZZ1}) holds when $s=2$ and $w\in A_r(\Gamma)$, $1<r<\infty$, and then to argue by using extrapolation (see the proof of \cite[Proposition 3.2]{MP}).

$\mathbf{Step \,\,2}$. We define, for every $r\in (0,\infty)$, the maximal operator $C_r$ by
$$
C_r(F)(x)=\sup_{x\in B=B(x_B,r_B)}\Big(\frac{1}{\mu(B)}\sum_{z\in B}\Big(\sum_{k=1}^{r_B}\sum_{d(y,z)<k}\frac{|F(y,k)|^2}{k\mu(B(y,k))}\mu(y)\Big)^{r/2}\mu(z)\Big)^{1/r},\,\,\,x\in \Gamma.
$$
For every $0<r<\infty$, $1<s<\infty$ and $w\in A_s(\Gamma)$, we have that there exists $C>0$ for which
\begin{equation}\label{ZZ2}
\|\mathcal{A}(F)\|_{L^s(\Gamma,w)}\le C\|C_r(F)\|_{L^s(\Gamma,w)},
\end{equation}
when $F:\Gamma\times \mathbb{N}_+\longrightarrow \mathbb{C}$. The proof of (\ref{ZZ2}) can be completed by following some ideas developed in the proof of \cite[Proposition 3.34]{MP}. We need to use a Whitney covering theorem in our setting that can be find in \cite[Lemma 2.9]{MS} to establish (\ref{ZZ2}) for every complex function $F$ with compact (finite) support in $\Gamma\times\mathbb{N}_+$ by employing a good $\lambda$ argument involving $\mathcal{A}^\beta$ for certain $\beta>1$ and using (\ref{ZZ1}). Then, (\ref{ZZ2}) can be proved for general complex valued functions $F$ on $\Gamma\times \mathbb{N}_+$ by approximation.

$\mathbf{Step \,\,3}$. We consider, for every $r\in (0,\infty)$, the maximal operator $\mathfrak{C}_r$ given by
$$
\mathfrak{C}_r(f)(x)=\sup_{x\in
B=B(x_B,r_B)}\Big(\frac{1}{\mu(B)}\sum_{z\in
B}\Big(\sum_{k=1}^{r_B}\sum_{d(y,z)<k}\frac{|k(I-P)P^{[k/2]}(f)(y)|^2}{k\mu(B(y,k))}\mu(y)\Big)^{r/2}\mu(z)\Big)^{1/r},
$$
for every $x\in \Gamma$, where $f:\Gamma\longrightarrow \mathbb{C}$.

For every $1<r<\infty$,
\begin{equation}\label{ZZ3}
\mathfrak{C}_r(f)(x)\le C(\mathcal{M}(|f|^r)(x))^{1/r},\,\,\,x\in \Gamma.
\end{equation}
We can prove (\ref{ZZ3}) arguing as in the proof of \cite[Proposition 4.1]{MP}. We need to make some modifications in our setting. Let $1<r<\infty$, $x\in \Gamma$ and $f:\Gamma\longrightarrow \mathbb{C}$. We take a ball $B=B(x_B,r_B)$ such that $x\in B$. We decompose $f$ in the following way $f=f_1+f_2$, where $f_1=f\chi_{B(x_B,4r_B)}$.

According to \cite[Proposition 4.6, (ii)]{B} $S_L$ is bounded from $L^r(\Gamma)$ into itself. Then, we can write
\begin{align*}
\Big(\frac{1}{\mu(B)}&\sum_{z\in B}\Big(\sum_{k=1}^{r_B}\sum_{d(y,z)<k}\frac{|k(I-P)P^{[k/2]}(f_1)(y)|^2}{k\mu(B(y,k))}\mu(y)\Big)^{r/2}\mu(z)\Big)^{1/r}\\
&\le \Big(\frac{1}{\mu(B)}\sum_{z\in \Gamma}|S_L(f_1)(z)|^r\mu(z)\Big)^{1/r}\\
&\le C\Big(\frac{1}{\mu(B)}\sum_{z\in \Gamma}|f_1(z)|^r\mu(z)\Big)^{1/r}\\
&\le C\Big(\frac{1}{\mu(B(x_B,4r_B))}\sum_{z\in B(x_B,4r_B)}|f(z)|^r\mu(z)\Big)^{1/r}\\
&\le C\Big(\mathcal{M}(|f|^r)(x)\Big)^{1/r}.
\end{align*}
We have taken into account that $\mu$ is doubling.

We define, for every $j\in \mathbb{N}$, $\mathcal{Z}_j(B)=B(x_B,2^{j+1}r_B)\setminus B(x_B,2^jr_B)$. By taking into account that $(I-P)f_2(y)=0$ when $y\in B(x_B,2r_B)$, and using (\ref{compuesto}) we get
\begin{align*}
\Big(\frac{1}{\mu(B)}&\sum_{z\in B}\Big(\sum_{k=1}^{r_B}\sum_{d(y,z)<k}\frac{|k(I-P)P^{[k/2]}(f_2)(y)|^2}{k\mu(B(y,k))}\mu(y)\Big)^{r/2}\mu(z)\Big)^{1/r}\\
&\le \Big(\sum_{k=1}^{r_B}\sum_{d(y,x_B)<2r_B}\frac{|k(I-P)P^{[k/2]}(f_2)(y)|^2}{k\mu(B(y,k))}\mu(y)\Big)^{1/2}\\
&\le \sum_{j=2}^\infty \Big(\sum_{k=2}^{r_B}\sum_{d(y,x_B)<2r_B}\frac{|k(I-P)P^{[k/2]}(f_2\chi_{\mathcal{Z}_j(B)})(y)|^2}{k\mu(B(y,k))}\mu(y)\Big)^{1/2}\\
&\le \sum_{j=2}^\infty \Big(\sum_{k=2}^{r_B}\sum_{d(y,x_B)<2r_B}\frac{\mu(y)}{k\mu(B(y,k))}\Big(\sum_{z\in \mathcal{Z}_j(B)}\frac{\mu(z)}{\mu(B(y,\sqrt{k}))}e^{-cd(y,z)^2/k}|f(z)|\Big)^2\Big)^{1/2}\\
&\le \sum_{j=2}^\infty \Big(\sum_{k=1}^{r_B}\sum_{d(y,x_B)<2r_B}\frac{\mu(y)e^{-c(2^jr_B)^2/k}}{k\mu(B(y,k))\mu(B(y,\sqrt{k}))^2}\\
&\hspace{2cm}\times\Big(\sum_{z\in \mathcal{Z}_j(B)}|f(z)|^r\mu(z)\Big)^{2/r}(\mu(\mathcal{Z}_j(B)))^{2(1-1/r)}\Big)^{1/2}\\
&\le \sum_{j=2}^\infty \Big(\frac{1}{\mu(B(x_B,2^{j+1}r_B))}\sum_{z\in B(x_B,2^{j+1}r_B)}|f(z)|^r\mu(z)\Big)^{1/r}\\
&\hspace{2cm}\times\Big(\sum_{k=1}^{r_B}\sum_{d(y,x_B)<2r_B}\frac{\mu(y)e^{-c(2^jr_B)^2/k}}{k}\frac{\mu(B(x_B,2^{j+1}r_B))^2}{\mu(B(y,k))\mu(B(y,\sqrt{k}))^2}\Big)^{1/2}.\\
\end{align*}
Since $B(x_B,r_B)\subset B(y,3r_B)$, when $y\in B(x_B,2r_B)$, and $\mu$ is doubling it follows that
\begin{align*}
\frac{\mu(B(x_B,2^{j+1}r_B))^2}{\mu(B(y,k))\mu(B(y,\sqrt{k}))^2}&\le C\frac{\mu(B(x_B,2^{j+1}r_B))^2\mu(B(y,r_B))^3}{\mu(B(y,k))\mu(B(y,\sqrt{k}))^2\mu(B(x_B,r_B))^3}\\
&\le C \Big(\frac{r_B^3}{k^{2}}\Big)^D \frac{2^{(j+1)2D}}{\mu(B(x_B,r_B))},
\end{align*}
for every $y\in B(x_B,2r_B)$, $j,k\in \mathbb{N}_+$, $k\le r_B$ and $j\ge 2$.

Hence, we obtain
\begin{align*}
\Big(&\frac{1}{\mu(B)}\sum_{z\in B}\Big(\sum_{k=1}^{r_B}\sum_{d(y,z)<k}\frac{|k(I-P)P^{[k/2]}(f_2)(y)|^2}{k\mu(B(y,k))}\mu(y)\Big)^{r/2}\mu(z)\Big)^{1/r}\\
&\le \sum_{j=2}^\infty \Big(\frac{1}{\mu(B(x_B,2^{j+1}r_B))}\sum_{z\in B(x_B,2^{j+1}r_B)}|f(z)|^r\mu(z)\Big)^{1/r}\\
&\hspace{1cm}\times\Big(\sum_{k=1}^{r_B}\frac{e^{-c(2^jr_B)^2/k}}{k}\Big(\frac{2^{j}r_B}{\sqrt{k}}\Big)^{2D}\Big(\frac{r_B}{k}\Big)^D\frac{1}{\mu(B(x_B,r_B))}\sum_{d(y,x_B)<2r_B}\mu(y)\Big)^{1/2}\\
&\le C\sum_{j=2}^\infty e^{-c2^{2j}r_B}r_B^{D+1}\Big(\mathcal{M}(|f|^r)(x)\Big)^{1/r}\\
&\le C\Big(\mathcal{M}(|f|^r)(x)\Big)^{1/r}.
\end{align*}
By taking the supremum over all the balls $B$ such that $x\in B$ we establish (\ref{ZZ3}).

$\mathbf{Step\,\,4}$. We now combine (\ref{ZZ2}) and (\ref{ZZ3}) to prove that $S_L$ is bounded from $L^q(\Gamma,w)$ into itself.

We choose $1<r<q$ such that $w\in A_{q/r}(\Gamma)$. Then we can write
\begin{align*}
\|S_L(f)\|_{L^q(\Gamma,w)}&\le C\|\mathfrak{C}_r(f)\|_{L^q(\Gamma,w)}\\
&\le C\Big\|\Big(\mathcal{M}(|f|^r)\Big)^{1/r}\Big\|_{L^q(\Gamma,w)}\\
&\le C\|f\|_{L^q(\Gamma,w)},\,\,\,f\in L^q(\Gamma,w).
\end{align*}
We have applied that $\mathcal{M}$ is bounded from $L^s(\Gamma,v)$ into itself, for every $s\in (1,\infty)$ and $v\in A_s(\Gamma)$.

(ii) This assertion can be proved by using (i) and extrapolation arguments as in \cite[Theorem 1.3]{CFMP}.
\end{proof}

Note that Proposition \ref{SL} is an extension of \cite[Proposition 4.6, (ii)]{B}.

\begin{Prop}\label{Equal}
Let $p\in \mathcal{P}^{\log}(\Gamma)$ such that $p_->1$. Then, $H^{p(\cdot)}_L(\Gamma)=L^{p(\cdot)}(\Gamma)$, algebraic and topologically.
\end{Prop}

\begin{proof} Suppose that $f\in L^2(\Gamma)\cap H^{p(\cdot)}_L(\Gamma)$. According to Proposition \ref{Hardy} by taking $r\ge 2$, $r>p_+$, $M\in \mathbb{N}_+$, and $>2D/p_-$, there exist, for every $j\in \mathbb{N}$, $\lambda_j\in \mathbb{C}$ and a $(r,p(\cdot),M)$-atom $a_j$ associated with the ball $B_j$ such that $f=\sum_{j=0}^\infty \lambda_ja_j$, where the series converges in $H^{p(\cdot)}_L(\Gamma)$ and pointwisely, and $\mathcal{A}(\{\lambda_j\},\{B_j\})\le C\|f\|_{H^{p(\cdot)}_L(\Gamma)}$. By using Lemma \ref{LemaSum} we obtain, for every $j_1,j_2\in \mathbb{N}$ and $j_1<j_2$,
$$
\Big\|\sum_{j=j_1}^{j_2}|\lambda_j||a_j|\Big\|_{p(\cdot)}\le C\Big\|\sum_{j=j_1}^{j_2}\frac{|\lambda_j|\chi_{B_j}}{\|\chi_{B_j}\|_{p(\cdot)}}\Big\|_{p(\cdot)}.
$$
Note that in this case $\mathfrak{p}=1$. Since the series
$$
\sum_{j=0}^{\infty}\frac{|\lambda_j|\chi_{B_j}}{\|\chi_{B_j}\|_{p(\cdot)}}
$$
is convergent in $L^{p(\cdot)}(\Gamma)$, the series $\sum_{j=0}^\infty \lambda_j a_j$ also converges in $L^{p(\cdot)}(\Gamma)$ and
$$
\Big\|\sum_{j=0}^\infty \lambda_j a_j\Big\|_{p(\cdot)}\le C\|f\|_{H^{p(\cdot)}_L(\Gamma)}.
$$
So we have that $f=\sum_{j=0}^\infty \lambda_j a_j$ in the sense of convergence in $L^{p(\cdot)}(\Gamma)$.

On the other hand, if $f\in L^2(\Gamma)\cap L^{p(\cdot)}(\Gamma)$, from Proposition \ref{SL} we deduce that $S_L(f)\in L^{p(\cdot)}(\Gamma)$ and $\|S_L(f)\|_{p(\cdot)}\le C\|f\|_{p(\cdot)}$, that is, $f\in L^2(\Gamma)\cap H^{p(\cdot)}_L(\Gamma)$ and $\|f\|_{H^{p(\cdot)}_L(\Gamma)}\le C\|f\|_{p(\cdot)}$.

By taking closures we conclude that $H^{p(\cdot)}_L(\Gamma)=L^{p(\cdot)}(\Gamma)$ and the two norms are equivalent.
\end{proof}

We recall the definition of molecules. Let $p\in \mathcal{P}^{\log}(\Gamma)$ and $M\in \mathbb{N}_+$, $1<q<\infty$, and $\varepsilon >0$. We say that a function $m\in L^q(\Gamma )$ is a $(q,p(\cdot ),M,\varepsilon )$-molecule when
there exists a function $b\in L^q(\Gamma )$ and a ball $B=B(x_B,r_B)$, with $x_B\in \Gamma$ and  $r_B\geq 1$, such that $m=L^Mb$ and, for every $k=0,...,M$,
$$
\|L^kb\|_{L^q(\mathfrak{S}_j(B))}\leq
(r_B)^{M-k}2^{-j\varepsilon}(\mu(B(x_B,2^jr_B)))^{1/q}\|\chi
_{B(x_B,2^jr_B)}\|_{p(\cdot )}^{-1},\quad j\in \mathbb{N},
$$
where $\mathfrak{S}_j(B)=B(x_B,2^{j+1}r_B)\setminus B(x_B,2^{j-1}r_B)$, $j\in \mathbb{N}_+$ and $\mathfrak{S}_0(B)=B$.\\

Suppose that $M\in \mathbb{N}$, $1<q<\infty$ and $a$ is a $(q,p(\cdot ),M)$-atom associated with the ball $B$ and
the function $b\in L^q(\Gamma )$. Let $k\in \mathbb{N}$, $k=0,...,M$. Since $\supp L^kb\subset B$, $\|L^kb\|_{L^q(\mathfrak{S}_j(B))}=0$, $j\in \mathbb{N}_+$. Also, $\|L^kb\|_{L^q(\mathfrak{S}_0(B))}=\|L^kb\|_q\leq r_B^{M-k}(\mu(B))^{1/q}\|\chi _B\|_{p(\cdot )}^{-1}$. Hence, $a$ is a $(q,p(\cdot ),M, \varepsilon)$-molecule, for every $\varepsilon >0$.\\

Also, there exists $C>0$ such that $C\Pi_M(a)$ is a $(q,p(\cdot),M,\varepsilon)$-molecule provided that $a$ is a $(T_2^{p(\cdot)},q)$-atom, $2\le q<\infty$, $\varepsilon>0$, and $M\in \mathbb{N}_+$. Indeed, assume that $2\le q<\infty$, $\varepsilon>0$, $M\in \mathbb{N}_+$, and $a$ is a $(T_2^{p(\cdot)},q)$-atom associated with the ball $B=B(x_B,r_B)$, with $x_B\in \Gamma$ and $r_B\ge 1$. We define $\alpha=\Pi_M(a)$. We can write $\alpha=L^Mb$, where
$$
b=\sum_{k=0}^{r_B-1}\frac{c_{k,M+1}}{k+1}P^{[k/2]}(a(.,k+1)).
$$
Assume that $j_0\in \mathbb{N}_+$ such that $2^{j_0-1}>M+1$. As it was established in (\ref{ABC}) we have that, for every $\ell\in \mathbb{N}$, $0\le \ell\le M$,
$$
\|L^\ell(b)\|_{q}\le Cr_B^{M-\ell}\mu(B)^{1/q}\|\chi_{B}\|_{p(\cdot)}^{-1}.
$$
Hence, according to Lemma \ref{LemaCoc}, (i), there exists $C>0$ such that
$$
\|L^\ell(b)\|_{L^q(\mathfrak{S}_0(B))}\le Cr_B^{M-\ell}2^{-j\varepsilon}\mu(B(x_B,2^jr_B))^{1/q}\|\chi_{B(x_B,2^jr_B)}\|_{p(\cdot)}^{-1},
$$
for every $\ell\in \mathbb{N}$, $0\le \ell\le M$, $j\in \mathbb{N}$ and $j\le j_0$.

Suppose now $\ell\in \mathbb{N}$, $0\le \ell\le M$, and $j\in \mathbb{N}$, $j>j_0$. Let $h\in L^{q'}(\Gamma)\cap L^2(\Gamma)$ such that $\supp(h)\subset \mathfrak{S}_j(B)$. By proceeding as in \cite[Proposition 3.2, (a)]{B} we obtain
\begin{align*}
\Big|\sum_{x\in \mathfrak{S}_j(B)}L^\ell(b)(x)h(x)\mu(x)\Big|&=\Big|\sum_{x\in \mathfrak{S}_j(B)}h(x)\sum_{k=0}^{r_B-1}\frac{c_{k,M+1}}{k+1}L^\ell P^{[k/2]}(a(.,k+1))(x)\mu(x)\Big|\\
&\le \sum_{x\in \Gamma}\sum_{k=0}^{r_B-1}\frac{c_{k,M+1}}{(k+1)^\ell} \frac{|a(x,k+1)||(k+1)^{\ell}L^\ell P^{[k/2]}(h)(x)|}{k+1}\mu(x)\\
&\le Cr_B^{M-\ell}\sum_{(x,k+1)\in T(B)}\frac{|a(x,k+1)||(k+1)^{\ell}L^\ell P^{[k/2]}(h)(x)|}{k+1}\mu(x)\\
&\le Cr_B^{M-\ell}\sum_{x\in \Gamma}\mathcal{A}(a)(x)\mathcal{S}_{\ell,B}(h)(x)\mu(x),
\end{align*}
where, for every $x\in \Gamma$,
$$
\mathcal{S}_{\ell,B}(h)(x)=\Big(\sum_{(y,k)\in \Upsilon(x)}\frac{\big(|k^{\ell}L^\ell P^{[(k-1)/2]}(h)(y)|\big)^2}{k\mu(B(x,k))}\chi_{T(B)}(y,k)\mu(y)\Big)^{1/2}.
$$
Then,
$$
\Big|\sum_{x\in \mathfrak{S}_j(B)}L^\ell(b)(x)h(x)\mu(x)\Big|\le C r_B^{M-\ell}\|a\|_{T_2^q(\Gamma)}\|\mathcal{S}_{\ell,B}(h)\|_{q'},
$$
being $q'=q/(q-1)$.

Note that $d(x,x_B)<r_B$ provided that $d(x,y)<k$ and $(y,k)\in T(B)$. Also, since $\supp(h)\subset B(x_B,2^{j+1}r_B)\setminus B(x_B,2^{j-1}r_B)$, we have that
\begin{align*}
\supp(L^\ell(h))&\subset B(x_B,2^{j+1}r_B+\ell)\setminus B(x_B,2^{j-1}r_B-\ell)\\
&\subset B(x_B,(2^{j+1}+M)r_B)\setminus B(x_B,(2^{j-1}-M)r_B)\subset B^c.
\end{align*}
 Hence $L^\ell(h)(y)=0$, $y\in B$. Since $q\ge 2$, Jensen's inequality leads to
\begin{align*}
\|&\mathcal{S}_{\ell,B}(h)\|_{q'}\\
&=\Big(\sum_{x\in \Gamma}\Big(\sum_{k=1}^\infty \sum_{d(y,x)<k}\frac{|k^{\ell}L^\ell P^{[(k-1)/2]}(h)(y)|^2}{k\mu(B(x,k))}\chi_{T(B)}(y,k)\mu(y)\Big)^{q'/2}\mu(x)\Big)^{1/q'}\\
&\le \Big(\sum_{x\in B}\Big(\sum_{k=3}^{r_B} \sum_{y\in B(x,k)\cap B}\frac{|k^{\ell}L^\ell P^{[(k-1)/2]}(h)(y)|^2}{k\mu(B(x,k))}\mu(y)\Big)^{q'/2}\mu(x)\Big)^{1/q'}\\
&\le \mu(B)^{1/q'-1/2}\Big(\sum_{x\in B}\sum_{k=3}^{r_B} \sum_{y\in B(x,k)\cap B}\frac{|k^{\ell}L^\ell P^{[(k-1)/2]}(h)(x)|^2}{k\mu(B(x,k))}\mu(y)\mu(x)\Big)^{1/2}\\
&\le \mu(B)^{1/q'-1/2}\Big(\sum_{k=3}^{r_B} \sum_{y\in B}\frac{|k^{\ell}L^\ell P^{[(k-1)/2]}(h)(y)|^2}{k}\mu(y)\Big)^{1/2}\\
&\le C\mu(B)^{1/q'-1/2}\Big(\sum_{k=1}^{r_B} \sum_{y\in B}\frac{1}{k}\Big(\sum_{z\in \mathfrak{S}_j(B)}e^{-cd(y,z)^2/k}|h(z)|\frac{\mu(z)}{\mu(B(y,\sqrt{k}))}\Big)^2\mu(y)\Big)^{1/2}\\
&\le C\mu(B)^{1/q'-1/2}\|h\|_{q'}\Big(\sum_{k=1}^{r_B} \sum_{y\in B}\frac{1}{k}e^{-c(2^jr_B)^2/k}\frac{\mu(\mathfrak{S}_j(B))^{2/q}}{\mu(B(y,\sqrt{k}))^2}\mu(y)\Big)^{1/2}\\
&\le C\mu(B)^{1/q'-1/2}\|h\|_{q'}\Big(\sum_{k=1}^{r_B} \sum_{y\in B}\frac{e^{-c(2^jr_B)^2/k}}{k}\\
&\hspace{5mm}\times\Big(\frac{\mu(B(y,2^jr_B))}{\mu(B(y,\sqrt{k}))}\Big)^2\frac{\mu(B(x_B,2^{j+1}r_B))^{2/q}}{\mu(B(y,2^jr_B))^2}\mu(y)\Big)^{1/2}\\
&\le C\mu(B)^{1/q'-1/2}\|h\|_{q'}\Big(\sum_{k=1}^{r_B} \frac{1}{k}\Big(\frac{2^jr_B}{\sqrt{k}}\Big)^{2D}e^{-c(2^jr_B)^2/k}\mu(B(x_B,2^jr_B))^{-2/q'}\mu(B)\Big)^{1/2}\\
&\le C\mu(B)^{1/q'}\|h\|_{q'}r_B^{1/2}e^{-c2^{2j}r_B}\mu(B(x_B,2^jr_B))^{-1/q'}\\
&\le C\|h\|_{q'}e^{-c2^{2j}}\mu(B)^{1/q'}\mu(B(x_B,2^jr_B))^{-1/q'}.
\end{align*}
We deduce that
$$
\Big|\sum_{x\in \mathfrak{S}_j(B)}L^\ell(b)(x)h(x)\mu(x)\Big|\le C r_B^{M-\ell}e^{-c2^{2j}}\|\chi_B\|_{p(\cdot)}^{-1}\mu(B)\mu(B(x_B,2^jr_B))^{-1/q'}\|h\|_{q'}.
$$
By Lemma \ref{LemaCoc}, (i), it follows that
$$
\|L^\ell(b)\|_{L^q(\mathfrak{S}_j(B))}\le Cr_B^{M-\ell}e^{-c2^{2j}}\|\chi_{B(x_B,2^jr_B)}\|_{p(\cdot)}^{-1}\mu(B(x_B,2^jr_B))^{1/q}.
$$
Note that $C$ does not depend on $b$ and $j$. Thus, we have shown
that $\alpha$ is a $(q,p(\cdot),M,\varepsilon)$-molecule.

Assume now that $q\in [1,\infty)\cap(p_+,\infty)$ and $m$ is a
$(q,p(\cdot ),M,\varepsilon )$-molecule associated with the ball
$B=B(x_B,r_B)$ and the function $b$ as in the definition. By using
Lemma \ref{LemaCoc}, (ii), we deduce, for every
$k=0,...,M$, that
\begin{align}\label{PMol}
\|L^kb\|_q^q&=\sum_{j=0}^\infty
\|L^kb\chi_{\mathfrak{S}_j(B)}\|_q^q\nonumber\\
&\le (r_B)^{(M-k)q}\sum_{j=0}^\infty2^{-j\varepsilon
q}\mu(B(x_B,2^jr_B))\|\chi
_{B(x_B,2^jr_B)}\|_{p(\cdot )}^{-q}\nonumber\\
&\le (r_B)^{(M-k)q}\mu(B)\|\chi _{B}\|_{p(\cdot
)}^{-q}\sum_{j=0}^\infty2^{-j\varepsilon q}\nonumber\\
&\le C(r_B)^{(M-k)q}\mu(B)\|\chi _{B}\|_{p(\cdot
)}^{-q}.
\end{align}

%The following result is useful to study the boundedness of operators in $H^{p(\cdot )}(\Gamma )$.

We are going to prove Theorem \ref{Th1.4}.

\begin{proof}[Proof of Theorem \ref{Th1.4}.]
In order to establish this result we proceed as in the proof of Proposition \ref{H6}. Suppose that, for every $j\in \mathbb{N}$,  $\lambda_j\in \mathbb{C}$ and $m_j$ is a $(q, p(\cdot ), M,\varepsilon )$-molecule associated to the function $b_j$ and the ball $B_j=B(x_{B_j},r_{B_j})$, with $x_{B_j}\in \Gamma$ and $r_{B_j}\ge 1$, such that $\mathcal{A}(\{\lambda_j\},\{B_j\})<\infty$. \\

We are going to see that, for every $\epsilon >0$, there exists $j_0\in \mathbb{N}$ such that, for every $j_1, j_2\in \mathbb{N}$, $j_0\leq j_1 <j_2$,
$$
\Big\|S_L\Big(\sum_{j=j_1}^{j_2}\lambda _j
m_j\Big)\Big\|_{p(\cdot)}<\epsilon.
$$
Let $j_1, j_2\in \mathbb{N}$, $j_1 <j_2$. We have that
$$
S_L\Big(\sum_{j=j_1}^{j_2}\lambda _j m_j\Big)\leq \sum_{j=j_1}^{j_2}|\lambda _j|S_L(m_j).
$$

If $B=B(x_B,r_B)$ is a ball, we recall that
$\mathcal{S}_0(B)=B(x_B,8Mr_B)$ and, for every $i\in \mathbb{N}_+$,
$\mathcal{S}_i(B)=B(x_B,2^{i+3}Mr_B)\setminus B(x_B,2^{i+2}Mr_B)$.

%Also, if $B=B(x_B,r_B)$ is a ball, we define
%$\mathcal{Z}_0(B)=B(x_B,r_B)$ and, for every $i\in \mathbb{N}_+$,
%$\mathcal{Z}_i(B)=B(x_B,2^{i+1}r_B)\setminus B(x_B,2^{i}r_B)$.
%
%Let $j\in \mathbb{N}$. We can write
%$$
%m_j=\sum_{\ell=0}^\infty m_j\chi_{\mathcal{Z}_\ell(B_j)},
%$$
%
%
%% We have that
%%$$
%%S_L(\sum_{j=j_1}^{j_2}\lambda _j m_j)\leq \sum_{j=1}^\infty |\lambda _j|S_L(m_j),
%%$$
%
%If $B=B(x_B,r_B)$ is a ball, we recall that
%$\mathcal{S}_0(B)=B(x_B,8Mr_B)$ and, for every $i\in \mathbb{N}_+$,
%$\mathcal{S}_i(B)=B(x_B,2^{i+3}Mr_B)\setminus B(x_B,2^{i+2}Mr_B)$. Also, if $B=B(x_B,r_B)$ is a ball, we define
%$\mathcal{Z}_0(B)=B(x_B,r_B)$ and, for every $i\in \mathbb{N}_+$,
%$\mathcal{Z}_i(B)=B(x_B,2^{i+1}r_B)\setminus B(x_B,2^{i}r_B)$.
%
%For every $j\in \mathbb{N}$, we can write
%$$
%b_j=\sum_{\ell=0}^\infty b_j\chi_{\mathcal{Z}_\ell(B_j)},
%$$
%We obtain
%%\begin{align*}
%%S_L\Big(\sum_{j=j_1}^{j_2}\lambda _j m_j\Big)&\le \sum_{i=0}^\infty\sum_{j=j_1}^{j_2}|\lamba_j|S_L(m_j)\\
%%&\leq
%%\sum_{\ell=0}^\infty\sum_{j=j_1}^{j_2} |\lambda
%%_j|S_L(L^{M}(b_j\chi_{\mathcal{Z}_\ell(B_j)})).
%%\end{align*}
Then, we get
\begin{equation}\label{ForSum}
\Big\|S_L\Big(\sum_{j=j_1}^{j_2}\lambda _j m_j\Big)\Big\|_{p(\cdot
)}\leq \Big(\sum_{i=0}^\infty \Big\|\Big(\sum_{j=j_1}^{j_2}
(|\lambda _j|S_L(m_j
)\chi_{\mathcal{S}_i(B_j)})^\mathfrak{p}\Big)^{1/\mathfrak{p}}\Big\|_{p(\cdot
)}^\mathfrak{p}\Big)^{1/\mathfrak{p}}.
\end{equation}

 Since $S_L$ is bounded from $L^q(\Gamma )$ into itself
(\cite[p. Proposition 4.6, (ii)]{B}), by (\ref{PMol}) we get,
\begin{align*}
\|S_L(m_j))\|_{L^q(\mathcal{S}_0(B_j))}&\leq C\|m_j\|_q \leq C
\mu(B_j)^{1/q}\|\chi _{B_j}\|_{p(\cdot )}^{-1},\quad j\in
\mathbb{N}.
%&\le C 2^{-\ell(\varepsilon-D(1+1/p_-) }\mu( \mathcal{S}_0(B_j))^{1/q}\|\chi _{\mathcal{S}_0(B_j)}\|_{p(\cdot )}^{-1} ,\quad j\in \mathbb{N}.
\end{align*}

By using Lemmas \ref{LemaSum}, \ref{LemaCoc}, (i), and \ref{LemaSumCoc} we
obtain
\begin{align*}
\Big\|\Big(\sum_{j=j_1}^{j_2} \Big(|\lambda _j|S_L(m_j)\chi
_{\mathcal{S}_0(B_j)}\Big)^{\mathfrak{p}}
\Big)^{1/\mathfrak{p}}\Big\|_{p(\cdot )}&\leq
C\Big\|\Big(\sum_{j=j_1}^{j_2}\Big(\frac{|\lambda_j|\chi_{B_j}}{\|\chi_{B_j}\|_{p(\cdot)}}\Big)^{\mathfrak{p}}\Big)^{1/\mathfrak{p}}\Big\|_{p(\cdot)},
\end{align*}
because $q\ge 2$ and $q>p_+$.

We are going to estimate
\begin{align*}
\Big\|\Big(\sum_{j=j_1}^{j_2} \Big(|\lambda _j|S_L(m_j)\chi
_{\mathcal{S}_i(B_j)}\Big)^{\mathfrak{p}}
\Big)^{1/\mathfrak{p}}\Big\|_{p(\cdot )}
\end{align*}
for every $i\in \mathbb{N}_+$.

We decompose, for every $j\in \mathbb{N}$, $S_L(m_j)$ as follows
\begin{align*}
S_L(m_{j})(x)&\leq
 \Big(\sum_{k=1}^{d(x,x_{B_j})/2}\sum_{d(x,y)<k}\frac{|k(I-P)P^{[k/2]}m_{j}(y)|^2}{k\mu(B(y,k))}\mu (y)\Big)^{1/2}\\
&\qquad + \Big(\sum_{k=d(x_{B_j},x)/2}^{\infty}\sum_{d(x,y)<k}\frac{|k(I-P)P^{[k/2]}m_{j}(y)|^2}{k\mu(B(y,k))}\mu (y)\Big)^{1/2}\\
&=I_{1,j }(x)+I_{2,j}(x),\quad x\in \Gamma.
\end{align*}

Let $j\in \mathbb{N}$ and $i\in \mathbb{N}_+$. Since $q\ge 2$ and $q>p_+$, by
(\ref{PMol}) and (\ref{compuesto}), we obtain
\begin{align*}
\sum_{x\in S_i(B_j)}&I_{2,j}(x)^{q}\mu(x)=\sum_{x\in
S_i(B_j)}\Big(\sum_{k>d(x_{B_j},x)/2}\sum_{d(x,y)<k}\frac{|k(I-P)P^{[k/2]}m_{j}(y)|^2}{k\mu(B(y,k))}\mu
(y)\Big)^{q/2}\mu(x)\\
&=\sum_{x\in
S_i(B_j)}\Big(\sum_{k>d(x_{B_j},x)/2}\sum_{d(x,y)<k}\frac{|k^{M+1}(I-P)^{M+1}P^{[k/2]}(b_{j})(y)|^2}{k^{2M+1}\mu(B(y,k))}\mu
(y)\Big)^{q/2}\mu(x)\\
&\le C\sum_{x\in
S_i(B_j)}\Big(\sum_{k>d(x_{B_j},x)/2}\sum_{d(x,y)<k}\frac{1}{k^{2M+1}\mu(B(x,k))}\\
&\hspace{3cm}\times\Big(\sum_{z\in
\Gamma}\frac{e^{-cd(y,z)^2/k}}{\mu(B(y,\sqrt{k}))}|b_j(z)|\mu(z)\Big)^2\mu(y)\Big)^{q/2}\mu(x)\\
&\le C\sum_{x\in
S_i(B_j)}\Big(\sum_{k>d(x_{B_j},x)/2}\frac{1}{k^{2M+1}\mu(B(x,k))^{2/q}}\|b_j\|_q^2\Big)^{q/2}\mu(x)\\
&\le C\sum_{x\in
S_i(B_j)}\frac{1}{d(x,x_{B_j})^{Mq}\mu(B(x,d(x,x_{B_j})))}\mu(x)\|b_j\|_q^q\\
 &\le
\frac{\mu(B(x_{B_j},2^{i+3}Mr_{B_j}))}{(2^ir_{B_j})^{Mq}\mu(B(x_{B_j},2^ir_{B_j}))}\|b_j\|_q^q\\
&\le C2^{-iMq}\mu(B_j)\|\chi_{B_j}\|_{p(\cdot)}^{-q}.
\end{align*}

In order to study $\sum_{x\in S_i(B_j)}I_{1,j}(x)^{q}\mu(x)$ we
split $m_j$ in the following way
\begin{align*}
m_j&=m_j\chi_{B(x_{B_j},2^{i-3}Mr_{B_j})}+m_j\chi_{B(x_{B_j},2^{i+4}Mr_{B_j})\setminus
B(x_{B_j},2^{i-3}Mr_{B_j})}+m_j\chi_{B(x_{B_j},2^{i+4}Mr_{B_j})^c}\\
&=m_{j,1}+m_{j,2}+m_{j,3}.
\end{align*}

Since $S_L$ is bounded in $L^q(\Gamma)$ we obtain
\begin{align*}
J_2&=\sum_{x\in
S_i(B_j)}\Big(\sum_{k=1}^{k=d(x_{B_j},x)/2}\sum_{d(x,y)<k}\frac{|k(I-P)P^{[k/2]}(m_{j,2})(y)|^2}{k\mu(B(y,k))}\mu
(y)\Big)^{q/2}\mu(x)\\
&\le C\|m_{j,2}\|_q^q\\
&\le
C\sum_{\ell=\ell_0+i-2}^{\ell=\ell_0+i+4}\|m_j\chi_{\mathfrak{S}_\ell(B_j)}\|_q^q,
\end{align*}
where $\ell_0\in \mathbb{N}$, $2^{\ell_0}\le M<2^{\ell_0+1}$. Molecular properties of $m_j$
and Lemma \ref{LemaCoc}, (i), imply that
\begin{align*}
J_2&\le C\sum_{\ell=\ell_0+i-2}^{\ell=\ell_0+i+4}2^{-\ell\varepsilon
q}\mu(B(x_{B_j},2^\ell r_{B_j}))\|\chi_{B(x_{B_j},2^\ell
r_{B_j})}\|_{p(\cdot)}^{-q}\\
&\le C2^{-i\varepsilon q}\mu(B(x_{B_j},2^{i+3}M
r_{B_j}))\|\chi_{B(x_{B_j},2^{i+3}M r_{B_j})}\|_{p(\cdot)}^{-q}.
\end{align*}

According to (\ref{compuesto}), (\ref{PMol}) and by taking into
account that $\supp(L(m_{j,1}))\cap S_i(B_j)=\varnothing$ and that $d(y,z)\ge c2^ir_{B_j}$ provided that $z\in
B(x_{B_j},2^{i-3}Mr_{B_j})$, $x\in S_i(B_j)$, $d(y,x)<k\le
d(x,x_{B_j})/2$, we can write
\begin{align*}
J_1&=\sum_{x\in
S_i(B_j)}\Big(\sum_{k=2}^{d(x_{B_j},x)/2}\sum_{d(x,y)<k}\frac{|k(I-P)P^{[k/2]}(m_{j,1})(y)|^2}{k\mu(B(y,k))}\mu
(y)\Big)^{q/2}\mu(x)\\
&\le C\sum_{x\in
S_i(B_j)}\Big(\sum_{k=2}^{d(x_{B_j},x)/2}\sum_{d(x,y)<k}\frac{1}{k\mu(B(x,k))}\\
&\hspace{1cm}
\times\Big(\sum_{z\in B(x_{B_j},2^{i-3}Mr_{B_j})}\frac{e^{-cd(y,z)^2/k}}{\mu(B(y,\sqrt{k}))}|m_j(z)|\mu(z)\Big)^2\mu(y)\Big)^{q/2}\mu(x)\\
&\le C\sum_{x\in
S_i(B_j)}\Big(\sum_{k=2}^{d(x_{B_j},x)/2}\sum_{d(x,y)<k}\frac{1}{k\mu(B(x,k))}\\
&\hspace{1cm}
\times\Big(\sum_{z\in B(x_{B_j},2^{i-3}Mr_{B_j})}\frac{e^{-cd(y,z)^2/k}}{\mu(B(y,\sqrt{k}))}|m_j(z)|^q\mu(z)\Big)^{2/q}\mu(y)\Big)^{q/2}\mu(x)\\
&\le C\sum_{x\in
S_i(B_j)}\Big(\sum_{k=2}^{d(x_{B_j},x)/2}\frac{1}{k}\|m_j\|_q^2\\
&\hspace{1cm}
\times\Big(\sum_{d(x,y)<k}\frac{\mu(y)e^{-c2^{2i}r_{B_j}^2/k}}{\mu(B(x,k))\mu(B(y,\sqrt{k}))}\Big)^{2/q}\Big)^{q/2}\mu(x)\\
&\le C\|m_j\|_q^q\sum_{x\in
S_i(B_j)}\Big(\sum_{k=2}^{d(x_{B_j},x)/2}\frac{1}{k}\Big(\sum_{d(x,y)<k}\frac{\mu(y)e^{-c2^{2i}r_{B_j}^2/k}}{\mu(B(x,k))}\\
&\hspace{1cm}\times\frac{\mu(B(y,d(x,x_{B_j})))}{\mu(B(y,\sqrt{k}))}\frac{1}{\mu(B(y,d(x,x_{B_j})))}\Big)^{2/q}\Big)^{q/2}\mu(x)\\
&\le C\|m_j\|_q^q\sum_{x\in
S_i(B_j)}\Big(\sum_{k=2}^{d(x_{B_j},x)/2}\frac{1}{k}\Big(\sum_{d(x,y)<k}\frac{\mu(y)e^{-c2^{2i}r_{B_j}^2/k}}{\mu(B(x,k))}\\
&\hspace{1cm}\times\Big(\frac{d(x,x_{B_j})}{\sqrt{k}}\Big)^D\frac{1}{\mu(B(y,d(x,x_{B_j})))}\Big)^{2/q}\Big)^{q/2}\mu(x)\\
&\le C\|m_j\|_q^q\sum_{x\in
S_i(B_j)}\Big(\sum_{k=2}^{d(x_{B_j},x)/2}\frac{1}{k}\Big(\sum_{d(x,y)<k}\frac{\mu(y)e^{-c2^{2i}r_{B_j}^2/k}}{\mu(B(x,k))}\\
&\hspace{1cm}\times\Big(\frac{r_{B_j}2^{i+3}}{\sqrt{k}}\Big)^D\frac{1}{\mu(B(x,d(x,x_{B_j})))}\Big)^{2/q}\Big)^{q/2}\mu(x)\\
&\le C\|m_j\|_q^q\sum_{x\in
S_i(B_j)}\Big(\sum_{k=2}^{d(x_{B_j},x)/2}\frac{e^{-c2^{2i}r_{B_j}^2/k}}{k\mu(B(x,d(x,x_{B_j})))^{2/q}}\Big)^{q/2}\mu(x)\\
&\le C\|m_j\|_q^q\sum_{x\in
S_i(B_j)}d(x_{B_j},x)^{q/2}\frac{e^{-c2^{2i}r_{B_j}^2/d(x_{B_j},x)}}{\mu(B(x_{B_j},d(x,x_{B_j})))}\mu(x)\\
&\le C\|m_j\|_q^q(2^ir_{B_j})^{q/2}e^{-c2^{i}r_{B_j}}\\
&\le C e^{-c2^i}\mu(B_j)\|\chi_{B_j}\|_{p(\cdot)}^{-q}.
\end{align*}

On the other hand, $d(y,z)\ge c2^ir_{B_j}$ provided that $z\notin
B(x_{B_j},2^{i+4}Mr_{B_j})$, $x\in S_i(B_j)$, $d(y,x)<k\le
d(x,x_{B_j})/2$. By proceeding as above we deduce that
\begin{align*}
J_3&=\sum_{x\in
S_i(B_j)}\Big(\sum_{k=1}^{d(x_{B_j},x)/2}\sum_{d(x,y)<k}\frac{|k(I-P)P^{[k/2]}(m_{j,3})(y)|^2}{k\mu(B(y,k))}\mu
(y)\Big)^{q/2}\mu(x)\\
&\le C e^{-c2^i}\mu(B_j)\|\chi_{B_j}\|_{p(\cdot)}^{-q}.
\end{align*}

By combining the above estimates we get
\begin{align*}
\sum_{x\in S_i(B_j)}I_{1,j}(x)^{q}\mu(x)&\le
C\Big(e^{-c2^i}\mu(B_j)\|\chi_{B_j}\|_{p(\cdot)}^{-q}+2^{-i\varepsilon
q}\mu(B(x_{B_j},2^{i+3}M r_{B_j}))\|\chi_{B(x_{B_j},2^{i+3}M
r_{B_j})}\|_{p(\cdot)}^{-q}\Big).
\end{align*}

By using Lemma \ref{LemaCoc}, (i), we obtain
\begin{align*}
\|S_L(m_j)&\chi_{S_i(B_j)}\|_q\le C \Big(\Big(\sum_{x\in
S_i(B_j)}I_{1,j}^{q}\mu(x)\Big)^{1/q}+\Big(\sum_{x\in
S_i(B_j)}I_{2,j}^{q}\mu(x)\Big)^{1/q}\Big)\\
&\le
C\Big(2^{-iM}\mu(B_j)^{1/q}\|\chi_{B_j}\|_{p(\cdot)}^{-1}+2^{-i\varepsilon
}\mu(B(x_{B_j},2^{i+3}M r_{B_j}))^{1/q}\|\chi_{B(x_{B_j},2^{i+3}M
r_{B_j})}\|_{p(\cdot)}^{-1}\Big)\\
&\le C(2^{-i\varepsilon}+2^{-i(M-D/w)})\mu(B(x_{B_j},2^{i+3}M
r_{B_j}))^{1/q}\|\chi_{B(x_{B_j},2^{i+3}M
r_{B_j})}\|_{p(\cdot)}^{-1},
\end{align*}
where $2D/M<w<p_-$ and $D/\varepsilon<w$.

According to Lemmas \ref{LemaSum} and \ref{LemaSumCoc} we get
\begin{align*}
\Big\|\Big(\sum_{j=j_1}^{j_2} &\Big(|\lambda _j|S_L(m_j)\chi
_{\mathcal{S}_i(B_j)}\Big)^{\mathfrak{p}}
\Big)^{1/\mathfrak{p}}\Big\|_{p(\cdot )}\\
&\le
C(2^{-i\varepsilon}+2^{-i(M-D/w)})\Big\|\Big(\sum_{j=j_1}^{j_2}\Big(\frac{|\lambda_j|\chi_{B(x_{B_j},2^{i+3}Mr_{B_j})}}{\|\chi_{B(x_{B_j},2^{i+3}Mr_{B_j})}\|_{p(\cdot)}}\Big)^{\mathfrak{p}}\Big)^{1/\mathfrak{p}}\Big\|_{p(\cdot)}\\
&\le(2^{-i(\varepsilon-D/w)}+2^{-i(M-2D/w)})\Big\|\Big(\sum_{j=j_1}^{j_2}\Big(\frac{|\lambda_j|\chi_{B_j}}{\|\chi_{B_j}\|_{p(\cdot)}}\Big)^{\mathfrak{p}}\Big)^{1/\mathfrak{p}}\Big\|_{p(\cdot)}.
\end{align*}

By (\ref{ForSum}) we conclude that
$$
\Big\|S_L\Big(\sum_{j=j_1}^{j_2}\lambda_j m_j\Big)\Big\|_{p(\cdot
)}\leq \Big\|\Big(\sum_{j=j_1}^{j_2} (|\lambda _j|S_L(m_j
)\chi_{\mathcal{S}_i(B_j)})^\mathfrak{p}\Big)^{1/\mathfrak{p}}\Big\|_{p(\cdot
)}\le
C\Big\|\Big(\sum_{j=j_1}^{j_2}\Big(|\lambda_j|\chi_{B_j}\|\chi_{B_j}\|_{p(\cdot)}^{-1}\Big)^\mathfrak{p}\Big)^{1/\mathfrak{p}}\Big\|_{p(\cdot)}.
$$

Let $\epsilon >0$. Since the series
$$
\sum_{j=0}^\infty\Big(|\lambda _j|\chi _{B_j}\|\chi _{B_j}\|_{p(\cdot)}^{-1}\Big)^\frak{p}
$$
converges in $L^{p(\cdot )/\frak{p}}(\Gamma )$, there exists $j_0\in \mathbb{N}$ such that, for every $j_1,j_2\in \mathbb{N}$, $j_0\leq j_1 <j_2$,
$$
\Big\|S_L\Big(\sum_{j=j_1 }^{j_2}\lambda _jm_j\Big)\Big\|_{p(\cdot
)}<\epsilon.
$$
The completeness of $H^{p(\cdot )}_L(\Gamma )$ implies that the series $\sum_{j=1}^\infty \lambda _jm_j$ converges in $H^{p(\cdot )}_L(\Gamma )$. By writing $f=\sum_{j=1}^\infty \lambda _jm_j$ we have that
\begin{align*}
\|f\|_{H^{p(\cdot )}_L(\Gamma )} &=\lim_{n\rightarrow \infty
}\Big\|\sum_{j=1}^n\lambda_jm_j\Big\|_{H^{p(\cdot )}_L(\Gamma )}
\leq C\lim_{n\rightarrow \infty }\Big\|\Big(\sum_{j=1}^n(|\lambda _j|\chi _{B_j}\|\chi _{B_j}\|_{p(\cdot )}^{-1})^\frak{p}\Big)^{1/\frak{p}}\Big\|_{p(\cdot )}\\
&\leq C\Big\|\Big(\sum_{j=0}^\infty (|\lambda _j|\chi _{B_j}\|\chi _{B_j}\|_{p(\cdot )}^{-1})^\frak{p}\Big)^{1/\frak{p}}\Big\|_{p(\cdot )}.
\end{align*}

\end{proof}

 For every complex function $f$ defined on $\Gamma $ we consider the maximal function $\mathcal{M}_+(f)$ given by
$$
\mathcal{M}_+(f)(x)=\sup_{n\in \mathbb{N}}|P^n(f)(x)|,\quad x\in \Gamma .
$$
According to \cite[(2.3)]{Fe} we have that
$|P^k(f)|\le C\mathcal{M}(f)$, $k\in \mathbb{N}_+$. Then, $\mathcal{M}_+(f)\le C \mathcal{M}(f)$ and the maximal function $\mathcal{M}_+$ defines a bounded (sublinear) operator from $L^q(\Gamma)$ into itself, for every $1<q\le \infty$.

We say that $f\in L^2(\Gamma)$ is in
$\mathbb{H}^{p(\cdot)}_{L,+}(\Gamma)$ when $\mathcal{M}_+(f)\in
L^{p(\cdot)}(\Gamma)$. The maximal Hardy space
$H^{p(\cdot)}_{L,+}(\Gamma)$ is the completion of
$\mathbb{H}^{p(\cdot)}_{L,+}(\Gamma)$ with respect to the norm
$\|\cdot\|_{ H^{p(\cdot)}_{L,+}(\Gamma)}$ defined by
$$
\|f\|_{H^{p(\cdot)}_{L,+}(\Gamma)}=\|\mathcal{M}_+(f)\|_{p(\cdot)},\,\,\,f\in
\mathbb{H}^{p(\cdot)}_{L,+}(\Gamma).
$$
Next we establish that $H^{p(\cdot)}_{L}(\Gamma)$ is a subspace of
$H^{p(\cdot)}_{L,+}(\Gamma)$.
\begin{Prop}\label{M+}
Let $p\in \mathcal{P}^{\log}(\Gamma)$. There exists $C>0$ such that,
for every $f\in H^{p(\cdot )}_L(\Gamma )\cap L^2(\Gamma )$ we have
that
$$
\|\mathcal{M}_+(f)\|_{p(\cdot )}\leq C\|f\|_{H^{p(\cdot )}_L(\Gamma )}.
$$
Hence, $H^{p(\cdot)}_{L}(\Gamma)$ is contained in
$H^{p(\cdot)}_{L,+}(\Gamma)$.
\end{Prop}
\begin{proof}
Let $f\in H^{p(\cdot )}_L(\Gamma )\cap L^2(\Gamma )$. We take
$q=2\max\{2,p_+\}$ and $M\in \mathbb{N}$, $M>4D/p_-$. According to the proof of Proposition \ref{Hardy}, there
exists, for every $j\in \mathbb{N}$, $\lambda _j\in \mathbb{C}$ and
a $(q,p(\cdot ),M)$-atom $a_j$ associated to the ball
$B_j=B(x_{B_j},r_{B_j})$, with $x_{B_j}\in \Gamma$ and $r_{B_j}\ge
1$, such that
$$
f=\sum_{j=0}^\infty \lambda _ja_j,\quad \mbox{ in }L^2(\Gamma
)\mbox{ and in }H^{p(\cdot )}_L(\Gamma ),
$$
and $\mathcal{A} (\{\lambda _j\},\{B_j\})\le
C\|f\|_{H^{p(\cdot)}_L(\Gamma)}$. Since, for every $n\in
\mathbb{N}$, $P^n$ is bounded in $L^2(\Gamma )$ we have that, for
every $n\in \mathbb{N}$,
$$
P^n(f)=\sum_{j=0}^\infty \lambda _jP^n(a_j),\quad \mbox{ in
}L^2(\Gamma ),
$$
and then
$$
P^n(f)(x)=\sum_{j=0}^\infty \lambda_j P^n(a_j)(x),\quad x\in \Gamma .
$$
Hence, we can write
$$
\mathcal{M}_+(f)(x)\leq \sum_{j=0}^\infty |\lambda _j|\mathcal{M}_+(a_j)(x),\quad x\in \Gamma .
$$
We deduce that
$$
\|\mathcal{M}_+(f)\|_{p(\cdot )}\leq \Big(\sum_{i=0}^\infty \Big\|\Big(\sum_{j=0}^\infty (|\lambda _j|\mathcal{M}_+(a_j)\chi _{S_i(B_j)})^\frak{p}\Big)^{1/\frak{p}}\Big\|_{p(\cdot )}^\frak{p}\Big)^{1/\frak{p}}.
$$
Since $\mathcal{M}_+$ is bounded in $L^q(\Gamma )$ we obtain
$$
\|\mathcal{M}_+(a_j)\|_{L^q(S_0(B_j))}^q\leq C\|a_j\|_{L^q(\Gamma )}^q\leq C(\mu(B_j)^{1/q}\|\chi _{B_j}\|_{p(\cdot )}^{-1})^q,\quad j\in \mathbb{N}.
$$
By Lemmas \ref{LemaSum}, \ref{LemaCoc}, (i), and \ref{LemaSumCoc} we get
$$
\Big\|\Big(\sum_{j=1}^\infty (|\lambda _j|\mathcal{M}_+(a_j)\chi _{S_0(B_j)})^\frak{p}\Big)^{1/\frak{p}}\Big\|_{p(\cdot )}\leq C\mathcal{A}(\{\lambda _j\},\{B_j\}).
$$
Let $j\in \mathbb{N}$ and $i\in \mathbb{N}_+$. Since $P^n(a_j)(x)=0$, when $x\in S_i(B_j)$ and $n\in \mathbb{N}$, $n<2^i$, we can write
$$
\mathcal{M}_+(a_j)(x)= \sup_{n\ge 2^i}|P^n(a_j)(x)|,\quad x\in S_i(B_j).
$$
%Let $i\in \mathbb{N}_+$. By using (\ref{UE}) we get
%\begin{align*}
%I_{j,1}(x)&=\sup_{n=1,...,d(x,x_{B_j})}\Big| \sum_{y\in B_j}p_n(x,y)a_j(y)\Big|\\
%&\leq C\sup_{n=1,...,d(x,x_{B_j})}\sum_{y\in B_j}\frac{1}{\mu(B(x,\sqrt{n}))}\exp \Big(-c\frac{d(x,y)^2}{n}\Big)|a_j(y)|\mu (y)\\
%\end{align*}
%Hence, for every $m\in \mathbb{N}$, there exists $C>0$ such that
%\begin{align*}
%I_{j,1}(x)^q&\leq C\|a_j\|_q^q\sup_{n=1,...,d(x,x_{B_j})}\Big(\frac{n}{d(x,x_{B_j})^2}\Big)^m\frac{\mu(B_j)^{q-1}}{\mu(B(x_{B_j},d(x,x_{B_j})))^q}\\
%&\leq C\|a_j\|_q^q\frac{\mu(B_j)^{q-1}}{\mu(B(x_{B_j},d(x,x_{B_j})))^q}\frac{1}{d(x,x_{B_j})^m}\\
%&\leq C\|a_j\|_q^q\mu(B_j)^{-1}(d(x,x_{B_j}))^{-m},\quad x\in \Gamma \setminus S_0(B_j).
%\end{align*}
%Then, the doubling property leads to
%\begin{align}\label{R2}
%\|I_{j,1}\|_{L^q(S_i(B_j))}&\leq C\|a_j\|_q\Big(\frac{\mu(B(x_{B_j},2^{i+3}r_{B_j}))}{\mu(B_j)}\Big)^{1/q}\frac{1}{(2^ir_{B_j})^m}\nonumber\\
%&\leq C2^{i(D/q-m)}\|a_j\|_q\leq C2^{i(D/q-m)}\mu(B_j)^{1/q}\|\chi
%_{B_j}\|_{p(\cdot )}^{-1}.
%\end{align}
%Here, $m\in \mathbb{N}$ will be fixed later.\\

Since $a_j$ is a $(q,p(\cdot),M)$-atom we can write $a_j=L^Mb_j$,
where $\mbox{supp }b_j\subset B_j$ and $\|b_j\|_q\leq
r_{B_j}^M\mu(B_j)^{1/q}\|\chi _{B_j}\|_{p(\cdot )}^{-1}$. According
to (\ref{compuesto}) and (\ref{max1}), for every $n\in \mathbb{N}_+$, the kernel
$\widetilde{p}_{n,M}$ of the operator $L^MP^n$ satisfies that
$$
\sum_{y\in \Gamma}|\widetilde{p}_{n,M}(x,y)|\leq \frac{C}{n^M},\quad
x\in \Gamma.
$$
We define $\mathbb{P}_{n,M}(x)=\sum_{y\in \Gamma
}|\widetilde{p}_{n,M}(x,y)|$, $x\in \Gamma $. By using Jensen's
inequality we obtain, for every complex function $g$ defined on
$\Gamma$,
\begin{align}\label{Z11}
 \Big(\sum_{y\in \Gamma }|\widetilde{p}_{n,M}(x,y)|g(y)\Big)^q
&\leq \Big(\sum_{y\in \Gamma }\frac{|\widetilde{p}_{n,M}(x,y)|}{\mathbb{P}_{n,M}(x)}g(y)\Big)^q\mathbb{T}_{M,n}(x)^q\nonumber\\
&  \leq \mathbb{P}_{n,M}(x)^{q-1}\sum_{y\in \Gamma }|\widetilde{p}_{n,M}(x,y)||g(y)|^q\nonumber\\
&\leq \frac{C}{n^{M(q-1)}}\sum_{y\in \Gamma
}|\widetilde{p}_{n,M}(x,y)||g(y)|^q,\quad x\in \Gamma .
\end{align}
By (\ref{Z11}) it follows that, for every $n\in \mathbb{N}_+$,
\begin{align*}
|P^n(a_j)(x)|^q&\leq \frac{C}{n^{M(q-1)}}\sum_{y\in B_j}|\widetilde{p}_{n,M}(x,y)||b_j(y)|^q\\
&\leq \frac{C}{n^{Mq}\mu(B(x,\sqrt{n}))}\sum_{y\in B_j}\exp\Big(-c\frac{d(x,y)^2}{n}\Big)|b_j(y)|^q\mu (y)\\
&\le  \frac{C}{n^{Mq}}\frac{e^{-cd(x_{B_j},x)^2/n}}{\mu(B(x,\sqrt{n}))}\|b_j\|_q^q\\
&\le C\Big(\frac{r_{B_j}}{n}\Big)^{Mq}\frac{e^{-cd(x_{B_j},x)^2/n}}{\mu(B(x,\sqrt{n}))}\mu(B_j)\|\chi_{B_j}\|_{p(\cdot)}^{-q}\\
&\le C\Big(\frac{r_{B_j}}{n}\Big)^{Mq}e^{-cr_{B_j}^2/n}\frac{e^{-cd(x_{B_j},x)^2/n}}{\mu(B(x,\sqrt{n}))}\mu(B_j)\|\chi_{B_j}\|_{p(\cdot)}^{-q},\quad x\in \Gamma\setminus S_0(B_j) .
\end{align*}
We have taken into account that $d(x,y)\ge d(x,x_{B_j})/2\ge 4r_{B_j}$, when $x\in \Gamma\setminus S_0(B_j)$ and $y\in B_j$.

By putting $D_n(x)=D$, when $d(x,x_{B_j})>\sqrt{n}$, and $D_n(x)=0$,
when $d(x,x_{B_j})\le \sqrt{n}$, we obtain
\begin{align*}
\sup_{n\ge 2^i}|P^n(a_j)(x)|&\le C\sup_{n\ge 2^i}\frac{e^{-cd(x_{B_j},x)^2/n}}{n^{M/2}\mu(B(x,\sqrt{n}))^{1/q}}\mu(B_j)^{1/q}\|\chi_{B_j}\|_{p(\cdot)}^{-1}\\
&\le C\sup_{n\ge 2^i}\frac{e^{-cd(x_{B_j},x)^2/n}\mu(B(x,d(x_{B_j},x))^{1/q}}{n^{M/2}\mu(B(x,\sqrt{n})))^{1/q}}\Big(\frac{\mu(B_j)}{\mu(B(x,d(x,x_{B_j})))}\Big)^{1/q}\|\chi_{B_j}\|_{p(\cdot)}^{-1}\\
&\le C\sup_{n\ge 2^i}\frac{e^{-cd(x_{B_j},x)^2/n}}{n^{M/2}}\Big(\frac{d(x,x_{B_j})}{\sqrt{n}}\Big)^{D_n(x)/q}\Big(\frac{\mu(B_j)}{\mu(B(x_{B_j},d(x,x_{B_j})))}\Big)^{1/q}\|\chi_{B_j}\|_{p(\cdot)}^{-1}\\
&\le C2^{-iM/2}\Big(\frac{\mu(B_j)}{\mu(B(x_{B_j},d(x,x_{B_j})))}\Big)^{1/q}\|\chi_{B_j}\|_{p(\cdot)}^{-1},\quad x\in \Gamma\setminus S_0(B_j) .
\end{align*}

Then,
\begin{align*}
\|\mathcal{M}_+(a_j)\|_{L^q(S_i(B_j))}&\le C2^{-iM/2}\Big(\frac{\mu(B_j)}{\mu(B(x_{B_j},2^{i+2}Mr_{B_j}))}\Big)^{1/q}(\mu(S_i(B_j)))^{1/q}\|\chi_{B_j}\|_{p(\cdot)}^{-1}\\
&\le C2^{-iM/2}(\mu(B_j))^{1/q}\|\chi_{B_j}\|_{p(\cdot)}^{-1}\\
&\le C2^{-i(M/2-D/w)}(\mu(B(x_{B_j},2^{i+3}Mr_{B_j})))^{1/q}\|\chi_{B(x_{B_j},2^{i+3}Mr_{B_j})}\|_{p(\cdot)}^{-1}.
\end{align*}

%, we get
%\begin{align}\label{R3}
%& \|I_{j,2}(x)\|_{L^q(S_i(B_j))}^q
%\leq C\sum_{x\in S_i(B_j)}\sup _{n>d(x,x_{B_j})}\frac{1}{n^{Mq}\mu(B(x,\sqrt{n}))}\sum_{y\in \Gamma}\exp\Big(-c\frac{d(x,x_{B_j})^2}{n}\Big)|b_j(y)|^q\mu (y)\mu (x)\nonumber\\
%&\qquad \leq C\sum_{y\in B_j}|b_j(y)|^q\Big(\sum_{x\in S_i(B_j)}\sup _{n>d(x,x_{B_j})}\frac{1}{n^{Mq}\mu(B(x,\sqrt{n}))} \exp\Big(-c\frac{d(x,x_{B_j})^2}{n}\Big)\mu (x)\Big)\mu (y)\nonumber\\
%&\qquad \leq C\sum_{y\in B_j}|b_j(y)|^q\Big(\sum_{x\in S_i(B_j)}\frac{1}{d(x,x_{B_j})^{Mq}}\nonumber\\
%&\hspace{2cm}\times\,\sup _{n>d(x,x_{B_j})}\Big(\exp\Big(-c\frac{d(x,x_{B_j})^2}{n}\Big)\Big(\frac{d(x,x_{B_j})}{\sqrt{n}}\Big)^{D_n(x)}\Big)\frac{\mu (x)}{\mu(B(x_{B_j},d(x,x_{B_j})))}\Big)\mu (y)\nonumber\\
%&\qquad \leq C\sum_{y\in B_j}|b_j(y)|^q\mu (y)(2^ir_{B_j})^{-Mq}\frac{\mu(S_i(B_j))}{\mu(B(x_{B_j},2^ir_{B_j}))}\nonumber\\
%&\qquad \leq C2^{-iMq}\|b_j\|_q^qr_{B_j}^{-Mq}\nonumber\\
%&\qquad \leq C2^{-iMq}(\mu(B_j)^{1/q}\|\chi _{B_j}\|_{p(\cdot )}^{-1})^q.
%\end{align}
%Here $D_n(x)=D$, when $d(x,x_{B_j})>\sqrt{n}$, and $D_n(x)=0$, when $d(x,x_{B_j})\le\sqrt{n}$.
%
%By combining (\ref{R2}), with $m$ large enough, and (\ref{R3}) we show that there exists $C>0$ such that
%$$
%\|\mathcal{M}_+(a_j)\|_{L^q(S_i(B_j))}\leq C2^{-iM}\mu(B_j)^{1/q}\|\chi _{B_j}\|_{p(\cdot )}^{-1}.
%$$

We take $2D/M<w<p_-$. We recall that $2D/M<p_-$. Lemmas
\ref{LemaSum} and \ref{LemaSumCoc} lead to
$$
\Big\|\Big(\sum_{j\in \mathbb{N}}(|\lambda _j|\mathcal{M}_+(a_j)\chi _{S_i(B_j)})^\frak{p}\Big)^{1/\frak{p}}\Big\|_{p(\cdot )}\leq C2^{i(2D/w-M/2)}\|f\|_{H^{p(\cdot )}_L(\Gamma )}.
$$
Then,
$$
\|\mathcal{M}_+(f)\|_{p(\cdot )}\leq C\|f\|_{H^{p(\cdot
)}_L(\Gamma)}
$$

\end{proof}

We now consider other type of atoms (see \cite[Definition
4.1]{ZSY}). We say that a complex valued function $a$ defined in
$\Gamma$ is a $(2,p(\cdot))$-atom associated with the ball $B$ when
the following properties are satisfied

(i) $\supp(a)\subset B$,

(ii) $\|a\|_2\le (\mu(B))^{1/2}\|\chi_B\|_{p(\cdot)}^{-1}$,

(iii) $\sum_{x\in B}a(x)\mu(x)=0$.

Note that if $M\in \mathbb{N}_+$ and $a$ is a $(2,p(\cdot),M)$-atom
associated with the ball $B$, then $a$ is a $(2,p(\cdot))$-atom
(\cite[Remark 3.4, (i)]{BD}).

A function $f\in L^2(\Gamma)$ is in the Hardy space
$\mathbb{H}^{p(\cdot)}_{atom}(\Gamma)$ when, for every $j\in
\mathbb{N}$, there exist $\lambda_j\in \mathbb{C}$ and a
$(2,p(\cdot))$-atom $a_j$ associated to the ball $B_j$ such that
$f=\sum_{j=0}^\infty \lambda_ja_j$, in $L^2(\Gamma)$, and
$\mathcal{A}(\{\lambda_j\},\{B_j\})<\infty$. For every $f\in
\mathbb{H}^{p(\cdot)}_{atom}(\Gamma)$ we define
$\|f\|_{H^{p(\cdot)}_{atom}(\Gamma)}$ by
$$
\|f\|_{H^{p(\cdot)}_{atom}(\Gamma)}=\inf
\,\mathcal{A}(\{\lambda_j\},\{B_j\}),
$$
where the infimum is taken over all the sequences
$\{\lambda_j\}_{j=0}^\infty$ of complex numbers and
$\{B_j\}_{j=0}^\infty$ of balls in $\Gamma$ such that, for every
$j\in \mathbb{N}$, there exists a $(2,p(\cdot))$-atom associated
with the ball $B_j$ and $f=\sum_{j=0}^\infty \lambda_ja_j$, in
$L^2(\Gamma)$. Thus, $\|.\|_{H^{p(\cdot)}_{atom}(\Gamma)}$ is a
(quasi)norm in $\mathbb{H}^{p(\cdot)}_{atom}(\Gamma)$.

We define the space $H^{p(\cdot)}_{atom}(\Gamma)$ as the completion
of $\mathbb{H}^{p(\cdot)}_{atom}(\Gamma)$ with respect to
$\|.\|_{H^{p(\cdot)}_{atom}(\Gamma)}$. We have that
$H^{p(\cdot)}_{L,M,at}(\Gamma)$ is a subspace of
$H^{p(\cdot)}_{atom}(\Gamma)$.

We say that a graph $(\Gamma,\mu,d)$ has the Poincar\'e property
when there exists $C>0$ such that for every
$f:\;\Gamma\;\rightarrow\;\mathbb{R}$, $x_0\in\Gamma$ and $r_0>0$,
$$\sum_{x\in B(x_0,r_0)} |f(x)-f_{B(x_0,r_0)}|^2\mu(x)\leq Cr_0^2\sum_{x,y\in B(x_0,2r_0)} |f(x)-f(y)|^2\nu(x,y),$$
where
$$f_{B(x_0,r_0)}=\frac{1}{\mu(B(x_0,r_0))}\sum_{x\in B(x_0,r_0)} f(x)\mu(x).$$

In (\cite[Lemma 4]{Ru1}) it was established that if $\mu$ is
doubling, $\Gamma$ satisfies the Poincar\'e property and property
$\Gamma$ satisfies $\Delta(\alpha)$, with $\alpha >0$, then there
exist $c_3,C_3>0$ and $h\in (0,1)$ such that, for every
$n\in\mathbb{N}$ and $x,y,y_0\in\Gamma$ being
$d(y_0,y)\leq\sqrt{n}$,
\begin{equation}\label{Poin}
 |p_n(y,x)-p_n(y_0,x)|
  \leq C_3\Big(\frac{d(y,y_0)}{\sqrt{n}}\Big)^h\frac{\mu(x)}{\mu(B(x,\sqrt{n}))}e^{-c_3d^2(x,y_0)/n}.
  \end{equation}

We now prove that, under certain conditions,
$H^{p(\cdot)}_{atom}(\Gamma)$ is a subspace of
$H^{p(\cdot)}_{L,+}(\Gamma)$.

\begin{Prop} Let $p\in \mathcal{P}^{log}(\Gamma)$. Assume that $(\Gamma,\mu,d)$ has, in addition to the properties we have adopted from the beginning, the Poincar\'e property, $p_+<2$, and $\frac{D}{D+h}<\mathfrak{p}$, where $h$ is the one in (\ref{Poin}). Then, $H^{p(\cdot)}_{atom}(\Gamma)$ is continuously contained in $H^{p(\cdot)}_{L,+}(\Gamma)$.
\end{Prop}

\begin{proof} Let $f\in H^{p(\cdot)}_{atom}(\Gamma)\cap L^2(\Gamma)$. Then, $f=\sum_{j=0}^\infty \lambda_ja_j$, in $L^2(\Gamma)$, where for every $j\in \mathbb{N}$, there exist $\lambda_j\in \mathbb{C}$ and a $(2,p(\cdot))$-atom $a_j$ associated to the ball $B_j=B(x_{B_j},r_{B_j})$ such that $\mathcal{A}(\{\lambda_j\},\{B_j\})<\infty$. For every $k\in \mathbb{N}$,
$$
P^k(f)=\sum_{j=0}^\infty \lambda_jP^k(a_j), \,\,\,\rm{in}\,\,\,L^2(\Gamma),
$$
and
$$
P^k(f)(x)=\sum_{j=0}^\infty \lambda_jP^k(a_j)(x), \,\,\,x\in \Gamma.
$$
Assume that $a$ is a $(2,p(\cdot))$-atom associated with the ball
$B=B(x_B,r_B)$.

In the following we are inspired by some ideas developed in \cite[p.
60-61]{Ru2}.

We now assume that $x\not\in B(x_B,2r_B)$. We take firstly $k\in\mathbb{N}$ such that $k\leq r_B^2$. According to (\ref{UE}) we have that
$$
|P^k(a)(x)|\leq C\sum_{d(y,x_B)<
r_B}\frac{1}{\mu(B(x,\sqrt{k}))}e^{-cd^2(x,y)/k}|a(y)|\mu(y).
$$
If $y\in B$, then $d(x,y)\geq \frac{1}{2}d(x,x_B)$ so, we have
\begin{align*}
   |P^k(a)(x)|&  \leq  C\frac{e^{-cd(x,x_B)^2/k}}{\mu(B(x,\sqrt{k}))}(\mu(B))^{1/2}\|a\|_2 \\
  & \leq C\frac{e^{-cd(x,x_B)^2/k}}{\mu(B(x,\sqrt{k}))}\mu(B)\|\chi_{B}\|_{p(\cdot)}^{-1} \\
  & \leq C\frac{e^{-cd(x,x_B)^2/k}}{\mu(B(x,\sqrt{k}))}\Big(\frac{r_{B}}{\sqrt{k}}\Big)^h\mu(B)\|\chi_{B}\|_{p(\cdot)}^{-1}.
  \end{align*}
  %& \leq C \frac{e^{-cd(x,x_B)^2/k}}{\|\chi_{B(x_B,r_B)}\|_{p(\cdot)}}\frac{\mu(B(x,r_B+\sqrt{k}))}{\mu(B(x,\sqrt{k}))}\frac{\mu(B(x_B,r_B))}{\mu(B(x,r_B+\sqrt{k}))}e^{-cr_B^2/k} \\
%   & \leq C \frac{e^{-cd(x,x_B)^2/r_B^2}}{\|\chi_{B(x_Br_B)}\|_{p(\cdot)}}\Big(\frac{r_B+\sqrt{k}}{\sqrt{k}}\Big)^D\frac{\mu(B(x_B,r_B))}{\mu(B(x,r_B+\sqrt{k}))}e^{-cr_B^2/k} \\
%    & \leq C \frac{\mu(B(x_B,r_B))}{\mu(B(x,r_B))}\frac{e^{-cd(x,x_B)^2/r_B^2}}{\|\chi_{B(x_B,r_B)}\|_{p(\cdot)}}.
%\end{align*}
%Let $\eta >0$. We can write
%$$
%e^{-cd(x,x_B)^2/r_B^2}\leq C\Big(\frac{r_B}{d(x,x_B)+r_B}\Big)^\eta
%e^{-cd(x,x_B)^2/r_B^2}\leq
%C\Big(\frac{\mu(B(x_B,r_B))}{\mu(B(x_B,d(x,x_B)+r_B))}\Big)^{\eta/D}
%e^{-cd^2(x,x_B)/r_B^2}.$$ Since $B(x,d(x,x_B)+r_B)\subset
%B(x_B,2(d(x,x_B)+r_B))\subseteq B(x,3(d(x,x_B)+r_B))$,
%\begin{align*}
%\mu(B(x,3(d(x,x_B)+r_B)))&\leq C \mu(B(x,d(x,x_B)+r_B))\\
%&\leq C \mu(B(x_B,2(d(x,x_B)+r_B)))\\
%&\leq C\mu(B(x_B,d(x_B,x)+r_B)).
%\end{align*}
%We get
%\begin{align*}
%    |P^k(a)(x)|  \leq & C\Big(\frac{\mu(B(x_B,r_B))}{\mu(B(x_B,d(x,x_B)+r_B))}\Big)^{\eta/D+1}e^{-cd(x,x_B)^2/r_B^2}\\
%    &\qquad \times\|\chi_{B(x_B,r_B)}\|_{p(\cdot)}^{-1} \frac{\mu(B(x,d(x,x_B)+r_B))}{\mu(B(x,r_B))} \\
%  & \leq C\Big(\frac{1}{\mu(B(x,d(x,x_B)+r_B))}\sum_{y\in B(x,d(x,x_B)+r_B)}\chi_{B(x_B,r_B)}(y)\mu(y)\Big)^{\eta/D+1} \\
%  & \qquad \times \frac{e^{-  cd(x,x_B)^2/r_B^2}}{\|\chi_{B(x_B,r_B)}\|_{p(\cdot)}}\Big(\frac{d(x,x_B)}{r_B} +1\Big)^D \\
%  & \leq C \Big(\mathcal{M}\Big(\chi_{B(x_B,r_B)}\Big)(x)\Big)^{\eta/D+1}\|\chi_{B(x_B,r_B)}\|_{p(\cdot)}^{-1} .
%\end{align*}
We now assume that $k> r_B^2$. Since $\sum_{\Gamma}a(y)\mu(y)=0$, we
can write
$$P^k(a)(x)=\sum_{y\in\Gamma}\Big(\frac{p_k(x,y)}{\mu(y)}-\frac{p_k(x,x_B)}{\mu(x_B)}\Big)a(y)\mu(y).$$

%\begin{equation}\label{G2}
%|p_k(y,x)-p_k(y_o,x)|\leq C_3\Big(\frac{d(y,y_0)}{\sqrt{k}}\Big)^h\frac{m(x)}{V(x,\sqrt{x})}e^{c_3d^2(x,y_0)/k},\;\;\;d(y_0,y)\leq\sqrt{k},
%\end{equation}
%for certain $C_3,c_3 >0$ and $h\in(0,1)$.\\
%
%Also, we assume that ((UE)property in \cite{BD})
%\begin{equation}\label{G3}
%p_k(x.y)\leq \frac{m(y)}{V(x,\sqrt{x})}e^{-c d^2(x,y)/k},\;\;\;x,y\in\Gamma\;\mbox{and}\;k\in\mathbb{N}.
%\end{equation}
%
%In \cite[Lemma 5]{Ru1} some conditions on the graphs are specified in order that (\ref{G2}) and (\ref{G3}) are satisfied (see also \cite[Theorem 3]{Ru1}).\\

  By taking into account that $p(x,y)\mu(x)=p(y,x)\mu(y)$, $x,y\in\Gamma$, we also obtain, for every $k\in\mathbb{N}$, $p_k(x,y)\mu(x)=p_k(y,x)\mu(y)$, $x,y\in\Gamma$. From (\ref{Poin}) and the relation we have just shown we deduce that
\begin{align*}
    \Big|\frac{p_k(x,y)}{\mu(y)}-  \frac{p_k(x,x_B)}{\mu(x_B)}\Big| & = \frac{|p_k(y,x)-p_k(x_B,x)|}{\mu(x)} \\
  & \leq C\Big(\frac{d(y,x_B)}{\sqrt{k}}\Big)^h\frac{1}{\mu(B(x,\sqrt{k}))}e^{-cd(x,x_B)^2/k},\;\;\;d(y,x_B)\leq\sqrt{k}.
\end{align*}
Since $k> r_B^2$, we get
\begin{align*}
    |P^k(a)(x)|  \leq & \frac{C}{\mu(B(x,\sqrt{k}))}e^{-cd(x,x_B)^2/k}\sum_{d(y,x_B)< r_B}\Big(\frac{d(y,x_B)}{\sqrt{k}}\Big)^h|a(y)|\mu(y) \\
  & \leq \frac{C}{\mu(B(x,\sqrt{k}))}e^{-cd(x,x_B)^2/k}\Big(\frac{r_B}{\sqrt{k}}\Big)^h \frac{\mu(B)}{\|\chi_{B}\|_{p(\cdot)}}.
\end{align*}

We define $D_k(x)=0$, when $r_B+d(x,x_B)<\sqrt{k}$, and $D_k(x)=D$, when $r_B+d(x,x_B)\ge\sqrt{k}$. It follows that, for every $k\in \mathbb{N}_+$,
\begin{align*}
|&P^k(a)(x)| \le \frac{C}{\mu(B(x,\sqrt{k}))}e^{-cd(x,x_B)^2/k}\Big(\frac{r_B}{\sqrt{k}}\Big)^h \frac{\mu(B)}{\|\chi_{B}\|_{p(\cdot)}}\\
&= C\frac{e^{-cd(x,x_B)^2/k}}{\|\chi_{B}\|_{p(\cdot)}}\Big(\frac{r_B}{\sqrt{k}}\Big)^h\frac{\mu(B(x,r_B+d(x,x_B)))}{\mu(B(x_B,\sqrt{k}))}\frac{\mu(B)}{\mu(B(x,r_B+d(x,x_B)))}\\
&\le C\frac{e^{-cd(x,x_B)^2/k}}{\|\chi_{B}\|_{p(\cdot)}}\Big(\frac{r_B+d(x,x_B)}{\sqrt{k}}\Big)^{D_k(x)}\frac{\mu(B)}{\mu(B(x,r_B+d(x,x_B)))}\\
&\le C\frac{e^{-cd(x,x_B)^2/k}}{\|\chi_{B}\|_{p(\cdot)}}\Big(\frac{r_B+d(x,x_B)}{\sqrt{k}}\Big)^{D_k(x)+h}\Big(\frac{r_B}{r_B+d(x,x_B)}\Big)^h\frac{\mu(B)}{\mu(B(x,r_B+d(x,x_B)))}\\
&\le \frac{C}{\|\chi_{B}\|_{p(\cdot)}}\Big(\frac{\mu(B)}{\mu(B(x_B,r_B+d(x,x_B))}\Big)^{h/D}\frac{\mu(B)}{\mu(B(x,r_B+d(x,x_B)))}\\
&\le \frac{C}{\|\chi_{B}\|_{p(\cdot)}}\Big(\frac{\mu(B)}{\mu(B(x_B,r_B+d(x,x_B))}\Big)^{1+h/D}\\
&\le \frac{C}{\|\chi_{B}\|_{p(\cdot)}}\Big(\mathcal{M}(\chi_B)(x)\Big)^{1+h/D},\,\,\,x\in \Gamma.
\end{align*}
We have used in the last line that $B\subset B(x,r_B+d(x,x_B))$.

Since $|P^k(a)|\le C\mathcal{M}(a)$, $k\in \mathbb{N}$, the above estimations lead to
\begin{align*}
    \mathcal{M}_+(a)(x)
    \leq & C\Big(\mathcal{M}(a)(x)\chi_{B(x_B,2r_B)}(x)\\
    &+\Big(\mathcal{M}(\chi_{B})(x)\Big)^{1+h/D}\chi_{B(x_B,2r_B)^c}(x)\|\chi_{B}\|_{p(\cdot)}^{-1}\Big), \;\;\;\;x\in\Gamma.
\end{align*}
We can write
\begin{align*}
    \mathcal{M}_+(f)(x)
  & \leq  C\Big(\sum_{j=0}^\infty\lambda_j\mathcal{M}(a_j)(x)\chi_{B(x_{B_j},2r_{B_j})}(x)   \\
& \qquad + \sum_{j=0}^\infty\frac{\lambda_j}{\|\chi_{B_j}\|_{p(\cdot)}}\Big(\mathcal{M}(\chi_{B_j})(x)\Big)^{1+h/D}\chi_{B(x_{B_j},2r_{B_j})^c}\Big) \\
  & \leq C\Big(\Big(\sum_{j=0}^\infty\Big(\lambda_j\mathcal{M}(a_j)(x)\chi_{B_j})(x)\Big)^{\mathfrak{p}}\Big)^{1/\mathfrak{p}}  \\
  & \qquad  + \Big(\sum_{j=0}^\infty\Big(\frac{\lambda_j}{\|\chi_{B_j}\|_{p(\cdot)}}\Big(\mathcal{M}(\chi_{B_j})(x)\Big)^{1+h/D}\Big)^{\mathfrak{p}}\Big)^{1/\mathfrak{p}}\Big).
\end{align*}
Let $j\in \mathbb{N}$. It is clear that
$\mbox{supp}\;(\mathcal{M}(a_j)\chi_{B(x_{B_j},2r_{B_j})})\subset
B(x_{B_j},2r_{B_j})$. Also, according to Lemma \ref{LemaCoc}, $(i)$,
there exists $C>0$ such that,
\begin{align*}
 \|\mathcal{M}(a_j)\chi_{B(x_{B_j},2r_{B_j})}\|_{2}&  \leq \|a_j\|_{2} \\
 &\leq (\mu(B_j))^{1/2}\|\chi_{B_j}\|_{p(\cdot)}^{-1}\\
 &\le C (\mu(B(x_{B_j},2r_{B_j})))^{1/2}\|\chi_{B(x_{B_j},2r_{B_j})}\|_{p(\cdot)}^{-1}.
\end{align*}

By Lemmas \ref{LemaSum}, \ref{LemaCoc}, $(i)$, and  \ref{LemaSumCoc}
we deduce that
$$
   \Big\| \Big(\sum_{j=0}^\infty\Big(\lambda_j\mathcal{M}(a_j)\chi_{B(x_j,2r_j)}\Big)^{\mathfrak{p}}\Big)^{1/\mathfrak{p}}\Big\|_{p(\cdot)}
  \leq  C\mathcal{A}(\{\lambda_j\},\{B_j\}).
$$

%  \Big\| \Big(\sum_{j=0}^\infty\Big(\lambda_j\frac{\chi_{B(x_j,2r_j)}}
%  {\|\chi_{B(x_j,2r_j)}\|_{p(\cdot)}}\Big)^{\mathfrak{p}}\Big)^{1/\mathfrak{p}}\Big\|_{p(\cdot)} \\
%   &\leq C \Big\| \Big(\sum_{j=0}^\infty\Big(\frac{\lambda_j\chi_{B_j}}
%  {\|\chi_{B_j}\|_{p(\cdot)}}\Big)^{\mathfrak{p}}\Big)^{1/\mathfrak{p}}\Big\|_{p(\cdot)}.
%\end{align*}
By using Lemma \ref{LemaFS}  we obtain
\begin{align*}
  & \Big\| \Big(\sum_{j=0}^\infty\Big(\frac{\lambda_j}{\|\chi_{B_j}\|_{p(\cdot)}}\Big(\mathcal{M}(\chi_{B_j})\Big)^{h/D+1}\Big)^{\mathfrak{p}}\Big)^{1/\mathfrak{p}}\Big\|_{p(\cdot)} \\
 & \qquad =  \Big\| \Big(\sum_{j=0}^\infty\Big(\mathcal{M}\Big(\Big(\frac{\lambda_j}{\|\chi_{B_j}\|_{p(\cdot)}}\Big)^{\frac{D}{D+h}}
  \chi_{B_j}\Big)\Big)^{\mathfrak{p}\frac{D+h}{D}}\Big)^{1/\mathfrak{p}}\Big\|_{p(\cdot)} \\
  & \qquad =  \Big\| \Big(\sum_{j=0}^\infty\Big(\mathcal{M}\Big(\Big(\frac{\lambda_j}{\|\chi_{B_j}\|_{p(\cdot)}}\Big)^{\frac{D}{D+h}}
  \chi_{B_j}\Big)\Big)^{\mathfrak{p}\frac{D+h}{D}}\Big)^{\frac{D}{(D+h)\mathfrak{p}}}\Big\|_{\frac{D+h}{D}p(\cdot)}^{\frac{D+h}{D}} \\
 & \qquad \leq C \Big\| \Big(\sum_{j=0}^\infty\Big(\frac{\lambda_j\chi_{B_j}}
  {\|\chi_{B_j}\|_{p(\cdot)}}\Big)^{\mathfrak{p}}\Big)^{1/\mathfrak{p}}\Big\|_{p(\cdot)},
\end{align*}
because $\mathfrak{p} >\frac{D}{D+h}$.\\

We conclude that
$$\|\mathcal{M}_+(f)\|_{p(\cdot)}\leq C\Big\| \Big(\sum_{j=0}^\infty\Big(\frac{\lambda_j\chi_{B_j}}
  {\|\chi_{B_j}\|_{p(\cdot)}}\Big)^{\mathfrak{p}}\Big)^{1/\mathfrak{p}}\Big\|_{p(\cdot)}.$$
\end{proof}

 \begin{Rem}\label{RemGmax3}
In \cite{Yan} (see also \cite{DP}) Hardy spaces associated with
operators are compared with the classical ones. In \cite[Theorem
6.1]{Yan} it is proved that imposing certain conditions on the heat
kernel is sufficient in order that Hardy operators, associated with
the operator they are dealing with,  coincide with the classical
ones. In Yang and Zhou obtained in \cite[Theorem 5.3]{YZ} a version
of \cite[Theorem 6.1]{Yan} in the variable exponent setting. As far
as we know this question has not been investigated for Hardy spaces
associated with operators on graphs even for constant exponents. Our
purpose is to study this question in a foregoing paper.
 \end{Rem}

 \section{Applications}

 In this section we present some applications of the atomic and molecular characterizations proved in the previous section. We establish $H^{p(\cdot)}_L(\Gamma)$-boundedness properties of some square functions, Riesz transforms, and L-spectral multipliers.

 \subsection{Square functions}

 Assume that $N\in \mathbb{N}_+$. The square function $G_{L,N}$ is defined as follows:
 $$
 G_{L,N}(f)(x)=\Big(\sum_{k=1}^\infty \frac{ |k^NL^NP^k(f)(x)|^2}{k}\Big)^{1/2},\,\,\,x\in
 \Gamma,
 $$
 where $f$ is complex function defined on $\Gamma$.

 The sublinear operator $G_{L,N}$ is bounded from $L^r(\Gamma)$ into itself, for every $1<r<\infty$ \cite[Proposition 4.6]{B} (see also \cite{BM}, \cite{BR}, and \cite{Fe}),
 and from $H^r_L(\Gamma)$ into $L^r(\Gamma)$, for every $0<r\le 1$ (\cite[Theorem 4.1]{B} and \cite[Theorem 6.1]{BD}). We now extend \cite[Theorem 6.1]{BD} to our variable exponent setting.

\begin{Prop} Let $p\in \mathcal{P}^{\log}(\Gamma)$ and $N\in \mathbb{N}_+$. Then, the square function $G_{L,N}$ can be extended from $H^{p(\cdot)}_L(\Gamma)\cap L^2(\Gamma)$ to $H^{p(\cdot)}_L(\Gamma)$ as a bounded operator from $H^{p(\cdot)}_L(\Gamma)$ to $L^{p(\cdot)}(\Gamma)$.
\end{Prop}
 \begin{proof} We are going to see that there exists $C>0$ such that
 $$
 \|G_{L,N}(f)\|_{p(\cdot)}\le C\|f\|_{H^{p(\cdot)}_L(\Gamma)},\,\,\,f\in H^{p(\cdot)}_L(\Gamma)\cap L^2(\Gamma).
 $$
 Assume that $f\in H^{p(\cdot)}_L(\Gamma)\cap L^2(\Gamma)$, $M\in \mathbb{N}$, $M>D/p_-$, and $q=2\max\{2,p_+\}$.
 According to Proposition \ref{Hardy}, there exist, for every $j\in \mathbb{N}$, $\lambda_j\in \mathbb{C}$ and a $(q,p(\cdot),M)$-atom $a_j$ associated
 to the ball $B_j=B(x_{B_j},r_{B_j})$, with $x_{B_j}\in\Gamma$ and $r_{B_j}\ge 1$, such that $f=\sum_{j=0}^\infty\lambda_ja_j$, where the series converges
 in $L^2(\Gamma)$ and in $H^{p(\cdot)}_L(\Gamma)$, and such that $\mathcal{A}(\{\lambda_j\},\{B_j\})\le C\|f\|_{H^{p(\cdot)}_L(\Gamma)}$. Here $C>0$ does not depend on $f$.

 Our objective is to see that
 $$
 \|G_{L,N}(f)\|_{p(\cdot)}\le C\mathcal{A}(\{\lambda_j\},\{B_j\}),
 $$
 for a certain $C>0$ independent of $f$. In order to do this we adapt the arguments in the proof of Proposition \ref{AD}.

 As above if $B=B(x_B,r_B)$, with $x_B\in\Gamma$ and $r_B\ge 1$, we denote $S_0(B)=B(x_B,8Mr_B)$ and, for every $i\in \mathbb{N}_+$, $S_i(B)=B(x_B,2^{i+3}Mr_B)\setminus B(x_B,2^{i+2}Mr_B)$. Since $P$ is bounded in $L^2(\Gamma)$ we can write
 $$
 (I-P)^NP^k(f)=\sum_{j=0}^\infty \lambda_j(I-P)^NP^k(a_j), \,\,\rm{in}\,\,\,L^2(\Gamma),
 $$
  and then
 $$
 (I-P)^NP^k(f)(x)=\sum_{j=0}^\infty \lambda_j(I-P)^NP^k(a_j)(x), \,\,x\in \Gamma,
 $$
 for every $k\in \mathbb{N}_+$. It follows that
 \begin{align*}
 \|G_{L,N}(f)\|_{p(\cdot)}&\le \Big\|\sum_{i=0}^\infty \sum_{j=0}^\infty |\lambda_j|G_{L,N}(a_j)\chi_{S_i(B_j)}\Big\|_{p(\cdot)}\nonumber\\
 &\le \Big(\sum_{i=0}^\infty \Big\|\Big(\sum_{j=0}^\infty ( |\lambda_j|G_{L,N}(a_j)\chi_{S_i(B_j)})^\mathfrak{p}\Big)^{1/\mathfrak{p}}\Big\|_{p(\cdot)}^\mathfrak{p}\Big)^{1/\mathfrak{p}}.
 \end{align*}

 Since $G_{L,N}$ is bounded in $L^q(\Gamma)$ we deduce that
 $$
 \|G_{L,N}(a_j)\|_{L^q(S_0(B_j))}\le C\|a_j\|_q\le C\mu(B_j)^{1/q}\|\chi_{B_j}\|_{p(\cdot)}^{-1},\,\,\,j\in \mathbb{N}.
 $$
 By taking into account Lemmas \ref{LemaSum}, \ref{LemaCoc}, (i), and \ref{LemaSumCoc} we get
 $$
 \Big\|\Big(\sum_{j=0}^\infty ( |\lambda_j|G_{L,N}(a_j)\chi_{S_0(B_j)})^\mathfrak{p}\Big)^{1/\mathfrak{p}}\Big\|_{p(\cdot)}\le \Big\|\Big(\sum_{j=0}^\infty ( \frac{|\lambda_j|\chi_{S_0(B_j)}}{\|\chi_{B_j}\|_{p(\cdot)}}\Big)^\mathfrak{p}\Big)^{1/\mathfrak{p}}\Big\|_{p(\cdot)}\le C\mathcal{A}(\{\lambda_j\},\{B_j\}).
$$

 Let $j\in \mathbb{N}$ and $i\in \mathbb{N}_+$. We have that $a_j=L^Mb_j$, where $b_j\in L^q(\Gamma)$
satisfying that, for every $\ell=0,...,M$,

 %As in \cite[p. 3478]{BD} we decompose $G_{L,N}(a_j)^2$ as follows
% $$
% G_{L,N}(a_j)^2=I_{j,1}+I_{j,2},
% $$
% where
% $$
% I_{j,1}(x)=\sum_{k=1}^{r_{B_j}^2} \frac{ |k^NL^NP^k(f)(x)|^2}{k},\,\,\,x\in \Gamma.
% $$
% Assume that $i\in \mathbb{N}_+$. According to (\ref{compuesto}) we obtain, for every $k\in \mathbb{N}_+$, $k\le r_{B_j}^2$ and $x\in S_i(B_j)$,
%\begin{align*}
%|k^N(I-P)^NP^k(f)(x)|&\le C\sum_{y\in B_j}\frac{e^{-cd(x,y)^2/k}}{\mu(B(x,\sqrt{k}))}|a_j(y)|\mu(y)\\
%&\le \frac{C}{\mu(B(x,\sqrt{k}))}\sum_{y\in B_j}e^{-cd(x,x_{B_j})^2/k}|a_j(y)|\mu(y)\\
%&\le \frac{C}{\mu(B(x,d(x,x_{B_j}))}\Big(\frac{d(x,x_{B_j})}{r_{B_j}}\Big)^De^{-cd(x,x_{B_j})^2/k}\sum_{y\in B_j}|a_j(y)|\mu(y)\\
%&\le \frac{C}{\mu(B(x,d(x,x_{B_j}))}e^{-c2^{2i}r_{B_j}^2/k}\sum_{y\in B_j}|a_j(y)|\mu(y)\\
%&\le \frac{C}{\mu(B(x_{B_j},2^iMr_{B_j}))}\|a\|_q\mu(B_j)^{1-1/q}e^{-c2^{2i}r_{B_j}^2/k}\\
%&\le C\frac{\mu(B_j)}{\mu(B(x_{B_j},2^iMr_{B_j}))}e^{-c2^{2i}r_{B_j}^2/k}\|\chi_{B_j}\|_{p(\cdot)}^{-1}\\
%&\le C\frac{\mu(B_j)}{\mu(B(x_{B_j},2^iMr_{B_j}))}e^{-c2^{2i}}\frac{\sqrt{k}}{r_{B_j}}\|\chi_{B_j}\|_{p(\cdot)}^{-1}.
%\end{align*}

(i) $\supp(L^\ell b_j)\subset B_j$,

(ii) $\|L^\ell b_j\|_q\le
r_{B_j}^{M-\ell}\mu(B_j)^{1/q}\|\chi_{B_j}\|_{p(\cdot)}^{-1}$.

Since $d(x,z)\ge d(x,x_{B_j})/2$, when $z\in B_j$ and $x\in
S_i(B_j)$, (\ref{compuesto}) leads to
\begin{align*}
(G_{L,N}(a_j)(x))^2&=
\sum_{k=1}^\infty\frac{|k^N(I-P)^NP^k(a_j)(x)|^2}{k}\\
&=\sum_{k=1}^\infty k^{2N-1}|(I-P)^{N+M}P^k(b_j)(x)|^2\\
&=\sum_{k=1}^\infty k^{2N-1}\Big|\sum_{z\in
B_j}\widetilde{p}_{k,N+M}(x,z)b_j(z)\Big|^2\\
&\le C\sum_{k=1}^\infty k^{2N-1}\Big(\sum_{z\in
B_j}\frac{|b_j(z)|e^{-cd(x,z)^2/k}}{k^{N+M}\mu(B(x,\sqrt{k}))}\mu(z)\Big)^2\\
&\le
C\sum_{k=1}^\infty\frac{e^{-cd(x,x_{B_j})^2/k}}{k^{2M+1}\mu(B(x,\sqrt{k}))^2}\Big(\sum_{z\in
B_j}|b_j(z)|\mu(z)\Big)^2\\
&\le
C\sum_{k=1}^\infty\frac{e^{-c2^{2i}r_{B_j}^2/k}}{k^{2M+1}\mu(B(x,\sqrt{k}))^2}\mu(B_j)^{2(1-1/q)}\|b_j\|_q^2\\
&\le
C\frac{r_{B_j}^{2M}\mu(B_j)^2}{\|\chi_{B_j}\|_{p(\cdot)}^2\mu(B(x_{B_j},2^ir_{B_j}))^2}\sum_{k=1}^\infty\frac{e^{-c2^{2i}r_{B_j}^2/k}}{k^{2M+1}}\Big(\frac{2^ir_{B_j}}{\sqrt{k}}\Big)^{D_k},\,\,\,x\in
S_i(B_j),
\end{align*}
where $D_k=2D$, when $\sqrt{k}\le 2^{i+1}Mr_{B_j}$, and $D_k=0$,
when $\sqrt{k}> 2^{i+1}Mr_{B_j}$.

Then,
\begin{align*}
(G_{L,N}(a_j)(x))^2&\le C
\frac{r_{B_j}^{2M}\mu(B_j)^2}{\|\chi_{B_j}\|_{p(\cdot)}^2\mu(B(x_{B_j},2^ir_{B_j}))^2}\sum_{k=1}^\infty\frac{e^{-c2^{2i}r_{B_j}^2/k}}{k^{2M+1}}\\
&\le
C\frac{r_{B_j}^{2M}\mu(B_j)^2}{\|\chi_{B_j}\|_{p(\cdot)}^2\mu(B(x_{B_j},2^ir_{B_j}))^2}\frac{1}{(2^{2i}r_{B_j}^2)^M}\sum_{k=1}^\infty
k^{-(M+1)}\\
&\le
C\frac{\mu(B_j)^22^{-i2M}}{\|\chi_{B_j}\|_{p(\cdot)}^2\mu(B(x_{B_j},2^ir_{B_j}))^2},\,\,\,x\in
S_i(B_j).
\end{align*}

We can write
\begin{align*}
\|G_{L,N}(a_j)\|_{L^q(S_i(B_j))}&\le
C\frac{\mu(B_j)2^{-iM}}{\|\chi_{B_j}\|_{p(\cdot)}\mu(B(x_{B_j},2^ir_{B_j}))}\mu(B(x_{B_j},2^{i+3}Mr_{B_j}))^{1/q}\\
&\le
C\frac{2^{-iM}\mu(B(x_{B_j},2^{i+3}Mr_{B_j}))^{1/q}}{\|\chi_{B_j}\|_{p(\cdot)}}.
\end{align*}

According to Lemmas \ref{LemaSum}, \ref{LemaCoc}, (i), and \ref{LemaSumCoc}, we have that
\begin{align*}
\Big\|\Big(\sum_{j=0}^\infty ( |\lambda_j|G_{L,N}(a_j)\chi_{S_i(B_j)})^\mathfrak{p}\Big)^{1/\mathfrak{p}}\Big\|_{p(\cdot)}&\le C2^{-i(M-D/w)}\mathcal{A}(\{\lambda_j\},\{B(x_{B_j},2^{i+3}Mr_{B_j})\})\\
 &\le C2^{-i(M-2D/w)}\mathcal{A}(\{\lambda_j\},\{B_j\}),
 \end{align*}
where $0<w<p_-$.

By choosing $2D/M<w<p_-$, we conclude that
$$
\|G_{L,M}(f)\|_{p(\cdot)}\le C\Big(\sum_{i=1}^ \infty
2^{-i(M-2D/w)}\Big)^{1/\mathfrak{p}}\mathcal{A}(\{\lambda_j\},\{B_j\})\le
C\mathcal{A}(\{\lambda_j\},\{B_j\}).
$$

 \end{proof}

 \subsection{Spectral multipliers for the discrete Laplacian $L$}

 The discrete Laplacian $L$ is a positive operator in $L^2(\Gamma)$
 and $\|Lf\|_2\le 2\|f\|_2$, $f\in L^2(\Gamma)$. Then, the spectrum
 of $L$ is contained in $[0,2]$ and there exists an spectral
 measure $E_L$ such that $L=\int_{[0,2]}\lambda
 dE_L(\lambda)$. If $F$ is a complex bounded measurable function
 defined in $[0,2]$ the operator $F(L)$ defined by
 \begin{equation}\label{M3}
 F(L)=\int_{[0,2]}F(\lambda)dE_L(\lambda)
 \end{equation}
is bounded from $L^2(\Gamma)$ into itself. The operators $F(L)$ in
(\ref{M3}) are called spectral multipliers for the operator $L$. If
the graph $\Gamma$ satisfies the property $\Delta(\alpha)$, for some
$\alpha>0$, $-1$ is not in the spectrum of $P$ and then the exists
$a_0\in (0,2)$ such that the spectral measure $E_L$ is supported in
$[0,a_0]$ (see \cite[Lemma 1.3]{Dung}).

Our objective is to give conditions on $F$ that allow us to extend
the operator $F(L)$ from $L^2(\Gamma)\cap H^{p(\cdot)}_L(\Gamma)$ to
$H^{p(\cdot)}_L(\Gamma)$ as a bounded operator from
$H^{p(\cdot)}_L(\Gamma)$ into itself.

Let $s>0$. If $f:\mathbb{R}\longrightarrow \mathbb{C}$ is a
$C^{[s]}(\mathbb{R})$-function we define $\|f\|_{\mathcal{C}^s}$ by
$$
\|f\|_{\mathcal{C}^s}=\sum_{k=0}^{[s]}\|f^{(k)}\|_\infty+M_s(f),
$$
where
$$
M_s(f)=\sup\Big\{t^{[s]-s}|f^{([s])}(x+t)-f^{([s])}(x)|:t>0,\,\,x\in
\mathbb{R}\Big\}.
$$
The space $\mathcal{C}^s(\mathbb{R})$ is defined by
$$
\mathcal{C}^s(\mathbb{R})=\{f\in
\mathcal{C}^{[s]}(\mathbb{R}):\|f\|_{\mathcal{C}^s}<\infty\}.
$$
In order to get boundedness properties for spectral multipliers it
is usual to consider the following property. We say that a complex
bounded measurable function $F$ defined on $[0,\infty)$ satisfies
the property $R_s$ with $s>0$ when
$$
\sup_{t>0}\|\eta(\lambda)F(t\lambda)\|_{\mathcal{C}^s}<\infty.
$$
Here, $\eta$ is a fixed $C^\infty((0,\infty))$ with compact support
and it is not identically zero.

It is known that $F(L)$ can be extended from

(i) $L^p(\Gamma)\cap L^2(\Gamma)$ to $L^p(\Gamma)$ as a bounded
operator from $L^p(\Gamma)$ into itself when $F$ satisfies $R_s$
with $s>D/2$, for every $1<p<\infty$ (\cite[p. 266, (i)]{BD1}).

(ii) $H^p_L(\Gamma)\cap L^2(\Gamma)$ to $H^p_L(\Gamma)$ as a bounded
operator from $H^p_L(\Gamma)$ into itself when $F$ satisfies $R_s$
with $s>D(1/p-1/2)$, for every $0<p\le 1$ (\cite[Theorem 6.2]{BD} and \cite{KM}).

Next we generalize \cite[Theorem 6.2]{BD} to our variable exponent
setting.

\begin{Prop} Let $p\in \mathcal{P}^{\log}(\Gamma)$, $p_+<2$, and $s>2D/p_-$. Assume that $F$ is a complex
bounded measurable function on $(0,\infty)$ satisfying $R_s$. Then,
$F(L)$ is bounded from $H^{p(\cdot)}_L(\Gamma)$ into itself.
\end{Prop}

\begin{proof} We use the procedure in the proof of \cite[Theorem 6.2]{BD} (see also \cite{Alex})
with the necessary modifications due to the variable exponent
context. We consider $M\in \mathbb{N}$, $M>2D/p_-$. Suppose
that $f\in L^2(\Gamma)\cap H^{p(\cdot)}_L(\Gamma)$. According to
Proposition \ref{Hardy} there exists, for every $j\in \mathbb{N}$,
$\lambda_j\in \mathbb{C}$ and a $(2,p(\cdot),2M)$-atom associated to
the ball $B_j=B(x_{B_j},r_{B_j})$, with $x_{B_j}\in \Gamma$ and
$r_{B_j}\ge 1$, such that $f=\sum_{j=0}^\infty \lambda_ja_j$ in the
sense of convergence in $L^2(\Gamma)$ and in
$H^{p(\cdot)}_L(\Gamma)$, and
$$
\mathcal{A}(\{\lambda_j\},\{B_j\})\le
C\|f\|_{H^{p(\cdot)}_L(\Gamma)}.
$$
Since $F(L)$ is bounded in $L^2(\Gamma)$ we have that
$$
F(L)(f)=\sum_{j=0}^\infty \lambda_jF(L)(a_j), \,\,\,\rm{in}\,\,\,
L^2(\Gamma).
$$
We are going to see that there exists $C>0$ such that, for every
$j\in \mathbb{N}$, $CF(L)(a_j)$ is a
$(2,p(\cdot),M,\varepsilon)$-molecule for some $\varepsilon>D/p_-$. When
this fact is proved the proof finishes because from Theorem
\ref{Th1.4} it follows that $F(L)(f)\in H^{p(\cdot)}_L(\Gamma)$ and
$$
\|F(L)(f)\|_{H^{p(\cdot)}_L(\Gamma)}\le C
\mathcal{A}(\{\lambda_j\},\{B_j\}).
$$

Suppose that $a$ is a  $(2,p(\cdot),2M)$-atom associated to the ball
$B=B(x_B,r_B)$, with $x_B\in \Gamma$ and $r_B\ge 1$, and the
function $b\in L^2(\Gamma)$. It is clear that
$F(L)(a)=L^MF(L)(L^M(b))$. Since $L$ and $F(L)$ are bounded
operators in $L^2(\Gamma)$, the function $w=F(L)(L^M(b))\in
L^2(\Gamma)$. We have to prove that, for some $\varepsilon>D/p_-$,
$$
\|L^kw\|_{L^2(\mathfrak{S}_j(B))}\leq
(r_B)^{M-k}2^{-j\varepsilon}(\mu(B(x_B,2^jr_B)))^{1/2}\|\chi _{B(x_B,2^jr_B)}\|_{p(\cdot
)}^{-1},
$$
for every $ j\in \mathbb{N}$ and $k=0,...,M$, where $\mathfrak{S}_j(B)=B(x_B,2^{j+1}r_B)\setminus B(x_B,2^{j-1}r_B)$, $j\in
\mathbb{N}_+$ and $\mathfrak{S}_0(B)=B$.

Let $k=0,...,M$. By using the properties of the function $b$ and
once more that $L$ and $F(L)$ are bounded operators in $L^2(\Gamma)$ we
deduce that
$$
\|L^kw\|_{L^2(\mathfrak{S}_0(B))}\leq
C(r_B)^{M-k}(\mu(B))^{1/2}\|\chi
_{B}\|_{p(\cdot )}^{-1}.
$$

Assume that $j\in \mathbb{N}_+$. We choose $0\le
\theta\in C^\infty(\mathbb{R})$ such that
$$
\theta(t)=1,\,\,\,t\in
(-\infty,1]\,\,\,\rm{and}\,\,\,\theta(t)=0,\,\,\,t\in [3/2,\infty),
$$
and, for every $\ell\in \mathbb{N}_+$, we define $\psi_\ell\in C^\infty(\mathbb{R})$ satisfying that
$$
\psi_\ell(t)=0,\,\,\,t\notin [2^{-\ell-1},3\times 2^{-\ell-1}]\,\,\,\rm{and}\,\,\,0<\psi_\ell(t)\le 1,\,\,\,t\in [2^{-\ell-1},3\times 2^{-\ell-1}],
$$
and
$$
\sum_{\ell=1}^\infty
\psi_\ell(t)=1, \,\,\,t\in (0,3/2).
$$
We can write the following decomposition \cite[p. 840]{B} (see also
\cite{Alex})
$$
F(\lambda)=\sum_{\ell=1}^\infty F(\lambda)\theta(\lambda)\psi_\ell(
\lambda)+(1-\theta(\lambda))F(\lambda)=\sum_{\ell=1}^\infty
F_\ell(\lambda)+F_0(\lambda),\,\,\,\lambda\in (0,2].
$$
It follows that
$$
F(L)=\sum_{\ell=0}^{\infty} F_\ell(L),
$$
where the convergence is understood in the space of operators in
$L^2(\Gamma)$. In order to prove this we consider firstly a complex valued function $f$ defined on $\Gamma$ with support in a finite set $\Omega$. According to (\ref{UE}) we have that, for every $n\in \mathbb{N}_+$,
\begin{align*}
|P^n(f)(x)|&\le C\sum_{y\in \Omega}e^{-d(x,y)^2/n}|f(y)|\frac{\mu(y)}{\mu(B(x,\sqrt{n}))}\\
&\le C\frac{1}{\mu(B(x,\sqrt{n}))},\,\,\,x\in \Gamma.
\end{align*}
Then, since $\mu(\Gamma)=\infty$, $\lim_{n\to\infty}P^n(f)(x)=0$, for every $x\in \Gamma$. By using monotone convergence theorem we deduce that
$$
\overline{\lim}_{n\to\infty}\|P^n(f)\|_2^2\le \lim_{k\to\infty}\sum_{x\in\Gamma}\sup_{n\ge k}|P^n(f)(x)|^2\mu(x)=0.
$$
Hence $\lim_{n\to\infty}\|P^n(f)\|_2=0$. A density argument allows us to conclude that $\lim_{n\to\infty}\|P^n(f)\|_2=0$, for every $f\in L^2(\Gamma)$.

As in \cite[p. 839]{Alex2} we can see that
$$
\Big\|\int_{[0,1/n]}dE_L(\lambda)\Big\|_2\to 0,\,\,\,\rm{as}
\,\,\,n\to\infty.
$$
Let $\eta>0$. There exists $\ell_0\in \mathbb{N}$ such that
$$
\Big\|\int_{[0,2^{-\ell}]}F(\lambda)dE_L(\lambda)\Big\|_2\le\|F(L)\|_2\Big\|\int_{[0,2^{-\ell}]}dE_L(\lambda)\Big\|_2<\eta,\,\,\,\ell\in
\mathbb{N},\,\,\,\ell\ge\ell_0.
$$
We have that
$$
\Big\|\int_{[0,2]}(F(\lambda)-\sum_{\ell=0}^{\ell_0}F_\ell(\lambda))dE_L(\lambda)\Big\|_2\le
\Big\|\int_{[0,2^{-\ell_0-1}]}F(\lambda)dE_L(\lambda)\Big\|_2<\eta.
$$

Then,
$$
F(L)(a)=\sum_{\ell=0}^{\infty} F_\ell(L)(a)= L^MW,
$$
being $W=\sum_{\ell=0}^{\infty} F_\ell(L)(L^Mb)$.

As in \cite[(19) and (20)]{BD} we get
$$
\|L^k(F_0(L)L^Mb)\|_{L^2(\mathfrak{S}_j(B))}\le
Cr_B^{M-k}2^{-sj}r_B^{-(s-D/2)}(\mu(B))^{1/2}\|\chi_{B}\|_{p(\cdot)}^{-1}.
$$
Then, since $s>2D/p_-$, $p_-<2$ and according to Lemma \ref{LemaCoc}, (i), if $0<w<p_-$ we obtain
\begin{equation}\label{1111}
\|L^k(F_0(L)L^Mb)\|_{L^2(\mathfrak{S}_j(B))}\le
Cr_B^{M-k}2^{-(s-D/w)j}\mu(B(x_B,2^jr_B))^{1/2}\|\chi_{B(x_B,2^jr_B)}\|_{p(\cdot)}^{-1}\\
\end{equation}

 Let $\ell\in \mathbb{N}_+$. We take $0<w<p_-$ such that $\delta=s-D/w>0$. We get
\begin{align}\label{MM7}
&\|L^kF_\ell(L)L^Mb\|_{L^2(\mathfrak{S}_j(B))}\nonumber\\
&\hspace{3cm}\le C\Big(\Big(\frac{2^jr_B}{2^\ell}\Big)^{M+k}2^{-j(M+k)}\chi_{[j+\log_2(r_B),+\infty)}(\ell)r_B^{M-k}\mu(B)^{1/2}\|\chi_B\|_{p(\cdot)}^{-1}\nonumber\\
&\hspace{4cm}+2^{-j\delta}\chi_{(0,j+\log_2(r_B))}(\ell)(2^jr_B)^{-s}r_B^{M-k}\mu(B)^{1/2}\|\chi_B\|_{p(\cdot)}^{-1}\Big)\nonumber\\
&\hspace{3cm}\le Cr_B^{M-k}\mu(B(x_B,2^jr_B))^{1/2}\|\chi_{B(x_B,2^jr_B)}\|_{p(\cdot)}^{-1}2^{jD/w}\nonumber\\
&\hspace{4cm}\times(2^{-j(M+k)}\chi_{[j+\log_2(r_B),+\infty)}(\ell)+2^{-j\delta}\chi_{(0,j+\log_2(r_B))}(\ell))\nonumber\\
&\hspace{3cm}\le C2^{-\varepsilon
j}r_B^{M-k}\mu(B(x_B,2^jr_B))^{1/2}\|\chi_{B(x_B,2^jr_B)}\|_{p(\cdot)}^{-1},
\end{align}
where $0<\varepsilon<\min\{M-D/w,s-D/w\}$ and $\varepsilon>D/p_-$. Note that $\varepsilon$ can be chosen in this way.

By combining (\ref{1111}) and (\ref{MM7}) we obtain that
$$
\|L^kW\|_{L^2(\mathfrak{S}_j(B))}\le C2^{-\varepsilon
j}r_B^{M-k}\mu(B(x_B,2^jr_B))\|\chi_{B(x_B,2^jr_B)}\|_{p(\cdot)}^{-1}.
$$
Thus the proof is finished.
\end{proof}

\subsection{Riesz transforms associated with $L$ on $\Gamma$}

The gradient $\nabla f$ of a complex function $f$ defined on
$\Gamma$ is defined by
$$
\nabla f(x)=\Big(\frac{1}{2}\sum_{y\in
\Gamma}p(x,y)|f(x)-f(y)|^2\Big)^{1/2},\,\,\,x\in \Gamma.
$$
It is clear that $\nabla$ is a sublinear operator.

By using spectral theory we can define, for every $\alpha\in \mathbb{R}$, the operator $L^\alpha$  as follows
$$
L^\alpha=\int_{[0,2]}\lambda^\alpha dE_L(\lambda).
$$
If $\alpha\ge 0$, then $L^\alpha$ is a bounded operator in $L^2(\Gamma)$. However, if $\alpha<0$, the operator $L^\alpha$ is unbounded in $L^2(\Gamma)$ and the domain $D(L^\alpha)$ of $L^\alpha$ is given by
$$
D(L^\alpha)=\{f\in L^2(\Gamma):\, \int_{[0,2]}\lambda^{2\alpha}d\langle E_L(\lambda)f,f\rangle <\infty\}.
$$
The square root $L^{1/2}$ of $L$ is bounded in $L^2(\Gamma)$ and $\|L^{1/2}f\|_2=\|\nabla f\|_2$, $f\in L^2(\Gamma)$. Thus the equality can be rewritten $\|\nabla L^{-1/2}f\|_2=\|f\|_2$, $f\in R_2(L^{1/2})$, where $R_2(L^{1/2})$ denotes the range of $L^{1/2}$ in $L^2(\Gamma)$. $R_2(L^{1/2})$ is dense in $L^2(\Gamma)$.

We define the Riesz transform $R_L$ associated with $L$ on $\Gamma$ by
$$
R_L=\nabla L^{-1/2}.
$$
$R_L$ can be extended from $R_2(L^{1/2})$ to $L^2(\Gamma)$ as an isometry in $L^2(\Gamma)$.

In \cite[Lemma 1.13]{BR} it was proved that, if $f\in R_2(L^{1/2})$ then
$$
\nabla\Big(\sum_{k=0}^n \beta_kP^kf\Big)\to \nabla L^{-1/2}(f), \,\,\,\rm{as}\,\,\,n\to\infty,
$$
in $L^2(\Gamma)$, where the sequence $\{\beta_k\}_{k\in \mathbb{N}}$ is defined by the expansion
$$
(1-z)^{-1/2}=\sum_{k=0}^\infty \beta_kz^k.
$$
Hence, we have that
$$
R_Lf=\lim_{n\to\infty}\nabla\Big(\sum_{k=0}^n \beta_kP^kf\Big), \,\,\,f\in R_2(L^{1/2}).
$$

If the properties (a), (c), and (\ref{UE}) for $x=y$ hold for
$(\Gamma,\mu,d)$, $R_L$ can be extended from $L^2(\Gamma)\cap
L^q(\Gamma)$ to $L^q(\Gamma)$ as a bounded operator from
$L^q(\Gamma)$ into itself, for every $1<q\le 2$, and from
$L^1(\Gamma)$ into $L^{1,\infty}(\Gamma)$, the weak $L^1(\Gamma)$
(see \cite{Ru3}). In order that $R_L$ can be extended from
$L^2(\Gamma)\cap L^q(\Gamma)$ to $L^q(\Gamma)$ as a bounded operator
from $L^q(\Gamma)$ into itself when $q>2$ it is not sufficient with
the properties (a), (c), and (\ref{UE}) for $x=y$, as the example of
two copies of $\mathbb{Z}^2$ linked between with an edge shows
\cite[Section 4]{Ru3}. In \cite[Theorem 1.3]{BR} it is established
that if $(\Gamma,\mu,d)$ satisfies (a), (c), Poincar\'e inequality and
there exists $C>0$ such that, for every $f\in L^{q_0}(\Gamma)$ and
$k\in \mathbb{N}_+$,
$$
\|\nabla P^k(f)\|_{q_0}\le \frac{C}{\sqrt{k}}\|f\|_{q_0},
$$
being $q_0>2$, then $R_L$ can be extended from $L^2(\Gamma)\cap L^q(\Gamma)$ to $L^q(\Gamma)$ as a bounded operator from $L^q(\Gamma)$ into itself when $2\le q<q_0$. Weighted $L^q$-inequalities for the Riesz transform $R_L$ were established by Badr and Martell \cite{BM}.

Boundedness properties for $R_L$ in Hardy spaces $H^q(\Gamma)$,
$0<q\le 1$, were proved in \cite[Theorem 4.2]{B} and \cite[Theorem
6.5]{BD}. In the next proposition we extend those results in
\cite{B} and \cite{BD} to our variable exponent Hardy spaces
$H^{p(\cdot)}_L(\Gamma)$.

\begin{Prop} Let $p\in \mathcal{P}^{\log}(\Gamma)$ such that $p_+<2$. Then, $R_L$ can be extended from $L^2(\Gamma)\cap H^{p(\cdot)}_L(\Gamma)$ to $H^{p(\cdot)}_L(\Gamma)$ as a bounded operator from $H^{p(\cdot)}_L(\Gamma)$
into $L^{p(\cdot)}(\Gamma)$.
\end{Prop}

\begin{proof} Suppose that $f\in H^{p(\cdot)}_L(\Gamma)\cap L^2(\Gamma)$. Let $M\in \mathbb{N}$ and $M>2D/p_-$. According to Proposition \ref{Hardy}, for every $j\in \mathbb{N}$, there exist $\lambda_j\in \mathbb{C}$ and a $(2,p(\cdot),M)$-atom $a_j$ associated with the ball $B_j=B(x_{B_j},r_{B_j})$, such that $f=\sum_{j=0}^\infty \lambda_ja_j$ in $L^2(\Gamma)$ and in $H^{p(\cdot)}_L(\Gamma)$, and $\mathcal{A}(\{\lambda_j\},\{B_j\})\le C\|f\|_{H^{p(\cdot)}_L(\Gamma)}$. Since $R_L$ is bounded in $L^2(\Gamma)$ we have that
$$
R_L(f)=\lim_{\ell\to\infty} R_L\Big(\sum_{j=0}^\ell \lambda_ja_j\Big)\le \sum_{j=0}^\infty |\lambda_j|R_{L}(a_j).
$$
Also, we can write
\begin{align*}
\|R_L(f)\|_{p(\cdot)}&\le \Big\|\sum_{j=0}^\infty\sum_{i=0}^\infty |\lambda_j|R_L(a_j)\chi_{S_i(B_j)}\Big\|_{p(\cdot)}\\
&\le \Big(\sum_{i=0}^\infty\Big\|\Big(\sum_{j=0}^\infty
\Big(|\lambda_j|R_L(a_j)\chi_{S_i(B_j)}\Big)^\mathfrak{p}\Big)^{1/\mathfrak{p}}\Big\|_{p(\cdot)}^\mathfrak{p}\Big)^{1/\mathfrak{p}}.
\end{align*}
Since $R_L$ is bounded in $L^2(\Gamma)$ it follows that
$$
\|R_L(a_j)\|_{L^2(S_0(B_j))}\le C\|a_j\|_2\le C\mu(B_j)^{1/2}\|\chi_{B_j}\|_{p(\cdot)}^{-1},\,\,\,j\in \mathbb{N}.
$$
According to Lemmas \ref{LemaSum}, \ref{LemaCoc}, (i), and \ref{LemaSumCoc} we deduce that
$$
\Big\|\Big(\sum_{j=0}^\infty \Big(|\lambda_j|R_L(a_j)\chi_{S_0(B_j)}\Big)^\mathfrak{p}\Big)^{1/\mathfrak{p}}\Big\|_{p(\cdot)}^\mathfrak{p}\le C\mathcal{A}(\{\lambda_j\},\{B_j\}).
$$
Assume now that $a$ is a $(2,p(\cdot),M)$-atom associated with the ball $B=B(x_B,r_B)$ and the function $b\in L^2(\Gamma)$. Since $a=L^Mb=L^{1/2}L^{M-1/2}b\in R_2(L^{1/2})$ we can write
\begin{align*}
R_L(a)(x)&=\lim_{n\to\infty}\nabla\Big(\sum_{k=0}^n\beta_kP^k(a)\Big)(x)\\
&\le \sum_{k=0}^\infty |\beta_k|\nabla(P^ka)(x)\\
&=\sum_{k=0}^{r_B^2}|\beta_k|\nabla(P^ka)(x)+\sum_{k>r_B^2}|\beta_k|\nabla(L^MP^kb)(x)\\
&\le \sum_{k=0}^{r_B^2}|\beta_k|\sum_{y\in B}\nabla_xp_k(x,y)|a(y)|+\sum_{k>r_B^2}|\beta_k|\sum_{y\in B}\nabla_x\widetilde{p_{k,M}}(x,y)|b(y)|\\
&=I_1(x)+I_2(x),\,\,\,x\in \Gamma,
\end{align*}
where $\widetilde{p_{k,M}}$ represents the kernel of the operator $L^MP^k$, for every $k\in \mathbb{N}$. Note that, for every $k\in \mathbb{N}$, $\supp(P^ka)\subset B(x_B,r_B+k)$. Then, $\supp(\nabla P^k(a))\subset B(x_B,r_B+k+1)$. We have that
$$
I_1(x)=\sum_{7Mr_B<k<r_B^2}|\beta_k|\sum_{y\in B}\nabla_xp_k(x,y)|a(y)|,\,\,\,x\in \Gamma\setminus S_0(B).
$$
By \cite[Lemma 6.6]{BD}, Lemma \ref{LemaCoc}, (i), and by proceeding as in \cite[p. 838]{B} we get, for every $i\in \mathbb{N}_+$,
$$
\|I_1\|_{L^2(S_i(B))}+\|I_2\|_{L^2(S_i(B))}\le C2^{-i(M-D/w)}\mu(B(x_B,2^{i+3}Mr_{B_j}))\|\chi_{B(x_B,2^{i+3}Mr_{B_j})}\|_{p(\cdot)}^{-1}.
$$
Here $0<w<p_-$. From Lemma \ref{LemaSumCoc}, for every $i\in
\mathbb{N}_+$, we deduce that
$$
\Big\|\Big(\sum_{j=0}^\infty
\Big(|\lambda_j|R_L(a_j)\chi_{S_i(B_j)}\Big)^\mathfrak{p}\Big)^{1/\mathfrak{p}}\Big\|_{p(\cdot)}^\mathfrak{p}\le
C2^{-i(M-2D/w)}\mathcal{A}(\{\lambda_j\},\{B_j\}).
$$
Then, since $M>2D/p_-$ we conclude that
$$
\|R_Lf\|_{p(\cdot)}\le C\mathcal{A}(\{\lambda_j\},\{B_j\}),
$$
and the proof is finished.
\end{proof}

%%%%%%%%%%%%%%%%%%%%%%%%%%%%%%%%%%%%%%%%%%%%%%%%%%%%%%%%%%%%%%%%%%%
%\AJC{Esto es s\'olo para que se muestren las referencias y comprobar que est\'an correctas.}
%\red{Referencias:
%\cite{B},
%\cite{BD},
%\cite{CMS},
%\cite{CrW},
%\cite{Ru},
%\cite{ZYL},
%\cite{Ne},
%\cite{Sa},
%\cite{HHP},
%\cite{AHH},
%\cite{DHHR},
%\cite{JY},
%\cite{YZ2},
%\cite{MS},
%\cite{YZ},
%\cite{HMY},
%\cite{GLY},
%\cite{ZSY},
%\cite{MS0},
%\cite{Ru1},
%\cite{DP},
%\cite{Fe},
%\cite{Yan},
%\cite{Ru2},
%\cite{SL},
%\cite{ZY},
%\cite{ERZ},
%\cite{DKKP},
%}
%%\bibliographystyle{siam}
%%\bibliography{references}

\def\cprime{$'$} \def\ocirc#1{\ifmmode\setbox0=\hbox{$#1$}\dimen0=\ht0
  \advance\dimen0 by1pt\rlap{\hbox to\wd0{\hss\raise\dimen0
  \hbox{\hskip.2em$\scriptscriptstyle\circ$}\hss}}#1\else {\accent"17 #1}\fi}

% Usar el archivo references.bib (lo pueden encontrar usando el botón "project" en el menú de arriba).
% Es más fácil copiar y pegar las referencias directamente de  http://www.ams.org/mathscinet en el formato BibTex:
%   - Ir a la página del artículo/libro en mathscinet
%   - "Seleccionar formato alternativo" --> BibTex
%\VA{ He a\~nadido el trabajo de Cruz-Uribe y Wang para poder citarlo en el trabajo.}
\end{document}